\documentclass[11pt,a4paper]{amsart}

\usepackage{epsfig}
\usepackage{graphicx}
\usepackage[usenames,dvipsnames]{xcolor}
\usepackage[colorlinks=true,linkcolor=Blue,citecolor=Green]{hyperref}
\usepackage{amsmath,amsthm,amsfonts,amssymb}
\usepackage[T1]{fontenc}
\usepackage[utf8]{inputenc}
\usepackage[capitalise,nameinlink,noabbrev]{cleveref}   % for references use \cref{..} in order to automatically fill in names of theorem, proposition, etc.
\crefname{enumi}{}{}
\usepackage[alwaysadjust]{paralist}
\usepackage{tikz}
\usetikzlibrary{shapes.geometric}
\usetikzlibrary{arrows}

\usepackage[skip=5pt]{caption}
\usepackage{chngpage}

\newtheorem{theorem}{Theorem}[section]
\newtheorem{lemma}[theorem]{Lemma}
\newtheorem{proposition}[theorem]{Proposition}

\newtheorem{corollary}[theorem]{Corollary}
\newtheorem{conjecture}[theorem]{Conjecture}

\newtheorem{introconjecture}{Conjecture}

\theoremstyle{definition}
\newtheorem{definition}[theorem]{Definition}

\newtheorem{example}[theorem]{Example}
\newtheorem{remark}[theorem]{Remark}

\DeclareMathOperator{\inter}{int}
\DeclareMathOperator{\relint}{relint}
\DeclareMathOperator{\conv}{conv}

\def\F{\mathcal{F}}

\def\R{\mathbb{R}}

\def\Z{\mathbb{Z}}

\def\N{\mathbb{N}}

\def\A{\mathrm{A}}

\DeclareMathOperator{\vol}{vol}
\DeclareMathOperator{\Vol}{Vol}
\DeclareMathOperator{\Surf}{Surf}

\DeclareMathOperator{\surf}{surf}
\DeclareMathOperator{\aff}{aff}
\DeclareMathOperator{\lin}{lin}
\DeclareMathOperator{\GL}{GL}

\DeclareMathOperator*{\argmax}{argmax}
\newcommand{\id}{\text{Id}}

\newcommand{\one}{\mathbf{1}}
\newcommand{\zero}{\mathbf{0}}

\renewcommand*{\a}{\omega}

\usepackage[colorinlistoftodos, textwidth=2.5cm,textsize=small]{todonotes}

\begin{document}

\title[Covering radius and discrete surface area for simplices]{The covering radius and a discrete surface area for non-hollow simplices}

% A discrete surface area bound for the cov rad of non-hollow simplices
% The cov radius of non-hollow simplices and a discrete surface area conjecture

\author{Giulia Codenotti}
\address{Institut f\"ur Mathematik, Freie Universit\"at Berlin, Arnimallee 2, 14195 Berlin, Germany}
%\email{giulia.codenotti@fu-berlin.de}
\email{codenotti@math.uni-frankfurt.de}

\author{Francisco Santos}
\address{Departamento de Matem\'{a}ticas, Estad\'{i}stica y Computaci\'{o}n, Universidad de Cantabria, Av.~de Los Castros 48, 39005 Santander, Spain}
\email{francisco.santos@unican.es}

\author{Matthias Schymura}
\address{Institut f\"ur Mathematik, BTU Cottbus-Senftenberg, Platz der Deutschen Einheit 1, 03046 Cottbus, Germany}
\email{schymura@b-tu.de}

\thanks{G.~Codenotti and F.~Santos were supported by the Einstein Foundation Berlin under grant EVF-2015-230.
F.~Santos  is also supported by grants MTM2017-83750-P/AEI/10.13039/501100011033 and PID2019-106188GB-I00/AEI/10.13039/501100011033 of the Spanish State Research Agency.
M.~Schymura was supported by the Swiss National Science Foundation (SNSF) within the project \emph{Convexity, geometry of numbers, and the complexity of integer programming (Nr.~163071)}.%
}

\date{August 24, 2021}

\begin{abstract}
We explore upper bounds on the covering radius of non-hollow lattice polytopes. In particular, we conjecture a general  upper bound of $d/2$ in dimension $d$, achieved by the ``standard terminal simplices'' and direct sums of them. We prove this conjecture up to dimension three and show it to be equivalent to the conjecture of Gonz\'alez-Merino \& Schymura (2017) that the $d$-th covering minimum of the standard terminal $n$-simplex equals $d/2$, for every $n \geq d$.

We also show that these two conjectures would follow from a discrete analog for lattice simplices of Hadwiger's formula bounding the covering radius of a convex body in terms of the ratio of surface area versus volume.
To this end, we introduce a new notion of discrete surface area of non-hollow simplices.
We prove our discrete analog in dimension two and we give strong evidence for its validity in arbitrary dimension.
\end{abstract}

\maketitle

\setcounter{tocdepth}{2}
%\tableofcontents

\section{Introduction}

% basic notation

%A convex body $K \subseteq \R^d$ is a full-dimensional compact convex set, a lattice~$\Lambda$ is a discrete subgroup of~$\R^d$, or equivalently, a linear image of the standard lattice $\Z^d$.
The \emph{covering radius} of a convex body $K$ in $\R^d$ with respect to a lattice $\Lambda$ is defined as
\[
\mu(K,\Lambda) = \min\left\{\mu \geq 0 : \mu K + \Lambda = \R^d\right\}.
\]
For us, a \emph{lattice} is always a full-dimensional linear image of~$\Z^d$.
Unless stated otherwise, we consider $\Lambda=\Z^d$ and just write~$\mu(K)$.
A convex body~$K$ is called \emph{hollow} or \emph{lattice-free} (with respect to~$\Lambda$) if $\inter(K) \cap \Lambda = \emptyset$, where $\inter(K)$ denotes the interior of~$K$.
With this notion, the covering radius $\mu(K,\Lambda)$ can be equivalently described as the greatest $\mu \geq 0$ such that the dilation $\mu K$ admits a hollow translate.
%The \emph{(normalized) volume} $\Vol_\Lambda(K)$ of a convex body $K$ with respect to a lattice~$\Lambda$ is the Euclidean volume of  $K$ normalized such that a unimodular simplex of~$\Lambda$ has volume one. Here and in the rest of the paper a unimodular simplex is one of the form $\conv(\{\zero,b_1,\ldots,b_d\})$ where $\{b_1,\ldots,b_d\}$ is a lattice basis for $\Lambda$, or a lattice translate of that.

% what we study and why

The covering radius is a classical parameter in the Geometry of Numbers, in particular in the realm of transference results, the reduction of quadratic forms, and Diophantine Approximations (cf.~\cite{gruberlekkerkerker1987geometry} for background).
In the context of the so-called \emph{flatness theorem} it also proved crucial in Lenstra's landmark paper~\cite{lenstra1983integer} on solving Linear Integer Programming in fixed dimension in polynomial time (see Kannan \& Lov\'{a}sz~\cite{kannanlovasz1988covering} for more on the flatness theorem).
More recent applications of the covering radius include (a) the classification of lattice polytopes in small dimensions (see Iglesias-Vali\~{n}o \& Santos~\cite{iglesiassantos2019classification} and the references therein), (b) distances between optimal solutions of mixed-integer programs and their linear relaxations (Paat, Weismantel \& Weltge~\cite{paatweismantelweltge2020distances}), (c) unique-lifting properties of maximal lattice-free polyhedra (Averkov \& Basu~\cite{averkovbasu2015lifting}), and (d) another viewpoint on the famous Lonely Runner Problem (Henze \& Malikiosis~\cite{henzemalikiosis}).

The covering radius is clearly invariant under translations of the body~$K$, and for every invertible matrix $A \in \R^{d \times d}$, we have $\mu(AK,A\Lambda) = \mu(K,\Lambda)$.
Hence, the covering radius is invariant under \emph{unimodular transformations}, which are affine maps $x \mapsto Ux + z$, where $z \in \Z^d$ and $U \in \GL_d(\Z)$ is a \emph{unimodular matrix}.
The behavior with respect to inclusions is as follows: For convex bodies $K \subseteq K'$ and lattices $\Lambda' \subseteq \Lambda$, 
%\paco{reversed inclusions in $\Lambda$}
we have $\mu(K',\Lambda) \leq \mu(K,\Lambda) \leq \mu(K,\Lambda')$.

We are interested in upper bounds on the covering radius of non-hollow \emph{lattice polytopes}, that is, polytopes all of whose vertices are lattice points.
If we drop the non-hollow condition, the maximum covering radius of a lattice $d$-polytope equals~$d$.
This follows since every lattice $d$-polytope contains a lattice $d$-simplex and for lattice simplices the bound is readily obtained (cf.~\cite[Eq.~(19)]{gonzalezschymura2017ondensities}).
Moreover, equality holds if and only if the lattice polytope is a \emph{unimodular simplex}; that is, one of the form $\conv(\{\zero,b_1,\ldots,b_d\})$, where $\{b_1,\ldots,b_d\}$ is a lattice basis for $\Lambda$, or a lattice translate of that (see \cref{coro:corner} for a proof of a more general statement).
 
 The existence of interior lattice points makes the problem more difficult and interesting.
The natural candidate to play the role of the unimodular simplex is  
\[
S(\one_{d+1}) := \conv(\left\{-\one_d,e_1,\ldots,e_d \right\}),
\]
since it is the unique non-hollow lattice $d$-polytope of minimum volume (see~\cite[Thm.~1.2]{beyhenkwills2007notes}). Here $\one_d = (1,\ldots,1)$ denotes the all-one vector in dimension~$d$, and $e_i$ denotes the $i$th coordinate unit vector.%
\footnote{The notation $S(\one_{d+1})$ comes from the fact that this is a particular case of the simplices $S(\a)$, $\a\in \R_{>0}^{d+1}$ introduced below.
We call $S(\one_{d+1})$  the \emph{standard terminal simplex} since
\emph{terminal} is used in the literature for lattice simplices with the origin in the interior and no lattice points other than the origin and the vertices.}

The covering radius of $S(\one_{d+1})$ was computed in~\cite[Prop.~4.9]{gonzalezschymura2017ondensities}:
\begin{align}
\mu(S(\one_{d+1}),\Z^d) = \frac{d}{2}.\label{eqn:covradSone}
\end{align}
Since the covering radius is additive with respect to direct sums (see \Cref{sec:direct_sum}),
direct sums of simplices of the form $S(\one_l)$ or lattice translates thereof also have covering radius equal to $d/2$. 
We conjecture that this procedure gives \emph{all} the non-hollow lattice polytopes of maximum covering radius in a given dimension:

\begin{introconjecture}
\label{conj:one}
Let $P \subseteq \R^d$ be a non-hollow lattice $d$-polytope. Then
\[
\mu(P) \le \frac{d}{2},
\]
with equality if and only if $P$ is obtained by direct sums and/or translations of simplices of the form $S(\one_l)$.
\end{introconjecture}

\begin{example}
In dimension two, $S(\one_3)$ has covering radius $1$, and so do the following  triangle and square:
\begin{align*}
S(\one_{2})\oplus ((1+ S(\one_{2})) &= \conv(\{(1,0), (-1,0), (0,2)\}),
\\
S(\one_{2})\oplus S(\one_{2}) &= \conv(\{(1,0), (-1.0), (0,1), (0, -1)\}).
\end{align*}
In dimension three, translations and/or direct sums of the $S(\one_l)$s produce nine pairwise non-equivalent non-hollow lattice $3$-polytopes of covering radius $3/2$, that we describe in \Cref{lemma:minimum_3d}.
\end{example}

One motivation for \cref{conj:one} is as follows.
The \emph{$d$-th covering minimum} of a convex body $K \subseteq \R^n$ with respect to a lattice $\Lambda \subseteq \R^n$ is defined as 
%Our second result (see \cref{thm:one} in \cref{sec:one}) is that
%\cref{conj:one} is equivalent to \cref{conj:minima} below, that was raised in~\cite[Rem.~4.10]{gonzalezschymura2017ondensities}.
%Let 
\[
%\mu_d(K) := \min_{A\in \Z^{d\times n}} \mu(A(K)).
%\quad
\mu_d(K, \Lambda) := \max_{\pi} \mu(\pi(K),\pi(\Lambda)),
\]
where $\pi$ runs over all linear projections $\pi: \R^n \to \R^d$ such that $\pi(\Lambda)$ is a lattice.
Covering minima were introduced by Kannan \& Lov\'{a}sz~\cite{kannanlovasz1988covering} and
interpolate between $\mu_n(K) = \mu(K)$ and $\mu_1(K)$, the reciprocal of the  \emph{lattice width} of $K$.

Since $S(\one_{n+1})$ projects to $S(\one_{d+1})$ for every $d < n$, we use~\eqref{eqn:covradSone} and get
\begin{align}
\label{eq:mud_one}
\mu_d(S(\one_{n+1})) \ge
\mu_d(S(\one_{d+1})) = \frac{d}2.
\end{align}
The converse inequality was conjectured in~\cite{gonzalezschymura2017ondensities}:

\begin{introconjecture}[\protect{\cite[Rem.~4.10]{gonzalezschymura2017ondensities}}]
\label{conj:minima}
For every $n \in \N$ and $d \le n$,
\begin{align}
\label{eq:minima}
\mu_d(S(\one_{n+1}))= \frac{d}2.
\end{align}
\end{introconjecture}
\noindent In \cref{sec:one-minima} we prove:
% \cref{conj:one} (perhaps without the equality case) is equivalent to \cref{conj:minima}, and that they both hold for $d\le 3$:

\begin{theorem}[Equivalence of Conjectures \ref{conj:one} and \ref{conj:minima}, \cref{sec:one}]
\label{thm:one}
For each $d\in \N$, the following are equivalent:
\begin{enumerate}[i)]
\item $\mu(P) \leq \frac{\ell}{2}$ for every non-hollow lattice $\ell$-polytope $P$ and for every $\ell\le d$.
\item \cref{conj:minima} holds for every $\ell\le d$. That is, 
$\mu_\ell(S(\one_{n+1}))=\frac{\ell}{2}$, for every $\ell, n \in \N$ with $\ell\le d\le n$.
\end{enumerate}
\end{theorem}

\begin{theorem}[\Cref{coro:ConjA_dim2} and \Cref{thm:ConjA_dim3}]
\label{thm:one_smalldim}
\cref{conj:one}, hence also \cref{conj:minima}, holds in dimension up to three.
\end{theorem}

The computation of the covering radius for $S(\one_{d+1})$ can be generalized to the following class of simplices:
For each $\a=(\a_0,\dots,\a_d) \in \R_{> 0}^{d+1}$, we define
\[
S(\a) := \conv(\left\{ -\a_0 \one_d, \a_1 e_1, \ldots, \a_d e_d\right\}).
\]
%Understanding the covering radius of this family turns out to be crucial in many of our proofs.
In \cref{sec:lambda} we derive the following closed formula for $\mu(S(\a))$.
Therein and in the rest of the paper we denote by $\Vol_\Lambda(K)$
the \emph{normalized volume}  of a convex body~$K$ with respect to a lattice~$\Lambda$, which equals the Euclidean volume $\vol(K)$ of~$K$ normalized such that a unimodular simplex of~$\Lambda$ has volume one.

\begin{theorem}[\Cref{sec:lambda_proof}]
\label{thm:lambda}
For every $\a \in \R^{d+1}_{>0}$, we have
\[
\mu(S(\a)) 
= 
\frac{\sum_{0 \leq i < j \leq d} \frac{1}{\a_i \a_j}} {\sum_{i=0}^d \frac{1}{\a_i}}
=
\frac12 \frac{\sum_{i=0}^d \Vol_{\pi_i(\Z^d)}(\pi_i(S(\a)))} {\Vol_{\Z^d}(S(\a))},
\]
where $\pi_i: \R^d \to \R^{d-1}$ is the linear orthogonal projection 
along the line through the origin and the $i$th vertex of~$S(\a)$.
%vanishing at the $i$th vertex of~$S(\a)$.
\end{theorem}

In \cite{gonzalezschymura2017ondensities}, the authors conjecture an optimal lower bound on the \emph{covering product} $\mu_1(K) \cdot \ldots \cdot \mu_d(K) \cdot \Vol_{\Z^d}(K)$ for any convex body~$K \subseteq \R^d$.
As a consequence of the explicit formula for $\mu(S(\a))$, we confirm this conjecture for the simplices~$S(\a)$ (see \cref{cor:lambda_cov_prod}).

Observe that the volume expression on the right in \cref{thm:lambda} can be defined for every simplex with the origin in its interior as follows:

\begin{definition}
\label{def:discretesurfarea}
Let $S = \conv(\{v_0,\ldots,v_d\})$ be a $d$-simplex with the origin in its interior.
We say that $S$ has \emph{rational vertex directions} if the line through the origin and the vertex $v_i$ has rational direction, for every $0 \leq i \leq d$.

Writing $\pi_i: \R^d \to \R^{d-1}$ for a linear projection vanishing at $v_i$, we define the \emph{discrete surface area} of such a simplex~$S$ as
\[
\Surf_{\Z^d}(S) := \sum_{i=0}^d \Vol_{\pi_i(\Z^d)}(\pi_i(S)).
\]
\end{definition}

Note that $\Vol_{\pi_i(\Z^d)}(\pi_i(S)) = \Vol_{\pi_i(\Z^d)}(\pi_i(F_i))$, with~$F_i$ being the facet of~$S$ opposite to the vertex $v_i$.
In this sense, the sum of these numbers is indeed a version of the ``surface area'' of~$S$, except that the volume of each facet is computed with respect to the lattice projected from the opposite vertex.

Motivated by this definition and \cref{thm:lambda} we propose the following conjecture, which is the main object of study in this paper:

\begin{introconjecture}
\label{conj:volume}
Let $S$ be a $d$-simplex with the origin in its interior and with rational vertex directions.
Then
\begin{align}
\mu(S) \leq \frac12 \frac{\Surf_{\Z^d}(S)}{\Vol_{\Z^d}(S)}.
\label{eq:volume}
\end{align}
\end{introconjecture}

We formulate this conjecture only for simplices~$S$ rather than for arbitrary polytopes that contain the origin and have rational vertex directions, because without further study it is not clear how the discrete surface area $\Surf_{\Z^d}(S)$ can be extended in a meaningful way.
For example, we could project along the vertex directions as in the simplex case, but then the correspondence with the opposite facet is lost.

In \Cref{sec:volume} we give additional motivation for \Cref{conj:volume}. We show that it implies \cref{conj:one} (\cref{coro:vol_one}), that it holds in dimension two (\cref{thm:volume_dim2}), and that in arbitrary dimension it holds up to a factor of two (\cref{prop:volume_factor_2}).

Covering criteria such as the one in \Cref{conj:volume} are rare in the literature, but very useful as they reduce the question of covering to computing less complex geometric functionals such as volume or (variants of the) surface area (cf.~\cite[Sect.~31]{gruber2007convex}).
A classical inequality of this type is the following result of Hadwiger.
We regard \cref{conj:volume} as a discrete analog thereof.

\begin{theorem}[Hadwiger~\cite{hadwiger1970volumen}]
\label{thm:hadwiger}
For every convex body $K$ in $\R^d$
\[
\mu(K) \leq \frac12 \frac{\surf(K)}{\vol(K)},
\]
where $\vol(K)$ and $\surf(K)$ are the Euclidean volume and surface area of~$K$.
\end{theorem}

Observe that the statement of \cref{conj:volume} is more intrinsic than Hadwiger's inequality. This is because the Euclidean surface area is not invariant under unimodular transformations, so that the bound in \cref{thm:hadwiger} depends on the particular representative of $K$ in its unimodular class.
Moreover the inequality only holds for the standard lattice $\Z^d$ and cannot easily be transfered to other lattices (cf.~\cite{schnell1992minimal} for partial results for arbitrary lattices).
In constrast, our proposed relation in \cref{conj:volume} \emph{is} unimodularly invariant and there is no loss of generality in restricting to the standard lattice as we do (see \cref{cor:discrsurfarea_invariance} for details on these claims).
Moreover, our proposed inequality in \cref{conj:volume} is tight for the large class of simplices $S(\a)$.

In \cref{sec:0inboundary}, we complement our investigations on \cref{conj:volume} by extending it to the case where the origin lies in the boundary of the simplex~$S$, rather than in the interior. 
%We state a \cref{conj:volume2} for this case which is seemingly stronger than \cref{conj:volume} but turns out to be equivalent (\cref{prop:volume2}).

Another way to extend \cref{conj:one} is to ask for the maximal covering radius among lattice polytopes with at least $k \geq 1$ interior lattice points.
The natural conjecture is:

\begin{introconjecture}
\label{conj:k}
Let $k,d\in \N$ be nonnegative integers. Then, for every lattice $d$-polytope $P$ with $k$ interior lattice points we have
\[
\mu(P) \le \frac{d-1}2 + \frac1{k+1}.
\]
Equality holds for $k=1$ if and only if $P$ is obtained by direct sums and/or translations of simplices of the form $S(\one_l)$, and for $k \geq 2$, if and only if $P$ is obtained by direct sums and/or translations of the segment $[0,k+1]$ and simplices $S(\one_l)$.
\end{introconjecture}

In \cref{sec:k_points} we  prove this conjecture in dimension two (see \cref{thm:dim2-k}).
%
%\begin{theorem}
%\label{thm:dim2-k-intro}
%Let $P$ be a non-hollow lattice polygon.
%% and assume that $P$ contains at least 4 collinear lattice points.
%Then $\mu(P)\le \frac12+\frac1{k+1}$, with  equality if and only if $P$ is equivalent to
%$[-1,1] \oplus [-i,k+1-i]$ for some $i\in \{0,\dots, k+1\}$.
%\end{theorem}
%
%
Observe that no analog of \cref{conj:k} makes sense for other covering minima.
Indeed, the maximum $d$th covering minimum $\mu_d$ among non-hollow lattice $n$-polytopes with $k$ interior lattice points does not depend on $k$ or $n$, for $d < n$:
It equals the maximum covering radius among non-hollow lattice $d$-polytopes, since every non-hollow lattice $d$-polytope can be obtained as the projection of a $(d+1)$-polytope with arbitrarily many interior lattice points.
In fact, assuming \cref{conj:one} this maximum is given by
\[
\mu_d(S(k,1,\ldots,1)) = \mu_d(S(\one_{d+1})) = \frac{d}2, \quad \textrm{for all } n>d \textrm{ and } k \in \N.
\]
Summing up, the relationship between our  conjectures is as follows:
%
%\begin{theorem}
%\label{thm:conjectures}
\[
\begin{matrix}
%\cref{conj:volume2}
&&\cref{conj:volume}\\
%\Updownarrow
&&\Downarrow\\
\cref{conj:k} & \Rightarrow & \cref{conj:one} & \\
&&\Downarrow\\
&&\cref{conj:one}&\text{without equality case}\\
&&\Updownarrow\\
&&\cref{conj:minima}\\
\end{matrix}
\]

A summary of our results is that all these conjectures hold in dimension two, that \cref{conj:one} holds in dimension three and that \cref{conj:volume} holds for the simplices of the form $S(\a)$.

%\Paco{things not yet mentioned in the intro:
%
%- the preliminaries (direct sums and tightness)
%
%- "failed attempt" (BTW, the 2-dim lemma in this section is used in \cref{sec:26bounds})
%
%- the case of $k$ interior points
%
%- the case of $\zero$ in the boundary
%
%- the covering product conjecture
%
%We do not need to mention (all of) them, but we should give them a thoguht
%}

\section{Preliminaries}
\label{sec:tools}

This section develops some tools that will be essential for our analyses.
We first describe how the covering radius behaves with respect to projections, and more importantly, that it is an additive functional on direct sums of convex bodies and lattices.
Afterwards we introduce and study the concept of \emph{tight covering} that facilitates our characterizations of equality, for example the one in \cref{thm:one_smalldim}.

\subsection{Projection and direct sum}
\label{sec:direct_sum}

\begin{lemma}
\label{lem:projection}
Let $K \subseteq \R^d$ be a convex body containing the origin, and let $\pi: \R^d \to \R^l$ be a rational linear projection, so that $\pi(\Z^d)$ is a lattice. Let $Q=K \cap \pi^{-1}(\zero)$  and let $L= \pi^{-1}(\zero)$ be the linear subspace spanned by $Q$.
%
%$Q \subseteq P$ be an $\ell$-dimensional convex body in~$K$ such that its affine hull is parallel to a lattice %plane.
%Furthermore, let $L_Q = \lin\{Q-q\}$, for some $q \in Q$ and let $\pi:\R^d \to L_Q^\perp$ be the orthogonal projection "forgetting"~$Q$.
Then, we have
\[
\mu(K,\Z^d) \leq \mu(Q,\Z^d \cap L) + \mu(\pi(K),\pi(\Z^d)).
\]
\end{lemma}

\begin{proof}
Let us abbreviate $\mu_{Q} = \mu(Q,\Z^d \cap L)$ and $\mu_{\pi} = \mu(\pi(K),\pi(\Z^d))$.
Let $x \in \R^d$ be arbitrary.
Then, $\pi(x)$ is covered by $\mu_{\pi} \cdot \pi(K) + \pi(\Z^d) = \pi \left(\mu_{\pi} K + \Z^d \right)$.
Hence, there exists a point $x' \in \R^d$ such that the segment $[x,x']$ is parallel to $L$ and such that $x'$ is covered by $\mu_{\pi} K + \Z^d$.
On the other hand, $y = x - x' \in L$ is covered by $\mu_Q Q + (\Z^d \cap L)$.
Since $Q \subseteq K$, this implies that $x = y + x'$ is covered by $(\mu_Q + \mu_{\pi}) K + \Z^d$, as claimed.
\end{proof}

A particularly interesting case of the above result is when~$K$ decomposes as a direct sum.
Let $\R^d = V \oplus W$ be a decomposition into complementary linear subspaces with $\dim(V) = \ell$ and $\dim(W) = d-\ell$.
The \emph{direct sum} of two convex bodies $K \subseteq V, L \subseteq W$ both containing the origin is defined as
\[
K \oplus L := \left\{ \lambda x +  (1 - \lambda) y : 
x \in K, y \in L, \lambda \in [0,1] \right\} \subseteq \R^d.
\]
The direct sum of two lattices $\Lambda \subseteq V$, $\Gamma \subseteq W$ is defined as
\[
\Lambda \oplus \Gamma := \left\{ x + y : x \in \Lambda, y \in \Gamma \right\} \subseteq \R^d.
\]
With these definitions we can now formulate:

\begin{corollary}
\label{cor:directsum}
Let $\R^d = V \oplus W$ be a decomposition as above, let $K \subseteq V$, $L \subseteq W$ be convex bodies containing the origin, and let $\Lambda \subseteq V$, $\Gamma \subseteq W$ be lattices.
Then,
\[
\mu_d(K \oplus L, \Lambda \oplus \Gamma) = \mu_\ell(K, \Lambda) + \mu_{d-\ell}(L, \Gamma).
\]
\end{corollary}

\begin{proof}
The inequality $\mu_d(K \oplus L, \Lambda \oplus \Gamma) \leq \mu_\ell(K, \Lambda) + \mu_{d-\ell}(L, \Gamma)$ is a special case of \Cref{lem:projection}, via the natural projection $\R^d = V \oplus W \to V$.

For the other inequality, let  $x \in V$ be a point not covered by $c K + \Lambda$ for some $c < \mu_\ell(K,\Lambda)$ and let $y \in W$ be a point not covered by $\bar c L + \Gamma$ for some $\bar c < \mu_{d-\ell}(L,\Gamma)$.
We claim that $x + y \in V \oplus W = \R^d$ is not covered by $(c + \bar c) (K \oplus L) + \Lambda \oplus \Gamma$, and thus $c + \bar c \leq \mu_d(K \oplus L, \Lambda \oplus \Gamma)$.
Since, $c$ and $\bar c$ were taken arbitrarily, this implies $\mu_\ell(K, \Lambda) + \mu_{d-\ell}(L, \Gamma) \leq \mu_d(K \oplus L, \Lambda \oplus \Gamma)$.

Assume to the contrary, that $x + y \in (c + \bar c)(K \oplus L) + \Lambda \oplus \Gamma$, that is, $x + y = (c + \bar c)(\lambda p + (1 - \lambda) q) + w + z$, for some $\lambda \in [0,1]$, $p \in K$, $q \in L$, $w \in \Lambda$, and $z \in \Gamma$.
Since the sums are direct, we get $x = (c + \bar c) \lambda p + w$ and $y = (c + \bar c)(1 - \lambda) q + z$, which by assumption implies $(c + \bar c) \lambda > c$ and $(c + \bar c) (1 - \lambda) > \bar c$.
These two inequalities cannot hold at the same time, and we arrive at a contradiction.
\end{proof}

\subsection{Tight covering}

\begin{definition}
\label{defi:tight}
Let $K \subseteq \R^d$ be a convex body and let $\Lambda$ be a lattice.
Then,~$K$ is called \emph{tight for $\Lambda$} if for every convex body $K' \supsetneq K$, we have 
\[
\mu(K',\Lambda) < \mu(K,\Lambda).
\]
\end{definition}

\begin{definition}
\label{defi:last_covered}
Let $K \subseteq \R^d$ be a convex body of covering radius $\mu$ with respect to a lattice  $\Lambda$. A point $p \in \R^d$ is \emph{last covered by $K$} if
\[
p \notin \inter(\mu\cdot K) + \Lambda.
\]
Let $P$ be a $d$-polytope, let $F$ be a facet of $P$, and let~$p$ be a point that is last covered by~$P$.
We say that $p$ \emph{needs $F$} if $p \in \relint(\mu\cdot F) + \Lambda$.
\end{definition}

\begin{lemma}
\label{lemma:tight_vs_facets}
Let $K \subseteq \R^d$ be a convex body of covering radius $\mu$ with respect to a lattice  $\Lambda$. Then, the following properties are equivalent:
\begin{enumerate}[i)]
\item $K$ is tight for $\Lambda$.
\item $K$ is a polytope and for every facet $F$ of~$K$ and for every last covered point $p$, $p$ needs $F$.
\item $K$ is a polytope and every facet of every hollow translate of $\mu\cdot K$ is non-hollow.
\item Every hollow translate of $\mu\cdot K$ is a maximal hollow convex body with respect to inclusion.
\end{enumerate}
\end{lemma}

\begin{proof}
The equivalence of iii) and iv) is the characterization of maximal hollow convex bodies by Lov\'asz~\cite{lovasz1989geometryofnumbers}. 
For the equivalence of i) and iv) observe that, by definition, $\mu$ is the largest constant such that (a) $\mu \cdot K$ has a hollow lattice translate and (b) the inequality $\mu(K',\Lambda) < \mu(K,\Lambda)$ in the definition of tightness is nothing but maximality of all such hollow translates.

%\paco{with the new statement it might be more natural to show (ii)$\Leftrightarrow$(iii) instead.}
We now show the equivalence of i) and ii).
Suppose there is a facet $F$ of~$K$ that is not needed by some last covered point $p$. Let $K'=\conv(K \cup \{x\})$, where $x \notin K$ is a point \emph{beyond $F$}, meaning that~$x$ violates the inequality that defines~$F$, but satisfies all other facet-inducing inequalities of~$K$.
Then
\[
\mu(K',\Lambda) = \mu(K,\Lambda),
\]
because $p$ is still a last covered point of $K'$ (for the same dilate $\mu$).

Conversely, if $K$ is not tight let $K'$ be a convex body strictly containing~$K$ and that has the same covering radius. Let $F$ be a facet of $K$ with $\relint(F) \subseteq \inter(K')$.
Let~$p$ be a point that is last covered by $K'$.
Since the covering radii are equal and $K \subsetneq K'$, $p$ must also be last covered by~$K$.
Since we chose~$F$ so that $\relint(F)$ is in the interior of $K'$, $p$ does not need~$F$.
\end{proof}

\begin{example}
\label{ex:tightness_char}
It is not sufficient for tightness that ``every facet is needed by \emph{some} last covered point.'' An example showing this is the hexagon $P=\conv(\{\pm(1,0), \pm(0,1), \pm(1,1)\})$ with respect to the integer lattice. $P$ has covering radius $2/3$, the same as the triangle $\conv(\{(-1,1), (2,1), (-1,-2)\})$ that properly contains it, so it is not tight.
It has two orbits of last covered points, with representatives $\pm\left(2/3,1/3\right)$, each of which needs three of the six edges of $P$.
\end{example}

\begin{lemma}
\label{lemma:simplex_monotone}
Every simplex is tight for every lattice.
\end{lemma}

\begin{proof}
We use \Cref{lemma:tight_vs_facets}. 
Let $\Delta$ be a simplex of covering radius $\mu$ with respect to a lattice $\Lambda$, and let~$p$ be a point last covered by $\Delta$.
That is, $p \notin \inter(\mu\Delta) + \Lambda$.
Let $F_0,F_1,\ldots,F_d$ be the facets of $\Delta$,
%We claim that for every $p \in \R^d$ that is last covered by $\Delta$, there are lattice points $z_0,z_1,\ldots,z_d \in \Lambda$ such that $p \in \relint(\mu F_i + z_i)$, for $0 \leq i \leq d$.
%
with interior facet normals $v_0,\dots,v_d$.

Every neighborhood of $p$ is covered by $\mu \Delta + \Lambda$, and $p$ can only lie in lattice translates of the boundary of $\mu \Delta$. Suppose, in order to get a contradiction, that a certain facet $F_i$ is not needed by $p$. This implies that for every 
$\mu \Delta + z$ ($z\in \Lambda$) containing $p$ there is a facet $F_j\ne F_i$ such that $
\mu \Delta + z \subset H_j^p$, where 
\[
H_j^p:=\{x \in \R^d : v_j^\intercal x \le v_j^\intercal p\}
\]
is the translation to $p$ of the $j$-th facet-defining half-space of $\Delta$. This implies that we have a neighborhood of $p$ covered by the $d$ affine half-spaces with~$p$ in the boundary corresponding to the indices $j\ne i$.
This is impossible since the corresponding $d$ normals are linearly independent.
%
%
%Since no distinct facets of $\Delta$ are translates of one another, there are $d+1$ affinely independent normals to $\Delta$ such that translates of the corresponding facets contain $p$ in their relative interior.
%As~$\Delta$ has exactly $d+1$ facets, the assertion is established.
%\matthias{Should I give more details here?}
%\paco{there is a slight ambiguity. A priori $p$ could be in lower dim faces, not relative interior of facets}
%
%Now, since $\Delta$ is strictly contained in~$K$, there is a facet $F_j$ of $\Delta$ such that $\relint(F_j) \subseteq \inter(K)$.
%Therefore, the above observation implies that every $p \in \R^d$ that is last covered by $\Delta$ lies in the interior of a lattice translate of~$\mu K$.
%Thus, $\mu(K,\Lambda) < \mu$, as claimed.
\end{proof}

\begin{lemma}
\label{lemma:sum_monotone}
Let $K_1$ and $K_2$ be convex bodies containing the origin and let~$\Lambda_1$ and $\Lambda_2$ be lattices.
Then, $K_1$ and $K_2$ are tight for $\Lambda_1$ and $\Lambda_2$, respectively, if and only if $K_1 \oplus K_2$ is tight for $\Lambda_1 \oplus \Lambda_2$.
\end{lemma}

\begin{proof}
First of all, let $K' \supsetneq K_1 \oplus K_2$ be a convex body and let $K'_1$ and $K'_2$ be the projection of $K'$ onto the linear span of $K_1$ and $K_2$, respectively.
Clearly, either $K'_1 \supsetneq K_1$ or $K'_2 \supsetneq K_2$, so that by \Cref{cor:directsum} and the tightness of~$K_1$ and $K_2$, we have
\begin{align*}
\mu(K_1 \oplus K_2, \Lambda_1 \oplus \Lambda_2) &= \mu(K_1,\Lambda_1) + \mu(K_2,\Lambda_2) > \mu(K'_1,\Lambda_1) + \mu(K'_2,\Lambda_2) \\
&= \mu(K'_1 \oplus K'_2, \Lambda_1 \oplus \Lambda_2) \geq \mu(K', \Lambda_1 \oplus \Lambda_2),
\end{align*}
since $K'_1 \oplus K'_2 \subseteq K'$.
Therefore, $K_1 \oplus K_2$ is tight for $\Lambda_1 \oplus \Lambda_2$.

Conversely, if say $K_1$ is not tight for $\Lambda_1$, then there exists $K'_1 \supsetneq K_1$ such that $\mu(K_1,\Lambda_1) = \mu(K'_1,\Lambda_1)$.
Then, $K'_1 \oplus K_2 \supsetneq K_1 \oplus K_2$ and by \Cref{cor:directsum}
\begin{align*}
\mu(K'_1 \oplus K_2,\Lambda_1 \oplus \Lambda_2) &= \mu(K'_1,\Lambda_1) + \mu(K_2,\Lambda_2) = \mu(K_1,\Lambda_1) + \mu(K_2,\Lambda_2) \\
&= \mu(K_1 \oplus K_2,\Lambda_1 \oplus \Lambda_2),
\end{align*}
so $K_1 \oplus K_2$ is not tight for $\Lambda_1 \oplus \Lambda_2$.
\end{proof}

\begin{lemma}
\label{lemma:lattice_monotone}
Let $\Lambda' \subsetneq \Lambda$ be two lattices in $\R^d$, and let $K \subseteq \R^d$ be a convex body. Then,
\[
\mu(K, \Lambda) \le \mu(K, \Lambda').
\]
\end{lemma}

\begin{proof}
Let $\mu = \mu(K, \Lambda)$ and $\mu' = \mu(K, \Lambda')$.
Then, $\mu' K + \Lambda' \subseteq \mu' K + \Lambda$, so $\mu \leq \mu'$.
\end{proof}

\begin{remark}\ 
\begin{enumerate}[i)]
 \item An example where equality holds in \Cref{lemma:lattice_monotone} is the following: Let $K= [-1,1]^d$ and let $\Lambda$ be an arbitrary refinement of $\Z^d$  contained in $\R^{d-1} \times \Z$. Then, $\mu(K,\Z^d) = \mu(K,\Lambda)=1/2$.
 \item The inequality in \Cref{lemma:lattice_monotone} may not be strict, even for simplices. An example is the simplex $(I\oplus I')' \oplus I$ of \Cref{lemma:minimum_3d} below. It has the same covering radius as $S(\one_4)$ (equal to $3/2$), yet it is isomorphic to $S(\one_4)$ when regarded with respect to the  sublattice of index two generated by its vertices and its interior lattice point. This can easily be derived from its depiction in the bottom-center of \Cref{fig:sums_of_I}, or from its coordinates in \Cref{tbl:minimal_tetrahedra} (in these coordinates the sublattice is $\{(x,y,z)\in \Z^3: x\in 2\Z\}$).
\end{enumerate}
\end{remark}

\section{\texorpdfstring{\cref{conj:one}}{\ref{conj:one}} and \ref{conj:minima}: Equivalence and small dimensions}
\label{sec:one-minima}

%\subsection{\texorpdfstring{\cref{conj:one}}{Conjecture A} \texorpdfstring{$\Leftrightarrow$}{is equivalent to} \texorpdfstring{\cref{conj:minima}}{Conjecture B}}
\subsection{Equivalence of Conjectures~\ref{conj:one} and~\ref{conj:minima}}
\label{sec:one}
%
%\begin{theorem}
%\label{thm:one}
%\cref{conj:one} and \cref{conj:minima} are equivalent.
%More precisely, in a given dimension $d$, we have $\mu(P) \leq \frac{d}{2}$, for every non-hollow lattice polytope $P \subseteq \R^d$, if and only if $\mu_\ell(S(\one_{n+1}))=\frac{\ell}{2}$, for every $\ell, n \in \N$ such that $\ell \leq d$ and $\ell \leq n$.
%\end{theorem}
%
As an auxiliary result we first reduce \cref{conj:one} to lattice simplices.

\begin{lemma}
\label{lemSimplexInside}
Every non-hollow lattice polytope contains a non-hollow lattice simplex of possibly smaller dimension.
\end{lemma}

\begin{proof}
Let $P$ be a non-hollow lattice polytope and $p \in \inter(P)\cap\Z^d$. Applying Carath\'eodory's Theorem to an expression of $p$ as a convex combination of the vertex set of $P$ we obtain an affinely independent subset of vertices that still has $p$ as a convex combination. The vertices involved in that convex combination form a non-hollow simplex contained in $P$.
\end{proof}

\begin{corollary}
\label{cor:ReduToSimpls}
\cref{conj:one} reduces to lattice simplices.
More precisely, \cref{conj:one} holds in every dimension~$\leq d$ if and only if it holds for lattice simplices in every dimension~$\leq d$.
\end{corollary}

\begin{proof}
One direction is trivially true.
We prove the other one by induction on~$d$.
Let $P \subseteq \R^d$ be a non-hollow lattice polytope.
In view of \cref{lemSimplexInside}, we find an $\ell$-dimensional non-hollow lattice simplex $S \subseteq P$.
If $\ell=d$, then we simply have $\mu(P) \leq \mu(S)$.
So, let us assume that $\ell < d$ and assume that \cref{conj:one} is proven for any dimension~$<d$.
Assume also that~$S$ contains the origin in its interior and write $L_S$ for the linear hull of~$S$.
We now apply \cref{lem:projection} to the projection~$\pi$ onto $L_S^\perp$.
Observe that~$S \subseteq P \cap \pi^{-1}(\zero) = P \cap L_S$, and that~$S$ is non-hollow with respect to $\Z^d \cap L_S$ and $\pi(P)$ is non-hollow with respect to the lattice $\pi(\Z^d)$.
We get that
\[
\mu(P) \leq \mu(S,\Z^d \cap L_S) + \mu(\pi(P),\pi(\Z^d)) \leq \frac{\ell}{2} + \frac{d-\ell}{2} = \frac{d}{2}.\qedhere
\]
\end{proof}

\begin{proof}[Proof of \cref{thm:one}]
Suppose first that for $\ell \le d$ every lattice $\ell$-polytope~$P$ has $\mu(P) \le \ell/2$.
Since $S(\one_{n+1})$ projects to $S(\one_{\ell+1})$, we have by~\eqref{eqn:covradSone}
\[
\mu_\ell(S(\one_{n+1}),\Z^n) \geq \mu_\ell(S(\one_{\ell+1}),\Z^\ell) = \frac{\ell}2.
\]
For the converse inequality, let $\pi : \R^n \to \R^\ell$ be an integer projection along which the value of $\mu_\ell(S(\one_{n+1}))$ is attained.
Then, $\pi(S(\one_{n+1}))$ is non-hollow with respect to the lattice $\pi(\Z^n)$, and thus
\[
\mu_\ell(S(\one_{n+1}),\Z^n) = \mu_\ell(\pi(S(\one_{n+1})), \pi(\Z^n)) \leq \frac{\ell}2.
\]
For the reverse implication (ii) $\Rightarrow$ (i), suppose \cref{conj:minima} holds in every dimension $\ell \leq d$.
Let $P$ be a lattice $\ell$-polytope with at least one interior lattice point, which without loss of generality we assume to be the origin~$\zero$.
By \cref{cor:ReduToSimpls} we can assume $P$ to be a simplex, and we let $v_0,\dots,v_\ell$ be its vertices.
Let $(b_0,\ldots,b_\ell) \in \N^{\ell+1}$ be a multiple of the barycentric coordinates of $\zero$ in $P$; that is, assume that
\begin{equation}
\label{eqn:barycentric}
\zero = \frac{1}{N} \sum_{i=0}^\ell b_i v_i,
\end{equation}
where $N = \sum_{i=0}^\ell b_i \geq \ell+1$.
Consider the $(N-1)$-dimensional simplex $S(\one_{N})$, and the affine projection $\pi : \R^{N-1} \to \R^\ell$ that sends exactly $b_i$ vertices of $S(\one_{N})$ to $v_i$, $i=0,\ldots,\ell$. 
Expression~\eqref{eqn:barycentric} implies that $\pi$ sends the origin to the origin, which in turn implies $\pi$ to be an integer projection.
In particular,
\[
\mu(P,\Z^\ell) \leq \mu_\ell(\pi(S(\one_{N})),\pi(\Z^{N-1})) \leq \mu_\ell(S(\one_{N}),\Z^{N-1}) = \frac{\ell}2,
\]
since $\pi(\Z^{N-1}) \subseteq \Z^\ell$.
\end{proof}

\subsection{\texorpdfstring{\cref{conj:one}}{Conjecture \ref{conj:one}} in dimensions \texorpdfstring{$2$}{2} and \texorpdfstring{$3$}{3}}
\label{sec:smalldim}

We here prove \cref{conj:one} in dimensions two and three, including the case of equality.

\subsubsection*{\texorpdfstring{\cref{conj:one}}{Conjecture \ref{conj:one}} in dimension two}

Let $I=[-1,1]$ and $I'=[0,2]$ be intervals of length two centered at $0$ and $1$, respectively. 

\begin{lemma}
\label{lemma:dimtwo-a}
The three polygons $S(\one_3)$, $I \oplus I$, and $I \oplus I'$ 
have covering radius equal to one.
\end{lemma}

\begin{proof}
For $S(\one_3)$ this is just \cref{eqn:covradSone}.
For the other two polygons it follows from \Cref{cor:directsum},
since they are unimodularly equivalent to direct sums of segments of length two.
\end{proof}

\begin{figure}[htb]
\centerline{\includegraphics[scale=.8]{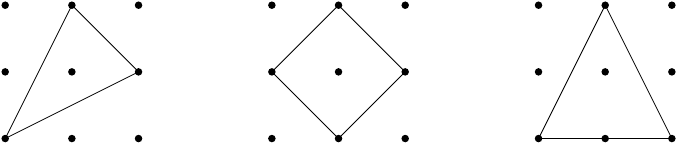}}
\caption{The non-hollow lattice polygons $S(\one_3)$, $I \oplus I$ and $I\oplus I'$ of covering radius equal to one.}
\label{fig:3cov1polys}
\end{figure}

We now show that every other non-hollow lattice polygon contains a (unimodularly equivalent) copy of one of these three, which implies \Cref{conj:one}.
For this let us consider the following auxiliary family of lattice triangles with~$k$ interior lattice points:
For each $k \in \N$, and $\alpha\in \{0,1\}$ let
\[
M_k(\alpha)=\conv(\{(-1,0), (1,\alpha), (0, k+1)\}).
\]
Observe that 
\[
M_1(0) = I\oplus I',
\quad
M_1(1) \cong S(\one_3),
%\]
\quad\text{and}\quad
%\[
\forall k\ge 2,
\ 
M_{k-1}(\alpha) \subsetneq M_{k}(\alpha).
\]

\begin{remark}
\label{rem:M_k-dim2}
The covering radius of $M_k(\alpha)$ can be computed explicitly via
\[
M_k(0) \cong I \oplus [0,k+1], \quad
\text{and}
\quad
M_k(1) \cong S(k,1,1).
\]
Indeed, this implies
\[
\mu(M_k(0)) = \frac12 + \frac1{k+1} = \frac{k+3}{2k+2} 
, \quad
\text{and}
\quad
\mu(M_k(1)) = \frac{1 + \frac2k}{2+\frac1k} = \frac{k+2}{2k+1},
\]
by \Cref{cor:directsum} and \Cref{thm:lambda}, respectively.
We see that their covering radius equals $1$ for $k=1$, and is strictly smaller for~$k \geq 2$.
\end{remark}

%\paco{previous statement was stringer, but htis version is enough since $I\oplus I$ is tight}

\noindent The following statement might be known in the literature on lattice polygons.
In absence of a clear reference we give a detailed proof for the sake of a complete presentation.
%\matthias{added these two sentences; do we want to add references for that 'literature'?}

\begin{lemma}
\label{lemma:dimtwo-b}
Every non-hollow lattice polygon~$P$ contains a unimodular copy of either $M_1(0)= I\oplus I'$,  $M_1(1)\cong S(\one_3)$ or $I\oplus I$. 
\end{lemma}

\begin{proof}
Without loss of generality, assume the origin is in the interior of $P$.
Consider the complete fan whose rays go through all non-zero lattice points in $P$. We call this the \emph{lattice fan} associated to $P$, and it is a complete unimodular fan. Since a 2-dimensional fan is uniquely determined by its rays, we denote by $\F\{v_1,\dots,v_m\}$ the fan with rays through $v_1,\dots, v_m \in \R^2$. In particular, the lattice fan of $P$ is denoted by $\F\{P\cap \Z^2\}$.

By the classification of complete unimodular fans, see~\cite[Thm.~V.6.6]{ewald1996combinatorial}, $\F\{P\cap \Z^2\}$ can be obtained (modulo unimodular equivalence) by successively refining  the lattice fan of either $S(\one_3)$ or
\[
\F_l:= \F\{(0,-1), (0,1), (-1,0), (1,l)\}, 
\]
for some $l\in \Z_{\ge0}$. Observe that $\F_0$ is the lattice fan of $I\oplus I$, $\F_1$ refines the lattice fan of $S(\one_3)\cong M_1(1)$ and, for every $l\ge 2$ we have that $\F_l$ is unimodularly equivalent to the lattice fan of
\[
\begin{cases}
M_k(0)& \text{ if $l=2k$ is even, and}\\
M_k(1)& \text{ if $l=2k-1$ is odd,}
\end{cases}
\]
independently of which interior lattice point of $M_k(\alpha)$ we consider the rays of its lattice fan emanating from.

This, together with the fact that $M_1(\alpha) \subseteq M_k(\alpha)$ for every $k\ge 1$, implies that $P$ contains one of $M_1(0)$, $M_1(1)$ or $I\oplus I$.  
%Since the three of them are tight by \Cref{lemma:sum_monotone}, we are done.
%
%
%To finish the proof we need to check that if $P$ \emph{properly} contains $I \oplus I$ then it also contains one of the other two.
%
%So, assume that $P$ properly contains $I \oplus I$ and let $p=(p_1,p_2)$ be a lattice point in $P\setminus (I \oplus I)$. By the symmetry of $I \oplus I$ with respect to the coordinate axes we can assume that $p_1,p_2\ge 0$.
%Now, $P$ has to contain one of the points in $\{(2,0), (0,2), (1,1)\}$, so we assume that $p$ is one of them. The first two possibilities imply that $P$ contains a copy of~$I \oplus I'$ and the third implies it contains a copy of $S(\one_3)$.
\end{proof}

\begin{corollary}
\label{coro:ConjA_dim2}
Let $P$ be a non-hollow lattice polygon. Then
\[
\mu(P)\le 1,
\]
with equality if and only if $P$ is unimodularly equivalent to one of $S(\one_3)$, $I \oplus I$, or $I \oplus I'$.
\end{corollary}

\begin{proof}
By \Cref{lemma:dimtwo-b}, unless $P$ is one of $S(\one_3)$, $I \oplus I$ or $I \oplus I'$ it strictly contains one of them. If the latter happens then its covering radius is strictly smaller than $1$, since the three of them are tight by \Cref{lemma:simplex_monotone} and \Cref{lemma:sum_monotone}.
\end{proof}

\subsubsection*{\texorpdfstring{\cref{conj:one}}{Conjecture \ref{conj:one}} in dimension three}
For the three-dimensional case we introduce the following concept:

\begin{definition}
\label{defi:minimal}
A \emph{minimal $d$-polytope} is a non-hollow lattice $d$-polytope not properly containing any other non-hollow lattice $d$-polytope.
\end{definition}

In this language, our results in dimension 2 can be restated as: There are exactly three minimal $2$-polytopes, they have covering radius $1$, and every other non-hollow lattice $2$-polytope has strictly smaller covering radius.

In dimension three things are a bit more complicated. To start with, instead of three direct sums of (perhaps translated) simplices of the form $S(\one_i)$ there are  nine, that we now describe. 
As in the previous section, let $I=[-1,1] = S(\one_2)$ and $I'=[0,2]$. In a similar way we define:
\begin{align*}
S'(\one_3) &= (1,1) + S(\one_3) = \conv(\{(0,0), (2,1), (1,2)\}),\\
(I \oplus I')^\circ &= (0,-1) + (I \oplus I') = \conv(\{(0,1), (\pm1,-1)\}),\\
(I \oplus I')' &= (0, -2) + (I \oplus I') = \conv(\{(0,0), (\pm1,-2)\}).
\end{align*}
%
%\begin{align*}
%{S\,'}^{(1,1,1)} = &(1,1) + S(\one_3) = \conv \{(0,0), (2,1), (1,2)\},\\
%(I \oplus I)' = &(0,-1) + (I \oplus I) = \conv \{(0,0), (0,-2), (\pm 1,-1)\},\\
%(I \oplus I')^\circ = &(0,-1) + (I \oplus I') = \conv \{(0,1), (\pm1,-1)\},\\
%(I \oplus I')' = &(0, -2) + (I \oplus I') = \conv \{(0,0), (\pm1,-2)\}.
%\end{align*}
Put differently, $S'(\one_3)$ is  $S(\one_3)$ translated to have the origin as a vertex; the other two are $I \oplus I'$ translated to have the origin in the interior and at the ``apex'', respectively.

\begin{lemma}
\label{lemma:minimum_3d}
There are the following nine non-equivalent lattice $3$-polytopes of covering radius $3/2$, obtained as direct sums of (perhaps translated) simplices of the form $S(\one_d)$:
\[
\begin{array}{c}
S(\one_4), 
\\
S(\one_3) \oplus I, \qquad
S'(\one_3) \oplus I, \qquad
S(\one_3) \oplus I',
\\
I\oplus I\oplus I,  \qquad
I\oplus I\oplus I',
\\
(I \oplus I')^\circ \oplus I, \qquad
(I \oplus I')'\oplus I , \qquad
(I \oplus I')^\circ \oplus I'.
\end{array}
\]
\end{lemma}

The last five polytopes are illustrated in \Cref{fig:sums_of_I}, which is borrowed from~\cite[p.~123]{blancosantos2019nonspanning}.
Observe that the last three can equally be written as
\[
\begin{array}{ccc}
I \oplus (I \oplus I)' , \qquad
I \oplus (I' \oplus I)'  , \qquad
I \oplus (I \oplus I')'' ,
\end{array}
\]
where $(I \oplus I)'$ denotes $I \oplus I$ translated to have the origin as a vertex and $(I \oplus I')''$ is $I \oplus I'$ translated to have the origin at an endpoint of its edge of length two.

\begin{figure}[htb]
\centerline{\includegraphics[scale=.54]{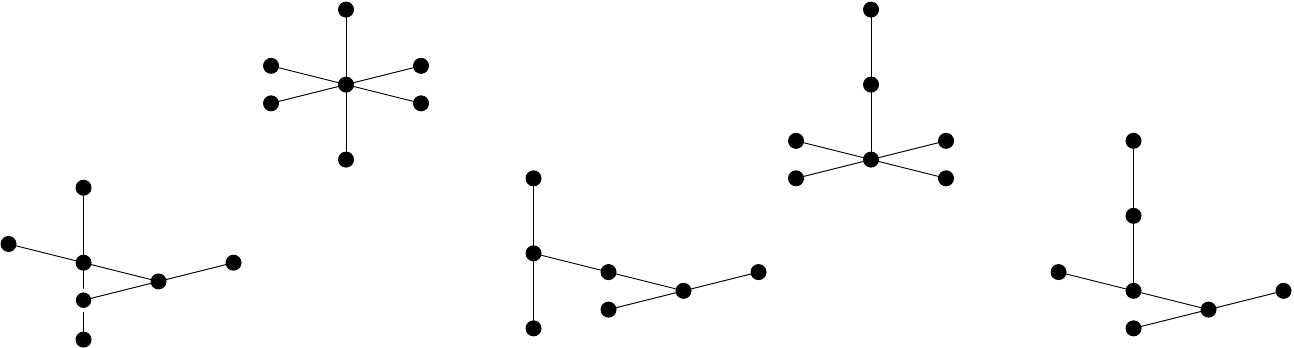}}
\caption{The five non-hollow lattice $3$-polytopes that can be obtained by translations and direct sums of $I=[-1,1]$ arise as the convex hull of the shown segments.}
\label{fig:sums_of_I}
\end{figure}

\begin{proof}
That all the described direct sums are non-hollow follows from the following more general fact: The direct sum of two or more non-hollow lattice polytopes containing the origin is non-hollow if (and only if) all but at most one of the summands has the origin in its interior. Indeed, if the summand exists then its interior point(s) are interior in the sum; if it doesn't then the origin is an interior point in the sum.

With this in mind, we only need to check that the nine described polytopes are pairwise unimodularly non-equivalent, which is left to the reader. 
\end{proof}

A second difference with dimension two is that these nine non-hollow lattice $3$-polytopes are no longer the only minimal ones. Minimal non-hollow $3$-polytopes have been classified and there are 26 with a single interior lattice point (see \cite[Thm.~3.1]{kasprzyk2010threefolds} and Tables 2 \& 4 therein) plus the infinite family described in \Cref{thm:M_k-3d} below.

To prove \cref{conj:one} in dimension three we show that, on the one hand, the covering radii of the 26 with a single interior lattice point can be explicitly computed and/or bounded, giving the following result, the proof of which we postpone to \Cref{sec:26minimal}.

\begin{theorem}
\label{thm:26minimal}
Among the 26 minimal non-hollow $3$-polytopes with a single interior lattice point, all except the nine in \Cref{lemma:minimum_3d}  have covering radius strictly smaller than $3/2$.
\end{theorem}

On the other hand, all the (infinitely many) minimal non-hollow $3$-poly\-topes with more than one interior lattice point have covering radius strictly smaller than $3/2$, as we now prove.
For any $k \in \N$ and $\alpha, \beta \in\{0,1\}$, we define $M_k(\alpha, \beta)$ to be the following lattice tetrahedron:
\[
M_k(\alpha,\beta)= \conv(\{(1,0,0), (-1,0,\alpha), (0,1,k+1), (0,-1, k+1-\beta)\}).
\]

\begin{theorem}[\protect{\cite[Prop.~4.2]{ballettikasprzyk2016twopoints}}]
\label{thm:M_k-3d}
Every minimal $3$-polytope with $k\ge2$ interior lattice points is equivalent by unimodular equivalence or refinement of the lattice to 
$M_k(\alpha,\beta)$ for some $\alpha,\beta \in \{0,1\}$.
\end{theorem}

\Cref{thm:M_k-3d} is a version of \cite[Prop.~4.2]{ballettikasprzyk2016twopoints}, although more explicit than the original one.
An example where refinement is needed in the statement is $M_k(0,0)$  considered with respect to the lattice $\Lambda$ generated by $\Z^3$ and $(1/q, 1-1/q,0)$, with $q$ and $k+1$ coprime. 
$M_k(0,0)$ is still minimal with respect to $\Lambda$ because it contains no point of $\Lambda \setminus \Z^3$.

\begin{proof}
Let $P$ be a minimal lattice $3$-polytope with more than one interior lattice point, and let $L$ be a line containing two of them. Without loss of generality we assume that $L = \{(0,0,z) : z\in \R\}$ and $L\cap P$ is the segment between $(0,0,z_1)$ and $(0,0,z_2)$, with $z_1 \in [0,1)$ and $z_2\in (r,r+1]$ for some $r \in \{2,\ldots,k\}$, so that $L$ contains $r$ interior lattice points of $P$.

{\bf Claim 1:} \emph{The minimal faces of $P$ containing respectively $(0,0,z_1)$ and $(0,0,z_2)$ are non-coplanar edges.}
Let $F_1$ and $F_2$ be those faces. If one of them, say $F_1$,  had dimension two, then $\conv(F_1 \cup \{(0,0,r)\})$ would be  a non-hollow lattice polytope strictly contained in~$P$. If one of them, say $F_1$,  had dimension zero then necessarily $F_1 = \{(0,0,z_1)\}=\{(0,0,0)\}$. This would imply $\conv(P \cap \Z^3 \setminus\{\zero\})$ to be a non-hollow lattice polytope strictly contained in~$P$.
Thus, $F_1$ and $F_2$ are both edges of~$P$.  They cannot be coplanar, since otherwise there would be vertices $p$ and $q$ of $P$, one on either side of the hyperplane $\aff(F_1 \cup F_2)$, and the polytope $\conv(F_1 \cup \{(0,0,r),p,q\})$ would be non-hollow and strictly contained in~$P$.

Hence, $\conv(F_1 \cup F_2)$ is a non-hollow lattice tetrahedron and by minimality, $P=\conv(F_1 \cup F_2)$. We denote $v_i$ and $w_i$ the vertices of $F_i$, for $i=1,2$.

{\bf Claim 2:} \emph{All the lattice points in the tetrahedron $P$ other than the four vertices are on the line $L$.}
%
%
%
%Minimality also implies that the only lattice points on the edges $F_1,F_2$ are the vertices $v_1, w_1$ and $v_2, w_2$, respectively plus, perhaps, $(0,0,z_1)$ and $(0,0,z_2)$.
%
%We now show that the only lattice points in $P$ other than the vertices are on the line $L$. Indeed, l
Let $H_i$ be the plane containing the line $L$ and the edge~$F_i$, for $i=1,2$. The polytope $Q= \conv(L\cap P \cup\{v_1, w_1, v_2\}) \subset P$ is contained in $H_1^+$, one of the two halfspaces defined by $H_1$; furthermore, the facet of $Q$ lying on $H_1$ is non-hollow, since $(0,0,1)$ is in its relative interior. Therefore, if $P$ contained any lattice point~$u$ other than the vertex $w_2$ in the open halfspace $(H_1^-)^o$ then $\conv(Q \cup \{u\})$ would be a non-hollow lattice polytope strictly contained in $P$.
%Since the relative interior of the non-hollow facet of $Q$ is in the interior of $Q'$, and the latter is strictly contained in $P$, this contradicts the minimality of~$P$. 
Thus there are no lattice points in the open halfspace $(H_1^-)^o$. Since the same can be said for the other halfspaces, $H_1^+$ and $H_2^\pm$, all lattice points of $P$ except its four vertices must lie on $L$.

In particular, we have $r=k$. 

{\bf Claim 3:} \emph{The endpoint $(0,0,z_i)$ equals the mid-point of the edge $F_i=\conv(\{v_i,w_i\})$}.
Let us only look at $i=1$, the other case being symmetric. Let $u_1=(0,0,1)$ and $u_2=(0,0,2)$ be the first two interior lattice points of $P$ along $L$. The triangles $\conv(\{u_1,u_2,v_1\})$ and $\conv(\{u_1,u_2,w_1\})$ are empty lattice triangles in the plane $H_1$, hence they have the same area. Thus, $v_1$ and $w_1$ are at the same distance from (and on opposite sides of) the line $L$, which implies the statement.

%Let $\Lambda \subseteq \Z^3$ be the affine lattice generated by the four vertices of $P$ $v_1, w_1, v_2, w_2$ and the lattice points along $P\cap L$. We want to show that $P$ with respect to $\Lambda$ is isomorphic to $M_k(\alpha,\beta)$ with respect to $\Z^3$.
%An affine basis for $\Lambda$ is $v_1, u_1=(0,0,1), u_2=(0,0,2)$, and $v_2$. To see this, observe that all points of~$L_P$ are in the lattice generated by $u_1$ and $u_2$; note further that the set $v_i, u_1, u_2$ is an affine basis for $\Lambda \cap H_i$, (as is $w_i, u_1, u_2$), since in the plane the vertices of any empty lattice triangle form a lattice basis. 
%Thus we can write $w_1 -u_1 =a(v_1-u_1)+b(u_2-u_1)$ and $v_1-u_1=c(w_1-u_1)+d(u_2-u_1)$, with $a,b,c,d \in \Z$.
%Since the edge $F_1$ is not parallel to~$L$, this implies that $w_1 = -v_1 + (b+2)u_1$.
%Similarly, $w_2 = -v_2 + (b'+2)u_1$, for some $b' \in \Z$.
%
%Thus, $(0,0,z_i)$ is the midpoint of the edge $F_i$, for $i=1,2$.

In particular, $z_1\in [0,1)$ and $z_2\in (k,k+1]$ are either integers or half-integers, so they can be written as 
$z_1=\alpha/2$ and $z_2=k+1-\beta/2$ for some $\alpha,\beta \in \{0,1\}$. It is now clear that the affine transformation that fixes $L$ and sends 
%Along with our conditions on $L_P$, we get $b+2 = \alpha \in \{0,1\}$ and $b'+2=2(k+1)-\beta$, for some $\beta \in \{0,1\}$.
%It is now easy to see that the isomorphism between~$\Lambda$ and~$\Z^3$ which sends 
$v_1 \mapsto (1,0,0)$ and $v_2 \mapsto (0,1,k+1)$, sends~$P$ to $M_k(\alpha, \beta)$. The map may send $\Z^3$ to a different lattice $\Lambda$, but $\Lambda$ refines $\Z^3$ since $(1,0,0)$, $(0,1,k+1)$, $(0,0,1)$ and $(0,0,2)$ are in $\Lambda$ and they generate $\Z^3$.
%
%Let $\Lambda \subseteq \Z^3$ be the affine lattice generated by the four vertices of $P$ $v_1, w_1, v_2, w_2$ and the lattice points $L_P$ on the line $L$. We want to show that $P$ with respect to $\Lambda$ is isomorphic to $M_k(\alpha,\beta)$ with respect to $\Z^3$.
%%To do this, observe that $\Lambda \cap H_1$ is affinely generated by $v_1$, $(0,0,\lceil z_1 \rceil)$, and $(0,0,\lceil z_1 \rceil +1)$. Similarly, $\Lambda \cap H_2$ is generated by $v_2$, $(0,0,\lfloor z_2 \rfloor)$, and $(0,0,\lfloor z_2 \rfloor -1)$. 
%An affine basis for $\Lambda$ is $\{v_1, l_0=(0,0,\lceil z_1 \rceil), l_1=(0,0,\lceil z_1 \rceil +1)$, and $v_2$\}. To see this, observe that all points of $L_P$ are in the lattice generated by $l_0$ and $l_1$; note further that the set $\{v_i, l_0, l_1\}$ is an affine basis for $\Lambda \cap H_i$ is, (as is $\{w_i, l_0, l_1\}$), since in the plane the vertices of any empty lattice triangle form a lattice basis. 
%Thus we can write, for some $a,b,c,d \in \Z$,
%\begin{align*}
%w_1 -l_0 &=a(v_1-l_0)+b(l_1-l_0)\\
%v_1-l_0 &=c(v_1-l_0)+d(l_1-l_0).
%\end{align*}
%
%This, along with our conditions on $L_P$, yields explicitly $w_1=-v_1 + \alpha l_1 +(2-\alpha)l_0$, for some $\alpha \in \{0,1\}$. Similarlyt we obtain $w_2=-v_2 +(2(k+\alpha)-\beta )l_1 + (\beta-2(k+\alpha-1))l_0$, where $\beta \in \{0,1\}$. It is now easy to see that the isomorphism between $\Lambda$ and $\Z^3$ defined by $v_1 \mapsto (1,0,0)$, $l_0 \mapsto (0,0,z_1)$, $l_1 \mapsto (0,0,z_1+1)$, $v_2 \mapsto (0,1,k+1)$, we send $P$ to $M_k(\alpha, \beta)$. 
\end{proof}

\begin{corollary}
\label{coro:M_k-3d}
Every minimal $3$-polytope with $k\ge2$ interior lattice points has covering radius strictly smaller than $3/2$.
\end{corollary}

\begin{proof}
The projection of $M_k(\alpha,\beta)$ along the $z$ direction is $I\oplus I$ and the fiber over the origin is the segment $\{0\}\times \{0\} \times [\alpha/2, k+1 - \beta/2]$, of length $k+1 - (\alpha + \beta)/2$. 
Thus, by \Cref{lem:projection},
\[
\mu(M_k(\alpha,\beta)) \le \mu(I\oplus I) + \mu([\alpha/2, k+1 - \beta/2]) = 1 + \frac1{k+1 -\frac{\alpha + \beta}2} \le \frac32.
\]
Moreover, the last inequality is met with equality only in the case $k=2$, $\alpha=\beta=1$.
But for $M_2(1,1)$ we can consider the projection $(x,y,z) \mapsto x$, whose image is $I$ and whose fiber is
\[
\conv(\{(0,1/2), (1,3), (-1,2)\}) \cong S(3/2,1,1).
\]
Thus, by \Cref{lem:projection} and \Cref{thm:lambda}, we have
\[
\mu(M_2(1,1)) \le \mu(I) + \mu(S(3/2,1,1)) = \frac12 + \frac{7/3}{8/3} = \frac{11}8<\frac32.\qedhere
\]
\end{proof}

\noindent In fact we can be more explicit:

\begin{remark}
\label{rem:M_k-dim3}
The covering radius of $M_k(\alpha,\beta)$ admits a closed expression:
\begin{align*}
\mu(M_k(0,0)) &= \mu (I \oplus [0,k+1] \oplus I) = 1 + \frac1{k+1}.
\\
\mu(M_k(1,0)) &= \mu(M_k(0,1)) = \mu (I \oplus M_k(1)) %= \frac12 + \frac{k+2}{2k+1}
= 1 + \frac{3}{4k+2},
\\
\mu(M_k(1,1)) &= 1+ \frac{1}{2k}.
\end{align*}
The first formula directly follows from \cref{lem:projection}. The second one also does, using \Cref{rem:M_k-dim2}. For the third one,
see \cref{lem:Mk11} in \cref{sec:graph-method}.
For $k=1$ the three expressions reduce to $3/2$, which follows also from $M_1(0,0) \cong I\oplus (I\oplus I')'$, $M_1(0,1) \cong I \oplus S(\one_3)$, and $M_1(1,1) \cong S(\one_4)$. 
%\paco{My conjecture is $\mu(M_k(1,1))=\frac{2k+1}{2k}$. 
%This gives $\mu(M_k(1-\alpha,1-\beta))=1 + \frac12\frac{2+\alpha+\beta}{2k+\alpha+\beta}$
%}
%We do not have a formula for $\mu(M_k(1,1))$, except  $\mu(M_1(1,1))= \mu(S(\one_4)) = 3/2$, and $\mu(M_2(1,1)) = 5/4$ which we have computed with the techniques described in \Cref{sec:26computations}.
%
%An upper bound can be computed via \cref{lem:projection}: The triangle $T_k = M_k(1,1) \cap \{x=0\}$ is equivalent to $S(k-1/2,1,1)$ and the projection of $M_k(1,1)$ onto the $x$-coordinate has length two.
%Thus, using \cref{thm:lambda} as well, we obtain
%\matthias{can we do better and prove Paco's conjecture?}
%\[
%\mu(M_k(1,1)) \leq \frac12 + \mu(T_k) = \frac12 + \frac{1+\frac{2}{k-\frac12}}{2+\frac{1}{k-\frac12}} = \frac{2k + 3/2}{2k}.
%\]
\end{remark}

We are now ready to prove \cref{conj:one} in dimension three:

\begin{theorem}
\label{thm:ConjA_dim3}
Let $P$ be a non-hollow lattice $3$-polytope. Then
\[
\mu(P)\le \frac32,
\]
with equality if and only if $P$ is unimodularly equivalent to one of the nine polytopes in \Cref{lemma:minimum_3d}.
\end{theorem}

\begin{proof}
Let $P$ be a non-hollow lattice $3$-polytope, and let~$T$ be a minimal one contained in it. If $T$ is not one of the nine in \Cref{lemma:minimum_3d} then $T$, and hence~$P$, has covering radius strictly smaller than $3/2$ by either \Cref{coro:M_k-3d} or
\Cref{thm:26minimal}. If $T$ is one of the nine and $P\ne T$ then
\[
\mu(P) < \mu(T) = \frac32,
\]
since these nine are tight by \Cref{lemma:simplex_monotone} and \Cref{lemma:sum_monotone}.
\end{proof}

\section{\texorpdfstring{\cref{conj:volume}}{Conjecture \ref{conj:volume}}}
\label{sec:volume}

We here focus on \cref{conj:volume}.
We show that it implies \cref{conj:one}, we prove it up to a factor of two in arbitrary dimension, and we prove it in dimension two.
Finally, in \cref{sec:0inboundary}, we investigate how the proposed bound changes if we allow the origin to be contained in the boundary of the given simplex.

As a preparation, let us first reinterpret \cref{conj:volume} in terms of (reciprocals of) certain lengths.
To this end, let $S = \conv(\{v_0,\ldots,v_d\})$ be a $d$-simplex with the origin in its interior, and assume that it has \emph{rational vertex directions}, that is, the line through the origin and the vertex $v_i$ has rational direction, for every $0\leq i \leq d$.

As in \cref{conj:volume}, let $\pi_i$ be the linear projection to dimension $d-1$ vanishing at $v_i$. 
Finally, let $\ell_i$ be the lattice length of $S \cap \pi_i^{-1}(\zero)$. 
%\paco{$\ell$ was a dimension earlier}
Put differently, let  $u_i$ be the point where the ray from $v_i$ through~$\zero$ hits the opposite facet of $S$ and let $\ell_i$ be the ratio between the length of $[u_i,v_i]$ and the length of the primitive lattice vector in the same direction. In formula:
\[
\ell_i :=  \Vol_{\Z^d \cap \R v_i}([u_i, v_i]).
\]

\begin{lemma}
\label{lemma:ell}
For every $i \in \{0,1,\ldots,d\}$, we have
\[
\frac1{\ell_i} = \frac { \Vol_{\pi_i(\Z^d)}(\pi_i(S))}{\Vol_{\Z^d}(S)}.
\]
In particular, \Cref{conj:volume} is equivalent to the inequality
\begin{align}
\label{eq:ell}
\mu(S) \leq \frac12 \sum_{i=0}^d \frac1{\ell_i}.
\end{align}
\end{lemma}

\begin{proof}
By construction, we have $\pi_i(S) = \pi_i(F_i)$, where $F_i$ is the facet of~$S$ opposite to the vertex~$v_i$.
Therefore, $\vol(S) = \frac{1}{d} \vol(\pi_i(S)) \vol([u_i,v_i])$.
The determinants of the involved lattices are related by $1=\det(\Z^d)=\det(\pi_i(\Z^d))\det(\Z^d \cap \R v_i)$ (cf.~\cite[Prop.~1.2.9]{martinet2003perfect}).
Hence,
\begin{align*}
\Vol_{\Z^d}(S) &= \frac{d! \vol(S)}{\det(\Z^d)} = \frac{(d-1)! \vol(\pi_i(S))}{\det(\pi_i(\Z^d))} \frac{\vol([u_i,v_i])}{\det(\Z^d \cap \R v_i)} \\
&= \Vol_{\pi_i(\Z^d)}(\pi_i(S)) \Vol_{\Z^d \cap \R v_i}([u_i, v_i]),
\end{align*}
as desired.
\end{proof}

We now also detail the claim in the introduction, that the discrete surface area defined in \cref{def:discretesurfarea} is invariant under unimodular transformations.

\begin{lemma}
\label{cor:discrsurfarea_invariance}
Let $S$ be a $d$-simplex with the origin in its interior and with rational vertex directions.
Let $A$ be an invertible linear transformation.
Then
\[
\Surf_{A\Z^d}(AS) = \Surf_{\Z^d}(S).
\]
In particular, if $A$ is unimodular, we have $\Surf_{\Z^d}(AS) = \Surf_{\Z^d}(S)$.
\end{lemma}

\begin{proof}
As before we write $S = \conv(\{v_0,\ldots,v_d\})$ and we let $\pi_i$ be the projection vanishing at~$v_i$, for $0 \leq i \leq d$.
Clearly, $AS = \conv(\{Av_0,\ldots,Av_d\})$ and the corresponding projection $\bar \pi_i$ vanishing at $Av_i$ can be written as $\bar \pi_i = \pi_i A^{-1}$.
Therefore, we get
\[
\Surf_{A\Z^d}(AS) = \sum_{i=0}^d \Vol_{\bar \pi_i(A\Z^d)}(\bar \pi_i(AS)) = \sum_{i=0}^d \Vol_{\pi_i(\Z^d)}(\pi_i(S)) = \Surf_{\Z^d}(S),
\]
as claimed.
\end{proof}

\subsection{\texorpdfstring{\cref{conj:volume}}{Conjecture \ref{conj:volume}} implies \texorpdfstring{\cref{conj:one}}{Conjecture \ref{conj:one}}}
\label{sec:volume_implies_one}

\begin{corollary}
\label{coro:vol_one}
\cref{conj:volume} \ $\Longrightarrow$ \cref{conj:one}.
\end{corollary}

\begin{proof}
In view of \Cref{cor:ReduToSimpls}, it suffices to consider lattice simplices.
Therefore, let $S = \conv(\{v_0,\ldots,v_d\})$ be a lattice $d$-simplex containing the origin in its interior.
Furthermore, let $\a_i$ be the lattice length of the segment $[\zero,v_i]$.
Then, $1- \a_i/\ell_i$ is the $i$-th barycentric coordinate of the origin with respect to the vertices of~$S$, so that
\[
\sum_{i=0}^d \left(1- \frac{\a_i}{\ell_i}\right) = 1
\]
and, hence, $\sum_{i=0}^d \a_i/\ell_i = d$.
On the other hand, for a lattice simplex we have $\a_i\ge 1$. Thus, assuming \Cref{conj:volume} holds for~$S$, we have
\[
\mu(S) \leq \frac12 \sum_{i=0}^d \frac1{\ell_i} 
\le
\frac12 \sum_{i=0}^d \frac{\a_i}{\ell_i} 
%= \frac{(d+1) -\sum_i b_i}2
=\frac{d}2. \qedhere
\]
\end{proof}

\subsection{\texorpdfstring{\cref{conj:volume}}{Conjecture \ref{conj:volume}} holds up to a factor of two}

In the formulation of \cref{lemma:ell}, \cref{conj:volume} is easily proved inductively up to a factor of two.

\begin{proposition}
\label{prop:volume_factor_2}
Let $S = \conv(\{v_0,\ldots,v_d\})$ be a $d$-simplex with the origin in its interior and with rational vertex directions.
Then
\[
\mu(S)\leq 
\sum_{i=0}^d \frac{1}{\ell_i},
\]
with the lattice lengths $\ell_i$ defined as above.
\end{proposition}

\begin{proof}
As above, let $u_i$ be the intersection of the line $\R v_i$ with the facet $F$ of $S$ opposite to $v_i$, so that $\ell_i$ is the lattice length of $Q := [u_i,v_i] \subseteq S$.
Note, that $u_i$ lies in the relative interior of~$F$.
Also, let $\pi_i$ be the linear projection vanishing at~$v_i$.
By the assumptions on~$S$, the projection $\pi_i$ is rational and thus $\pi_i(S)$ is  
a $(d-1)$-dimensional simplex having the origin in its interior and with rational vertex directions with respect to~$\pi_i(\Z^d)$.

Using \cref{lem:projection} and the induction hypothesis for~$\pi_i(S)$, we get
\begin{align}
\mu(S,\Z^d) &\leq \mu(Q,\Z^d \cap L_Q) + \mu(\pi_i(S),\pi_i(\Z^d))
\leq \frac{1}{\ell_i} + \sum_{j \neq i} \frac{1}{\ell_j'},\label{eq:induct_conj:volume}
\end{align}
where the $\ell_j'$ are the corresponding lattice-lengths in $\pi_i(S)$.
Thus, to prove the proposition we only need to show that $\ell_j' \geq \ell_j$, for all $j \neq i$.
In fact, since the one-dimensional lattice $\pi_i(\Z^d) \cap \pi_i(\R v_j)$ refines 
$\pi_i(\Z^d \cap \R v_j)$, we have
\begin{align*}
\ell_j = \Vol_{\Z^d \cap \R v_j}([u_j, v_j])
&=  \Vol_{\pi_i(\Z^d \cap \R v_j)} ([\pi_i(u_j), \pi_i(v_j)]) \\
&\le  \Vol_{\pi_i(\Z^d) \cap \pi_i(\R v_j)} ([\pi_i(u_j), \pi_i(v_j)]) \leq \ell_j'. 
\end{align*}
%\matthias{I changed to non-strict inequality, because we just need this for the claimed inequality. I uncommented the subsequent remark. If we want to discuss strict inequality, then we should want to prove a strict upper bound on $\mu(S)$. I decided for simplicity and conciseness.}
Here, the last inequality comes from the fact that $[\pi_i(u_j), \pi_i(v_j)] \subseteq \pi_i(S)$ is contained in the ray from the vertex $\pi_i(v_j)$ of $\pi_i(S)$ through the origin.
\end{proof}

%\begin{remark}
%In fact, it holds $\ell_j' > \ell_j$.
%This is made explicit in \cref{eq:ell_diff}, where the difference 
%$\frac{1}{\ell_j} - \frac{1}{\ell'_j} $ is computed under the assumption that the lattices $\pi_i(\Z^d) \cap \pi_i(\R v_j)$ and $\pi_i(\Z^d \cap \R v_j)$ coincide.
%\end{remark}

\begin{remark}
\cref{coro:volume_dim2} in the next section proves \cref{conj:volume} in the plane.
So we can base the inductive proof above on the stronger assumption that $\mu(S') \leq c_{d-1} \sum_{i=0}^{d-1} \frac{1}{\ell_i'}$, where~$S'$ is a $(d-1)$-dimensional simplex and~$c_{d-1}$ is a suitable constant with $c_2 = 1/2$.
Summing the thus modified inequality~\eqref{eq:induct_conj:volume} for all indices $0 \leq i \leq d$, yields the recursion $(d+1)c_d = 1+d c_{d-1}$.
Solving it shows that
\[
\mu(S) \leq \frac{2d-1}{2d+2} \sum_{i=0}^d \frac{1}{\ell_i},
\]
for all $d$-simplices~$S$ with the origin in its interior and with rational vertex directions.
This is a good bound in $\R^3$ since $c_3 = 5/8$.
\end{remark}

\subsection{\texorpdfstring{\cref{conj:volume}}{Conjecture \ref{conj:volume}} in dimension two}
\label{sec:conj:volume_dim2}

In this section we prove \Cref{conj:volume} in dimension two. Our first remarks are valid in arbitrary dimension.

Throughout, let $S = \conv(\{v_0,\ldots,v_d\})$ be a simplex with the origin in its interior and with rational vertex directions. 
For each $i=0,\ldots,d$, let $p_i$ be the primitive positive multiple of $v_i$.
%a lattice point in the (positive) direction of $v_i$ and let
%Our primary case of interest is when $p_i$ is the primitive multiple of $v_i$, 
%but in what follows we allow $p_i$ not to be primitive.
%\Paco{for the $d=2$ case we take $p_i$ primitive, but for the higher-dim attempt we may need to allow it not to be primitive.}
Let $\alpha = (\alpha_0,\dots,\alpha_d)\in \N^{d+1}$ be the primitive integer linear dependence among the $p_i$'s. That is,
\[
\sum_{i=0}^d \alpha_i p_i = \zero \qquad \text{and} \qquad \gcd(\alpha_0,\dots,\alpha_d) = 1.
\]
Denoting the Euclidean length of a vector $x \in \R^d$ by $\|x\|$, and writing $\beta_i = \alpha_i \|p_i\| / \|v_i\| \in \R_{>0}$, for each $i = 0,\ldots,d$, we have
%\matthias{introduced $\|x\|$ here} 
\[
\sum_{i=0}^d \beta_i v_i = \sum_{i=0}^d \alpha_i p_i  = \zero.
\]

\begin{remark}
\label{rem:coprime}
The fact that the $p_i$'s are primitive imposes some condition on the vector $\alpha \in \N^{d+1}$.
Namely, for each $i\in \{0,\ldots,d\}$, we have
\[
\gcd(\alpha_j : j\ne i) = 1.
\]
Indeed, let $\Lambda$ be the lattice generated by $\{p_0,p_1,\ldots,p_d\}$, and let $\Lambda_i$ be the sublattice generated by $\{p_j : j \ne i\}$.
Then, the primitive vector of $\Lambda_i$ in the direction of $p_i$ is
\[
\frac{\sum_{j\ne i} \alpha_j p_j}{\gcd(\alpha_j : j\ne i)}=
\frac{- \alpha_i p_i}{\gcd(\alpha_j : j\ne i)},
\]
which is an integer multiple of $p_i$ if, and only if, $\gcd(\alpha_j : j\ne i)= 1$.
%Indeed, for a given positive primitive integer vector $\alpha$, let $p_0,\dots,p_d$ be arbitrary points in general position satisfying $\sum_i \alpha_i p_i=0$ and without loss of generality let $\Lambda$ be the linear lattice generated by them. Let $\Lambda_i$ be the sublattice generated by $\{p_j : j \ne i\}$. Then, the primitive vector of $\Lambda_i$ in the direction of $p_i$ is
%\[
%\frac{\sum_{j\ne i} \alpha_j p_j}{\gcd(\alpha_j : j\ne i)}=
%\frac{- \alpha_i p_i}{\gcd(\alpha_j : j\ne i)},
%\]
%which is an integer multiple of $p_i$ if, and only if, $\gcd(\alpha_j : j\ne i)= 1$.
\end{remark}

As in the previous sections, for each $i$ let $\ell_i$ be the lattice length of $S \cap \R v_i$. 
The following lemma says that the vectors $\alpha$ and $\beta = (\beta_0,\beta_1,\ldots,\beta_d)$  contain all the information about $S$ needed to compute the right-hand side in~\eqref{eq:ell}.

\begin{lemma}
\label{lemma:barycentric}
The lattice length of $S \cap \R v_i$ equals 
\[
\ell_i = \frac{\alpha_i}{\beta_i} + \frac{\alpha_i}{\sum_{j\ne i} \beta_j}=
\frac{\alpha_i}{\beta_i} \cdot  \frac{\sum_{j=0}^d \beta_j}{\sum_{j\ne i} \beta_j}.
\]
\end{lemma}

\begin{proof}
To slightly simplify notation, we do the computations for $i=0$.
For this, let us use the vectors $p_1,\dots,p_d$ as the basis for a linear coordinate system in $\R^d$.
In these coordinates, $p_0$ becomes
\[
p_0 = - \frac{1}{\alpha_0} \left( \alpha_1,\dots,\alpha_d \right).
\]
On the other hand, the equation of the facet of $S$ opposite to $v_0$ is
\[
\sum_{j=1}^d \frac{\beta_j}{\alpha_j} x_j = 1, 
\]
so that this facet intersects the line spanned by $p_0$ in the point
\begin{equation}
\frac{\left({\alpha_1}, \dots, {\alpha_d}\right)}{ \sum_{j=1}^d \beta_j}
=
\frac{- \alpha_0} {\sum_{j=1}^d \beta_j} \, p_0.
\end{equation}
Thus, the segment $S \cap \R v_0$ has endpoints $\frac{\alpha_0}{\beta_0} p_0$ and $ \frac{-\alpha_0}{\sum_{j=1}^d \beta_j} p_0$, which implies the statement.
\end{proof}

\begin{remark}
Observe that the quantity $\a_i$ in the proof of \cref{coro:vol_one} equals $\alpha_i/\beta_i$. With this in mind, one easily recovers the equality $\sum_i \frac{\a_i}{\ell_i}=d$ used in that proof, from \cref{lemma:barycentric}.
\end{remark}

\subsubsection*{Specializing to dimension two}

Our proof of \Cref{conj:volume} in two dimensions is based on applying \Cref{lem:projection} to the projection $\pi: \R^2 \to \R$ along the direction of $v_i$, for some fixed $i \in \{0,1,2\}$.
Then, with the notation above,
\begin{enumerate}[i)]
\item $\alpha_0$, $\alpha_1$ and $\alpha_2$ are pairwise coprime, by \Cref{rem:coprime}.
\item The lattice length of $S\cap \pi^{-1}(\zero)$ is $\ell_i$.
\item The lattice length of $\pi(S)$ equals
\[
 \frac{\alpha_j\alpha_k}{\beta_j}+ \frac{\alpha_j\alpha_k}{\beta_k} =
\frac{\alpha_j\alpha_k}{\beta_j\beta_k}(\beta_j + \beta_k),
\]
where $\{j,k\} = \{0,1,2\} \setminus \{i\}$.
Here we use that 
the projection  of the segment $[\zero,v_j]= \frac{\alpha_j}{\beta_j} [\zero,p_j]$ has length $\alpha_k \frac{\alpha_j}{\beta_j}$,
since $\gcd(\alpha_j,\alpha_k)=1$ implies that $\pi(\frac{p_j}{\alpha_k})$ is a primitive lattice point in the projection.
\end{enumerate}
Writing $L=\pi^{-1}(\zero)$, \Cref{lem:projection} gives us
\[
\mu(S) \le \mu(S\cap L,\Z^2 \cap L) + \mu(\pi(S),\pi(\Z^2)).
\]
Hence, the inequality~\eqref{eq:ell} would follow from:
\begin{equation}
\label{eq:volume_dim2}
\frac1{\ell_j} + \frac1{\ell_k} - \frac1{\ell_i} \ge \frac{2\beta_j\beta_k}{\alpha_j\alpha_k(\beta_j + \beta_k)} . 
\end{equation}
We prove this inequality under mild assumptions.

\begin{theorem}
\label{thm:volume_dim2}
Let $S = \conv(\{v_0,v_1,v_2\}) \subseteq \R^2$ be a triangle with the origin in its interior and with rational vertex directions.
Let the vectors $\alpha$ and $\beta$, and the lengths $\ell_i$ be defined as above, and let $p_0$, $p_1$ and $p_2$ be primitive in the directions of $v_0$, $v_1$ and $v_2$.
Assume that $(\alpha_0,\alpha_1,\alpha_2) \ne (1,1,1)$.
Then, the inequality~\eqref{eq:volume_dim2} holds for some choice of $i\in \{0,1,2\}$.

Moreover, the inequality is strict unless $(\alpha_0,\alpha_1,\alpha_2) = (2,1,1)$ and $\beta_1=\beta_2$, up to reordering the indices.
\end{theorem}

\begin{example}\ 
\begin{enumerate}[i)]
\item
\label{exm:equal_alphas}
The necessity of $(\alpha_0,\alpha_1,\alpha_2) \ne (1,1,1)$ is shown by the following example.
If  $S=S(1,1,1)$ (so that $\alpha_i = \beta_i = 1$ for all $i$), then
\[
\frac1{\ell_j} + \frac1{\ell_k} - \frac1{\ell_i}  = \frac23
\qquad \text{and} \qquad
\frac{2\beta_j \beta_k}{\alpha_j \alpha_k(\beta_k + \beta_k)} =1,
\]
so the inequality fails.
%More generally, if $\alpha_0=\alpha_1=\alpha_2=1$ then \Cref{eq:volume_dim2} is equivalent to
%\[
%\beta_i \ge   \frac{2\beta_j\beta_k(\beta_j + \beta_k)}{\beta_j^2 + \beta_k^2},
%\]
%which holds only of $\beta_i$ is large enough compared to $\beta_j$ and $\beta_k$.

\item
\label{exm:small_ell}
Even if $(\alpha_0,\alpha_1,\alpha_2) \ne (1,1,1)$, it is not true that~\eqref{eq:volume_dim2} holds \emph{for every} $i\in \{0,1,2\}$.
For $\a>0$, consider the simplex
\[
S = \conv(\{(0,\a), (-1,-1), (1,-1)\}).
\]
It has parameters $(\alpha_0, \alpha_1, \alpha_2) = (2,1,1)$, $(\beta_0, \beta_1, \beta_2) = \left(\frac2\a, 1,1\right)$, $\ell_0=\a+1$, and $\ell_1 = \ell_2 = \frac{2\a+2}{\a+2}$.
For $i=0$, we indeed have
\[
\frac1{\ell_1} + \frac1{\ell_2} - \frac1{\ell_0} = 1 
=
\frac{2\beta_1\beta_2}{\alpha_1\alpha_2(\beta_1 + \beta_2)}.
\]
But for $i \in \{1,2\}$, we get 
\[
\frac1{\ell_j} + \frac1{\ell_k} - \frac1{\ell_i} = \frac1{\ell_0} = \frac1{\a+1}
<
\frac2{\a+2}
=
\frac{2\beta_j\beta_k}{\alpha_j\alpha_k(\beta_j + \beta_k)}.
\]
\end{enumerate}
\end{example}

\begin{proof}[Proof of \Cref{thm:volume_dim2}]
\emph{Case 1: At most one of the $\alpha_i$s equals 1.}
Say $\alpha_1\ne 1\ne \alpha_2$. 
With no loss of generality assume $\ell_2\ge \ell_1$.
%$\frac{\beta_1}{\alpha_1}(\beta_0 + \beta_2)  \ge \frac{\beta_2}{\alpha_2}(\beta_0 + \beta_1)$.
Then, by \Cref{lemma:barycentric},
\begin{align*}
\frac1{\ell_0} + \frac1{\ell_1} - \frac1{\ell_2} 
&% > \frac3{2\ell_1} - \frac1{2\ell_2} 
 \ge \frac1{\ell_0} 
= \frac{\beta_0}{\alpha_0} \cdot \frac{\beta_1 + \beta_2}{\beta_0+\beta_1+\beta_2} 
 > \frac{\beta_0}{\alpha_0} \cdot \frac{\beta_1}{\beta_0 + \beta_1}
  \ge \frac{2\beta_0\beta_1}{\alpha_0\alpha_1(\beta_0 + \beta_1)}.    
\end{align*}

\emph{Case 2: Two of the $\alpha_i$s equal 1.}
Assume that $\alpha_1=\alpha_2=1$. The condition $(\alpha_0,\alpha_1,\alpha_2) \ne (1,1,1)$ then implies $\alpha_0\ge 2$, so that \Cref{lemma:barycentric} gives 
\begin{align*}
\frac1{\ell_1} + \frac1{\ell_2} - \frac1{\ell_0} 
& = \frac{\beta_1(\beta_0 + \beta_2)}{\beta_0+\beta_1+\beta_2} 
     + \frac{\beta_2(\beta_0 + \beta_1)}{\beta_0+\beta_1+\beta_2} 
     - \frac{\beta_0}{ \alpha_0} \cdot \frac{\beta_1 + \beta_2}{\beta_0+\beta_1+\beta_2} \\
& = \frac{
          2 \beta_1\beta_2
          +  \left(1-\frac1{\alpha_0}\right) \beta_0 (\beta_1 +\beta_2) 
          }
     {\beta_0+\beta_1+\beta_2} 
\overset{*}{\ge} \frac{
          2 \beta_1\beta_2
          + \frac12 \beta_0 (\beta_1 +\beta_2)
          }
     {\beta_0+\beta_1+\beta_2}.
%   \ge
%    \frac{\sqrt{ \beta_1\beta_2}}{\beta_0+\beta_1+\beta_2} \cdot \left(2\sqrt{\beta_1\beta_2} + \beta_0\right),
\end{align*}
Thus, the inequality we want to prove is 
\begin{align*}
\frac{
          2 \beta_1\beta_2
          + \frac12 \beta_0 (\beta_1 +\beta_2)
          }
     {\beta_0+\beta_1+\beta_2} \ge
 \frac{2\beta_1\beta_2}{\beta_1 + \beta_2}
\end{align*}
or, equivalently,
\begin{align*}
          2 \beta_1\beta_2 (\beta_1 +\beta_2)
          + \frac12 \beta_0 (\beta_1 +\beta_2)^2 \ge
    2\beta_1\beta_2 (\beta_0+\beta_1+\beta_2).
\end{align*}
This is equivalent to $(\beta_1 +\beta_2)^2 \overset{*}{\ge} 4 \beta_1\beta_2$, which clearly holds.

The two inequalities we used, marked with ``$\overset{*}{\ge}$'', are equalities if and only if $\alpha_0 =2$ and $\beta_1=\beta_2$, respectively.
\end{proof}

We now prove \Cref{conj:volume} for $d=2$ which also gives another proof of \cref{conj:one} in the plane.

\begin{corollary}
\label{coro:volume_dim2}
\Cref{conj:volume} holds in dimension two.
\end{corollary}

\begin{proof}
Let $S = \conv(\{v_0,v_1,v_2\}) \subseteq \R^2$ be a triangle with the origin in its interior and with rational vertex directions.
Let the vectors $\alpha$ and $\beta$, and the lengths $\ell_i$ be defined as above, taking $p_0$, $p_1$ and $p_2$ primitive. 
In view of \Cref{lemma:ell} we need to show that
\[
\mu(S) \le \frac12\left( \frac1{\ell_0} + \frac1{\ell_1} + \frac1{\ell_2} \right).
\]
If $(\alpha_0,\alpha_1,\alpha_2) = (1,1,1)$, then consider the lattice $\Lambda$ generated by $p_0,p_1,p_2$.
Let $A$ be the linear transformation sending $e_i$ to $p_i$, for $i=1,2$.
Then, $\Lambda = A\Z^2$ and $S=A S(\a)$ for a suitable $\a \in \R^3_{>0}$.
Moreover, since the $p_i$s are primitive, the lattice lengths $\ell_i$ are the same for every pair $(S,\Z^2)$, $(S,\Lambda)$, and $(S(\a),\Z^2)$.
Observing that $\Lambda \subseteq \Z^2$ is a sublattice, we may therefore apply \Cref{thm:lambda} and get
\[
\mu(S) \leq \mu(S,\Lambda) = \mu(S(\a),\Z^2) = \frac12\left( \frac1{\ell_0} + \frac1{\ell_1} + \frac1{\ell_2} \right).
\]
So, we assume that $(\alpha_0,\alpha_1,\alpha_2) \neq (1,1,1)$ and thus we can apply \Cref{thm:volume_dim2}, which provides us with an index $i \in \{0,1,2\}$ such that the inequality~\eqref{eq:volume_dim2} holds.
As we saw above, this implies the desired bound.
%
%As mentioned in \Cref{rem:coprime} the fiber of the origin has length $\ell_0$, and the projection of $S$ has length
%\[
%%\alpha_2 \frac{\alpha_1}{\beta_2}+\alpha_2 \frac{\alpha_1}{\beta_2} =
%\frac{\alpha_1\alpha_2}{\beta_1\beta_2}(\beta_1 + \beta_2).
%\]
%Then, \Cref{lem:projection} and \Cref{thm:lambda} give:
%\begin{align*}
%\mu(S) 
%&\le \mu(S\cap \pi^{-1}(0)) + \mu(\pi(S) ) 
%= \frac1{\ell_0} + \frac{\beta_1\beta_2}{\alpha_1\alpha_2(\beta_1 + \beta_2)} \\
%&\le  \frac1{\ell_0} + \frac12\left( \frac1{\ell_1} + \frac1{\ell_2} - \frac1{\ell_0} \right)
% =   \frac12\left( \frac1{\ell_0} + \frac1{\ell_1} + \frac1{\ell_2}\right). \qedhere
%\end{align*}
\end{proof}

%\subsection{What if the origin lies in the boundary?}
\subsection{Analogs to \texorpdfstring{\cref{conj:volume}}{Conjecture \ref{conj:volume}} with the origin in the boundary}
\label{sec:0inboundary}
%\texorpdfstring{\cref{conj:volume}}{Conjecture C} when the origin is in the boundary}
%
As we said in the introduction, the question analogous to \Cref{conj:one} for general lattice polytopes has an easy answer: the maximum covering radius among all $d$-dimensional lattice polytopes equals~$d$ and is attained by, and only by, unimodular simplices. This phenomenon generalizes to analogs of \cref{thm:lambda} and \cref{conj:volume}, which admit easy proofs. The generalization concerns the simplices $S(\a)$, except we now allow one of the entries of $\a$ (typically the first one) to be zero so that the origin becomes a vertex:

%\paco{notation: changed $S_\zero(\a)$ to $S(0,\a)$?}
%\matthias{I agree, this is more consistent.}
\begin{proposition}
\label{prop:corner}
For an $\a \in \R^{d}_{>0}$ let 
\[
S(0,\a):=\conv(\left\{\zero, \a_1 e_1, \ldots, \a_d e_d\right\}). 
\]
Then
\[
\mu(S(0,\a) )
=
 {\sum_{i=1}^d \frac{1}{\a_i}}
=
 \frac{\sum_{i=1}^d \Vol_{\pi_i(\Z^d)}(\pi_i(S(\a)))} {\Vol_{\Z^d}(S(\a))},
\]
where $\pi_i: \R^d \to \R^{d-1}$ is the linear projection that forgets the $i$-th coordinate.
\end{proposition}

\begin{proof}
$S(0,\a) $ can be redescribed as 
\[
\left\{x\in \R_{\ge0}^d : \sum_{i=1}^d \frac{x_i}{\a_i} \le 1\right\}.
\]
In this form it is clear that $\mu(S(0,\a))$ equals the unique $\mu\in [0,\infty)$ such that~$\one_d$ lies in the boundary of $\mu\cdot S(0,\a)$, which equals $\sum_i \frac1{\a_i}$, as stated.
\end{proof}

\begin{corollary}
\label{coro:corner}
Let $S = \conv(\{\zero,v_1,\ldots,v_d\}) \subseteq \R^d$ be a $d$-simplex with rational vertex directions.
For each $i=1,\ldots,d$, let $\pi_i : \R^d \to \R^{d-1}$ be the linear projection vanishing at~$v_i$.
Then,
\[
\mu(S) \leq \frac{\sum_{i=1}^d \Vol_{\pi_i(\Z^d)}(\pi_i(S))}
{\Vol_{\Z^d}(S)},
\]
with equality if and only if $S$ is unimodularly equivalent (by a transformation fixing the origin) to $S(0,\a)$ for some $\a \in \R_{> 0}^d$.
\end{corollary}

\begin{proof}
Let $p_1,\ldots,p_d \in \Z^d$ be the primitive vertex directions of~$S$, so that $v_i = \a_i p_i$, where $\a_i$ is the lattice length of the segment $[\zero,v_i]$, for each $i=1,\ldots,d$. Then, the linear map sending $p_i \mapsto e_i$, $i=1,\ldots,d$, sends $S$ to $S(0,\a)$ and $\Z^d$ to a lattice $\Lambda$ containing $\Z^d$. This implies
\[
\mu(S, \Z^d)= \mu(S(0,\a), \Lambda) \leq \mu(S(0,\a),\Z^d) = \frac{\sum_{i=1}^d \Vol_{\pi_i(\Z^d)}(\pi_i(S))}
{\Vol_{\Z^d}(S)},
\]
by \Cref{prop:corner}.

The `if' in the equality case is obvious: in this case $\Lambda = \Z^d$.
For the `only if' suppose that  $\Lambda$ is a proper superlattice of $\Z^d$ and let $p \in \Lambda \cap [0,1)^d\setminus\{\zero\}$ be a non-zero lattice point in the half-open unit cube. Let $\mu=\mu(S(0,\a), \Z^d) = \frac{\sum_{i=1}^d \Vol_{\pi_i(\Z^d)}(\pi_i(S))} {\Vol_{\Z^d}(S)}$. Then,
the point $\one$ is the only point in the unit cube $[0,1]^d$ that is last covered by $\Z^d + \mu \cdot S(0,\a)$. Since $\one$ lies in the interior of $p + \mu \cdot S(0,\a)$, the covering radius of 
$S(0,\a)$ is strictly smaller with respect to $\Lambda$  than it is with respect to $\Z^d$.
\end{proof}

%Observe that \cref{conj:volume} is to \cref{conj:one} as this result is to  the fact that unimodular simplices maximize the covering radius among all lattice simplices, since a unimodular simplex is $S{(0,\one_{d})}$.

Our next results say that \cref{prop:corner} and \cref{coro:corner} are not only analogs (without the factor of two) of \cref{thm:lambda} and \cref{conj:volume}, but also a limit of them when we make one of the vertices tend to zero. We consider this as additional evidence for \Cref{conj:volume}. Formally:

\begin{theorem}
\label{thm:limit}
Let $S= \conv(\{v_0,\ldots,v_d\})$ be a $d$-simplex with the origin in its interior and with rational vertex directions. For each $i \in \{0,\ldots,d\}$ consider the one-parameter family of simplices
\[
S^{(i)}_t:= \conv(\{v_0,\ldots,tv_i,\ldots,v_d\}), \quad t \in [0,1],
\]
so that $S^{(i)}_1=S$ and $S^{(i)}_0 = \conv(\{v_1,\ldots,\zero,\ldots,v_d\})$.
For each $i=0,\dots,d$ let $\pi_i : \R^d \to \R^{d-1}$ be the linear projection vanishing at~$v_i$. 

Then, there is an index $j \in \{0,\ldots,d\}$ such that
\begin{equation}
\label{eq:limit}
\lim_{t\to 0} \frac12 \frac{\sum_{i=0}^d \Vol_{\pi_i(\Z^d)}(\pi_i(S^{(j)}_t))} {\Vol_{\Z^d}(S^{(j)}_t)}
\ge
\frac{\sum_{i=0, i \neq j}^d \Vol_{\pi_i(\Z^d)}(\pi_i(S^{(j)}_0))} {\Vol_{\Z^d}(S^{(j)}_0)},
\end{equation}
with equality if and only if the primitive lattice vectors parallel to $v_0,\ldots,v_d$ add up to zero.
\end{theorem}

Observe that the condition for equality includes, but is more general than, the case  when $S$ is of the form $S(\a)$.

%
%\paco{this result is actually an immediate corollary of \cref{thm:lambda} and \cref{prop:corner}. The reason to include it is to refer to its proof in \cref{rem:limit}...}
%\begin{proposition}
%\label{prop:limit}
%Let $\a\in \R_{>0}^{d+1}$ and let $S_t:= \conv(-t\one_{d},\a_1e_1,\dots,\a_d e_d)$ for each  $t\in[0,1]$, so that $S_1=S(\a)$ and $S_0 =S(0,\a_1,\dots,\a_d)$.
%For each $i=0,\dots,d$, let $\pi_i : \R^d \to \R^{d-1}$ be the linear projection vanishing on~$e_i$, with $e_0=-\one_d$. Then:
%\begin{equation}
%\label{eq:limit}
%\lim_{t\to 0} \frac12 \frac{\sum_{i=0}^d \Vol_{\pi_i(\Z^d)}(\pi_i(S_t))} {\Vol_{\Z^d}(S_t)}
%=
%\frac{\sum_{i=1}^d \Vol_{\pi_i(\Z^d)}(\pi_i(S_0))} {\Vol_{\Z^d}(S_0)},
%\end{equation}
%%
%%\paco{don't think this is also an `only if'}
%%When $S$ is of the form $S(\a)$, there is equality in the formula.
%\end{proposition}

\begin{proof}
For each $i$, let $u_i$ be the primitive lattice vector parallel to $v_i$, and let $U=\{u_0,\ldots,u_d\}$.
We choose $j$ to be an index minimizing the (absolute value of the) determinant of $U \setminus \{u_i\}$ among all $i$. Observe that $S$ is of the form $S(\a)$ if and only if all those determinants are equal to $1$.

To simplify notation, in the rest of the proof we assume $j=0$ and we drop the superindex from the notation $S^{(j)}_t$.

Since the volume functional is continuous, we have
\[
\lim_{t\to 0}  {\Vol_{\Z^d}(S_t)} = {\Vol_{\Z^d}(S_0)},
\]
and, for each $i=1,\dots,d$, 
\[
\lim_{t\to 0}  \Vol_{\pi_i(\Z^d)}(\pi_i(S_t)) = \Vol_{\pi_i(\Z^d)}(\pi_i(S_0)).
\]
Thus, the only thing to prove is that
\[
\lim_{t\to 0} \Vol_{\pi_0(\Z^d)}(\pi_0(S_t))
\ge
 \sum_{i=1}^d \Vol_{\pi_i(\Z^d)}(\pi_i(S_0)).
\]
The volume on the left-hand side does not depend on~$t$ because the vertex of~$S_t$ that depends on~$t$ is projected out by~$\pi_0$.
Moreover, this volume equals $\sum_{i=1}^d\Vol_{\pi_0(\Z^d)}(\pi_0(F_i))$, where $F_i$ is the facet of $S_0$ opposite to $v_i$. Similarly, 
$\Vol_{\pi_i(\Z^d)}(\pi_i(S_0)) = \Vol_{\pi_i(\Z^d)}(\pi_i(F_i))$. Hence, the inequality follows from
\begin{equation}
\label{eq:two_projections}
\Vol_{\pi_0(\Z^d)}(\pi_0(F_i)) \ge \Vol_{\pi_i(\Z^d)}(\pi_i(F_i)).
\end{equation}
Both sides of \cref{eq:two_projections} are integer multiples of  $\Vol_{\Z^d \cap \aff(F_i)}(F_i)$, with the proportionality factors being 
the lattice distances from $F_i$ to $u_0$ and to~$u_i$, respectively. These distances are proportional to the determinants of $U\setminus \{u_i\}$ and $U\setminus \{u_0\}$, so our assumption on $u_0$ minimizing this implies the statement.
Moreover, we have equality  if, and only if, all the determinants of $U\setminus \{u_i\}$ are equal to that of  $U\setminus \{u_0\}$. This in turn is equivalent to $\sum_{i=0}^d u_i=\zero$. 
\end{proof}

\begin{corollary}
\label{coro:limit}
In the conditions of \cref{thm:limit} and for the index~$j$ mentioned therein, we have
\[
\lim_{t\to 0} \mu(S^{(j)}_t) \le
\lim_{t\to 0} \frac12 \frac{\sum_{i=0}^d \Vol_{\pi_i(\Z^d)}(\pi_i(S^{(j)}_t))} {\Vol_{\Z^d}(S^{(j)}_t)},
\]
with equality if and only if the primitive lattice vectors parallel to $v_0,\dots,v_d$ add up to zero.
\end{corollary}

\begin{proof}
This follows from \cref{thm:limit} since
\[
\lim_{t\to 0} \mu(S^{(j)}_t) = \mu(S^{(j)}_0) \le \frac{\sum_{i=0,i \neq j}^d \Vol_{\pi_i(\Z^d)}(\pi_i(S^{(j)}_0))} {\Vol_{\Z^d}(S^{(j)}_0)},
\]
where the last inequality is \cref{coro:corner}.
\end{proof}

\begin{remark}
\label{rem:limit}
\cref{eq:limit} is not true for all choices of $j$. Without any assumption on $j$
the proof of \cref{thm:limit} carries through up to the point where we say that \cref{eq:limit} would follow from \cref{eq:two_projections}, but the latter inequality is not  true in general. For a specific example, let $S=\conv(\{(0,-1), (1,1), (-1,1)\})$ and consider $j=0$. Then, for $i=1,2$,
\[ 
\Vol_{\pi_0(\Z^d)}(\pi_0(F_i)) = 1 < 2 =  \Vol_{\pi_i(\Z^d)}(\pi_i(F_i)).
\]
This gives
\[
\lim_{t\to 0} \frac12 \frac{\sum_{i=0}^d \Vol_{\pi_i(\Z^d)}(\pi_i(S^{(0)}_t))} {\Vol_{\Z^d}(S^{(0)}_t)} = \frac12\cdot\frac{2+2+2}{2} = \frac32,
\]
and
\[
\frac{\sum_{i=1}^d \Vol_{\pi_i(\Z^d)}(\pi_i(S^{(0)}_0))} {\Vol_{\Z^d}(S^{(0)}_0)} = \frac{2+2}{2} = 2.
\]
\end{remark}

We finally look at the intermediate case where $\zero$ is in the boundary of $S= \conv(\{v_0,\ldots,v_d\})$ but not a vertex. 
%For each $i=0,\dots,d$ we can still consider the projection $\pi_i$ vanishing on the $i$-th vertex and call
%\[
%\ell_i =  \frac{\Vol_{\Z^d}(S)}{\Vol_{\pi_i(\Z^d)}(\pi_i(S))},
%\]
%which is the lattice length of the segment $S\cap \aff(\zerov_i)$.
%
We can generalize \cref{conj:volume} to

\begin{introconjecture}
\label{conj:volume2}
Let $S = \conv(\{v_0,\ldots,v_d\})$ be a $d$-simplex with $\zero\in S\setminus \{v_0,\ldots,v_d\}$, and with rational vertex directions.
Let $\pi_i:\R^d \to \R^{d-1}$ be the linear projection vanishing at~$v_i$. Let $I\subset\{0,\dots,d\}$ be the set of labels of facets of $S$ containing $\zero$. 
Then
\begin{align}
\mu(S) \leq \frac12 \frac{\sum_{i=0}^d \Vol_{\pi_i(\Z^d)} (\pi_i(S)) + \sum_{i\in I} \Vol_{\pi_i(\Z^d)} (\pi_i(S))}
{\Vol_{\Z^d}(S)}.
\label{eq:volume2}
\end{align}
\end{introconjecture}

\begin{proposition}
\label{prop:volume2}
\cref{conj:volume2} $\Longleftrightarrow$ \cref{conj:volume}.
\end{proposition}

\begin{proof}
The implication \cref{conj:volume2} $\Longrightarrow$ \cref{conj:volume} is obvious, since the latter is the case $I=\emptyset$ of the former.

For the other implication, for each $i=0,\ldots,d$, let
\[
\ell_i  =\frac{\Vol_{\Z^d}(S)}{\Vol_{\pi_i(\Z^d)} (\pi_i(S))},
\]
which equals the lattice length of the segment $S \cap \lin(\{v_i\})$.
The inequality in \cref{conj:volume2} we want to prove becomes
\[
\mu(S) \leq \frac12 \sum_{i\not\in I} \frac1{\ell_i} + \sum_{i\in I} \frac1{\ell_i}.
\]
Let $S_I= \conv(\{v_i : i \not\in I\})$, and  $S_{\overline I}=\conv(\{\zero\} \cup \{v_i: i \in I\})$. Observe that 
$S_I$ equals the intersection of the facets of $S$ containing $\zero$, hence it is a $(d-|I|)$-simplex with $\zero$ in its relative interior. $S_{\overline I}$ is an $|I|$-simplex with~$\zero$ as a vertex. Hence, \cref{conj:volume} and \cref{prop:corner} respectively say:
\[
\mu(S_I) \le \frac12 \sum_{i \not\in I} \frac1{\ell_i} \qquad\text{and}\qquad
\mu(S_{\overline I}) \le \sum_{i \in I} \frac1{\ell_i}.
\]
Consider the linear projection $\pi_I: \R^d \to \R^{I}$ vanishing on $S_I$. By \cref{lem:projection}
\[
\mu(S) \le \mu(S_I) + \mu (\pi_I(S)),
\]
so it only remains to show that 
\[
\mu (\pi_I(S)) \le \mu(S_{\overline I}).
\]
This holds because $\pi_I$ is an affine bijection from $S_{\overline I}$ to $\pi_I(S)$, so that $\pi_I(S)$ can be considered to be the same as  $S_{\overline I}$ except regarded with respect to a (perhaps) finer lattice.
\end{proof}

\section{Covering minima of the simplex \texorpdfstring{$S(\a)$}{S(w)}}
\label{sec:lambda}

\subsection{The covering radius of \texorpdfstring{$S(\a)$}{S(w)}}\label{sec:lambda_proof}

We here prove \cref{thm:lambda} and thus compute the covering radius of $S(\a)=\conv(\left\{-\a_0 \one_d,\a_1 e_1,\ldots,\a_d e_d\right\})$.

\begin{proof}[Proof of \cref{thm:lambda}]
The simplex $S(\a)$ can be triangulated into the $d+1$ simplices
\[
S_i = \conv(\left\{\zero,\a_0 e_0, \a_1 e_1, \ldots, \a_d e_d\right\} \setminus \{\a_i e_i\}), \quad 0 \leq i \leq d,
\]
where $e_0 = -\one_d$.
Writing $[d]_0 := \{0,1,\ldots,d\}$, we define 
\[
\mathring{P}_i = \bigg\{ \sum_{j \in [d]_0 \setminus \{i\}} \alpha_j e_j : 0 \leq \alpha_j < 1 \bigg\}
\]
the half-open parallelotope spanned by the primitive edge directions of~$S_i$ incident to the origin. 
Let $i \in [d]_0$ be fixed. 
Then, for any $x \in \R^d$ there is a lattice point $v_i \in \Z^d$ such that $x \in v_i + \lambda S_i$ and the dilation factor $\lambda \geq 0$ is the smallest possible.
Let $L_i(x)$ be the set of all such lattice points~$v_i$.
For a fixed $v \in \Z^d$, we define
\[
R_i(v) = \left\{x \in \R^d : v \in L_i(x)\right\}
\]
to be the region of points that are associated to~$v$ in this way.

Explicitly these regions are translates of the $\mathring{P}_i$, more precisely we claim that $R_i(v)=v + \mathring{P}_i$, for all $i \in [d_0]$.

Indeed, let $x \in R_i(v)$, and let $\lambda \geq 0$ be smallest possible such that $x \in v + \lambda S_i$.
By the definition of $S_i$, we can write $x-v = \sum_{j \in [d]_0 \setminus \{i\}} \alpha_j e_j$, for some $\alpha_j \geq 0$.
If there would be an index~$j$ such that $\alpha_j \geq 1$, then $x \in v + e_j + \lambda S_i$ and the intersection of this simplex and $v + \lambda S_i$ is a smaller homothetic copy of $S_i$ containing~$x$.
Thus, $\lambda$ is not minimal and this contradiction implies that $x \in v + \mathring{P}_i$.
Conversely, if $x-v = \sum_{j \in [d]_0 \setminus \{i\}} \alpha_j e_j \in \mathring{P}_i$, and $\lambda \geq 0$ is minimal such that $x \in v + \lambda S_i$, then $x-v$ lies in the facet of $\lambda S_i$ not containing the origin.
Since $0 \leq \alpha_j < 1$, for all $j \in [d]_0 \setminus \{i\}$, the scalar $\lambda$ is not only minimal for~$v$, but for any lattice point.
Hence, $v \in L_i(x)$.

%Indeed, let $x$ be a point in the parallelotope $v+ \sum_{\{j \neq i\}} [0,e_j)$, let $w$ be a lattice point other than $v$. Suppose that $x$ is in the simplex $w+ \lambda S_i$. Then the intersection of this simplex with $v+ \sum_{\{j \neq i\}} [0,e_j)$ is a simplex and it is similar to the first. Thus it can be written $v+ \lambda' S_i$, with $\lambda' < \lambda$.

With this observation, the regions $R_i(v)$ are seen to be induced by the arrangement of the hyperplanes $\{x_i = a \}, \{x_i -x_j = a\}$ for all $j \in [d]_0 \setminus \{i\}$ and $a \in \Z$, where we define $x_0 = 0$. We call this arrangement $A_d^i$.
Moreover, for a point~$x$ in the interior of $R_i(v)$, the associated lattice point is unique, and we call it~$v_i(x)$.

The smallest common refinement $\A_d$ of the arrangements $\A_d^0,\ldots,\A_d^d$ is known as the \emph{alcoved arrangement} (see~\cite[Ch.~7]{becksanyal2018combinatorial} for a detailed description).
The full-dimensional cells of $\A_d$, also called its \emph{chambers}, are lattice translations of the simplices
\[
C_\pi =\conv\left(\left\{\zero, e_{\pi(1)}, e_{\pi(1)}+e_{\pi(2)}, \ldots, e_{\pi(1)}+\ldots+e_{\pi(d)}\right\}\right),
\]
where $\pi$ is a permutation of $\{1,\ldots,d\}$.

Each chamber of $A_d$ is the intersection of regions $R_i(v)$.
More precisely,
\begin{align*}
\inter(C_\pi) &= R_0(\zero) \cap R_{\pi(1)}(e_{\pi(1)}) \cap  \ldots \cap R_{\pi(d)}(e_{\pi(1)} +\ldots+e_{\pi(d)}) \\
&= \mathring{P}_0 \cap (e_{\pi(1)} + \mathring{P}_{\pi(1)} )\cap  \ldots \cap (e_{\pi(1)} +\ldots+e_{\pi(d)} + \mathring{P}_{\pi(d)}).
\end{align*}
Therefore, the chambers $C_\pi$ are exactly those regions of points in $\R^d$ that, for each $i \in [d]_0$, are associated to the same lattice point, that is, $v_i(x)=v_i(y)$ for all $x, y \in \inter(C_\pi)$.

After these preparations, we are ready to compute the covering radius of~$S(\a)$.
Note that, since $[0,1]^d$ is a fundamental cell of~$\Z^d$, we only need to find the smallest dilation factor $\mu$ so that the lattice translates of $\mu S(\a)$ cover the unit cube.
Moreover, we may focus on what happens within one chamber~$C_\pi$, and by symmetry we assume that $\pi = \id$.
Among all points in $C_\id = \conv\left(\left\{\zero, e_1, e_1+e_2, \ldots, e_1+\ldots+e_d\right\}\right)$, we are looking for a point~$y$ which is last covered by dilations of $S_i + e_{[i]}$, for some $i \in [d]_0$, and the factor of dilation needed.
Here, we write $e_{[i]} = e_1+\ldots+e_i$.
If we let $\ell_i : \R^d \to \R$ be the linear functional which takes value $1$ on the facet $F_i$ of $S(\a)$ that is opposite to $\a_i e_i$, this is equivalent to 
\[
y = \argmax_{x \in C_\id} \min_{i \in [d]_0} |\ell_i(x-e_{[i]})|.
\]
The key observation is that $y$ is the point where all the values $|\ell_i(y-e_{[i]})|$, $0 \leq i \leq d$, are equal.
This is because $\ell_i(x-e_{[i]})$ is nonnegative for $x \in C_\id$ and because there is a positive linear dependence among the functionals $\ell_i$, so there cannot be a point $y'$ where they all achieve a larger value than at a point where they all achieve the same value.
Therefore, ~$y$ satisfies the conditions 
\begin{align*}
\ell_0(y) = \ell_i(y-e_{[i]}), \quad \text{for every }1 \leq i \leq d.
\end{align*}
The explicit expression of the functionals $\ell_i$ is
\begin{align*}
\ell_0(x) = \sum_{j=1}^d \a_j^{-1} x_j \quad \text{and} \quad \ell_i(x) = \sum_{j \in [d] \setminus \{i\}} \a_j^{-1}x_j -\left(\sum_{j \in [d]_0 \setminus \{i\}}\a_j^{-1} \right) x_i.
\end{align*}
Thus we need to solve the system of the following equations:
\[
\sum_{j=1}^d \a_j^{-1} y_j  = \sum_{j \in [d] \setminus \{i\}} \a_j^{-1}y_j -\left(\sum_{j \in [d]_0 \setminus \{i\}}\a_j^{-1} \right) y_i + \a_0^{-1} + \sum_{j>i} \a_j^{-1}, \quad 1 \leq i \leq d.
\]
This system is solved by $y=(y_1,\ldots,y_d)$ with
\[
y_i = \frac{\a_0^{-1} + \a_{i+1}^{-1} + \ldots +\a_d^{-1}}{\a_0^{-1} + \a_1^{-1} + \ldots +\a_d^{-1}}.
\]
The value that the functionals take at $y$ is by what we said above the covering radius of $S(\a)$, and it is given by
\[
\mu(S(\a)) = \ell_0(y) =  \frac{\sum_{0 \leq i < j \leq d} \a_i^{-1}\a_j^{-1}}{\sum_{i=0}^d \a_i^{-1}},
\]
as desired.
\end{proof}

\begin{corollary}
\label{cor:conjC_primitiveSimplices}
Let $S \subseteq \R^d$ be a simplex with the origin it its interior and with rational vertex directions.
If the primitive vertex directions $p_0,p_1,\ldots,p_d$ of $S$ satisfy $p_0+p_1+\ldots+p_d=\zero$, then \Cref{conj:volume} holds for~$S$.
\end{corollary}

\begin{proof}
The proof is basically given already in \Cref{coro:volume_dim2}.
Consider the lattice $\Lambda$ generated by $p_0,p_1,\ldots,p_d$, and let $A$ be the linear transformation sending $e_i$ to $p_i$, for $i=1,\ldots,d$.
Then, $\Lambda = A\Z^d$ and $S=A S(\a)$ for a suitable $\a \in \R^{d+1}_{>0}$.
Since the $p_i$s are primitive, the lattice lengths $\ell_i=\frac{\Vol_{\Z^d}(S)}{ \Vol_{\pi_i(\Z^d)}(\pi_i(S))}$ are the same for every pair $(S,\Z^d)$, $(S,\Lambda)$, and $(S(\a),\Z^d)$.
Using that $\Lambda \subseteq \Z^d$ is a sublattice, we therefore apply \Cref{thm:lambda} and get
\[
\mu(S) \leq \mu(S,\Lambda) = \mu(S(\a),\Z^d) = \frac12 \sum_{i=0}^d \frac{1}{\ell_i}. \qedhere
\]
\end{proof}

Observe that \cref{thm:lambda} says that \cref{eq:volume} in \cref{conj:volume}  is an equality for simplices of the form $S(\a)$. 
Other simplices may also produce an equality,  as the triangle~$T = S(\one_2) \oplus S'(\one_2)$ shows:
\[
\frac12 \frac{\sum_{i=0}^2 \Vol_{\pi_i(\Z^2)} (\pi_i(T))} {\Vol_{\Z^2}(T)} = 
\frac12 \cdot \frac{3 + 3 + 2}{4} = 1 = \mu(T).
\]

\subsection{The covering product conjecture}
\label{sec:covprod}

The following conjecture was proposed in \cite{gonzalezschymura2017ondensities}, which was the initial motivation to compute the covering minima of the simplex~$S(\one_{d+1})$.

\begin{introconjecture}[\protect{\cite[Conj.~4.8]{gonzalezschymura2017ondensities}}]
\label{conj:cov_prod}
For every convex body $K \subseteq \R^d$,
\[
\mu_1(K)\cdot\ldots\cdot\mu_d(K)\cdot\vol(K) \geq \frac{d+1}{2^d}.
\]
Equality is attained for the simplex $S(\one_{d+1})$.
\end{introconjecture}

\cref{conj:cov_prod} is known to hold for $d=2$ \cite{schnell1995a}.
We show it in arbitrary dimension for the simplices $S(\a)$.

\begin{corollary}
\label{cor:lambda_cov_prod}
For every $\a \in \R^{d+1}_{>0}$, we have
\[
\mu_1(S(\a)) \cdot \ldots \cdot \mu_d(S(\a)) \cdot \Vol_{\Z^d}(S(\a)) \geq \frac{(d+1)!}{2^d}.
\]
Equality can hold only if $\a_0 = \a_1 = \ldots = \a_d$.
\end{corollary}
\begin{proof}
Since every permutation of the vertices of $S(\one)$ is a unimodular transformation, and since the considered product functional is invariant under unimodular transformations, we can assume that $\a_0 \leq \a_1 \leq \ldots \leq \a_d$.
By \Cref{thm:lambda}, the covering radius of $S(\a)$ is given by
\[
\mu(S(\a)) = \frac{\sigma_{d-1}(\a_0,\a_1,\ldots,\a_d)}{\sigma_d(\a_0,\a_1,\ldots,\a_d)},
\]
where $\sigma_j(\a_0,\a_1,\ldots,\a_d) = \sum_{0\leq i_1 < \ldots < i_j \leq d} \prod_{\ell=1}^j \a_{i_\ell}$ is the \emph{$j$-th elementary symmetric function} in the $\a_i$'s.
Writing $\a_I = (\a_0,\a_{i_1},\ldots,\a_{i_j})$, for every index set $I = \{i_1,\ldots,i_j\} \subseteq \{1,\ldots,d\}$, $|I|=j$, we project onto the $j$-dimensional coordinate plane indexed by~$I$ and obtain $\mu_j(S(\a)) \geq \mu_j(S(\a_I))$.
In particular, choosing $I=\{1,\ldots,j\}$, we have
\begin{align}
\mu_j(S(\a)) \geq \frac{\sigma_{j-1}(\a_0,\a_1,\ldots,\a_j)}{\sigma_j(\a_0,\a_1,\ldots,\a_j)}.\label{eqn:covmins_lambda}
\end{align}
Next, in view of $\a_j \geq \a_{j-1} \geq \ldots \geq \a_0$, we get
\begin{align}
\frac{\sigma_{j-1}(\a_0,\ldots,\a_j)}{\sigma_{j-1}(\a_0,\ldots,\a_{j-1})} &= \frac{\sigma_{j-1}(\a_0,\ldots,\a_{j-1}) + \a_j \, \sigma_{j-2}(\a_0,\ldots,\a_{j-1})}{\sigma_{j-1}(\a_0,\ldots,\a_{j-1})} \nonumber \\
  &= 1 + \frac{\a_j \, \sigma_{j-2}(\a_0,\ldots,\a_{j-1})}{\sigma_{j-1}(\a_0,\ldots,\a_{j-1})} \geq 1+\frac{\binom{j}{2}}{j} = \frac{j+1}{2}, \label{eqn:covmins_lambda_2}
\end{align}
with strict inequality unless $\a_j = \a_{j-1} = \ldots = \a_0$.

Finally, computing the volumes of the pyramids over the $d+1$ facets of $S(\a)$ with apex at the origin, we obtain $\Vol_{\Z^d}(S(\a)) = \sigma_d(\a_0,\a_1,\ldots,\a_d)$.
Combining this with~\eqref{eqn:covmins_lambda} and~\eqref{eqn:covmins_lambda_2} yields
\begin{align*}
\mu_1(S(\a)) \cdot \ldots \cdot \mu_d(S(\a)) \cdot \Vol_{\Z^d}(S(\a)) &\geq \prod_{j=1}^d \frac{\sigma_{j-1}(\a_0,\ldots,\a_j)}{\sigma_j(\a_0,\ldots,\a_j)} \sigma_d(\a_0,\ldots,\a_d) \\
&= \prod_{j=1}^{d} \frac{\sigma_{j-1}(\a_0,\ldots,\a_{j})}{\sigma_{j-1}(\a_0,\ldots,\a_{j-1})} \geq \frac{(d+1)!}{2^d}.
\end{align*}
Furthermore, equality can only hold if $\a_0 = \a_1 = \ldots = \a_d$ as otherwise~\eqref{eqn:covmins_lambda_2} would be strict for $j=d$.
\end{proof}

Note that if \Cref{conj:minima} holds, then the simplex $S(\one_{d+1})$ attains equality in \Cref{cor:lambda_cov_prod} (this was the original motivation in~\cite{gonzalezschymura2017ondensities} to state \Cref{conj:minima}).

With the notation of the proof above, for each $I \subseteq\{1,\ldots,d\}$, $|I| = j$, we have $\mu_j(S(\a_I)) \leq \mu_j(S(\a_0,\a_1,\ldots,\a_j))$, just because $S(\a) \subseteq S(\bar\a)$, whenever $\a_i \leq \bar\a_i$, for all~$i$.
Therefore, the bound in~\eqref{eqn:covmins_lambda} is maximal among coordinate projections of $S(\a)$.
This suggests the following common generalization of \Cref{conj:minima} and \Cref{thm:lambda}.

\begin{conjecture}
\label{conj:mu_j_lambda}
For every $\a \in \R^{d+1}_{>0}$ with $\a_0 \leq \a_1 \leq \ldots \leq \a_d$, and every $j \in \{1,\ldots,d\}$, the $j$-th covering minimum of the simplex $S(\a)$ is attained by the projection to the first $j$ coordinates. That is:
\[
\mu_j(S(\a)) = \mu_j(S(\a_0,\dots, \a_j)) = \frac{\sigma_{j-1}(\a_0,\a_1,\ldots,\a_j)}{\sigma_j(\a_0,\a_1,\ldots,\a_j)}.
\]
\end{conjecture}

Besides the case $j=d$ (\Cref{thm:lambda}) also the case $j=1$ of \cref{conj:mu_j_lambda} holds.
Assuming that $\a_0 \leq \a_1 \leq \ldots \leq \a_d$, it states that $\mu_1(S(\a))=\frac{1}{\a_0+\a_1}$.
Since~\eqref{eqn:covmins_lambda} provides the lower bound, this is equivalent to
\[
\det(\Z^d | L_z) \leq \frac{\| S(\a) | L_z \|}{\a_0+\a_1},
\]
for all primitive $z \in \Z^d\setminus\{\zero\}$, where $L_z = \lin\{z\}$.
In view of $\det(\Z^d | L_z) = \|z\|^{-1}$ and $e_i | L_z = \frac{z_i}{\|z\|^2} z$, it follows from an elementary computation.

\section{\texorpdfstring{\cref{conj:k}}{Conjecture \ref{conj:k}}: Lattice polytopes with \texorpdfstring{$k$}{k} interior lattice points}
\label{sec:k_points}

This section is devoted to prove \cref{conj:k} in dimension two.
The conjectured maximum covering radius $\frac{d-1}2+\frac1{k+1}$ is attained, in arbitrary dimension, by the polytopes of the form
\[
[0,k+1] \oplus T_1\oplus \dots \oplus T_m,
\]
where each  $T_i$ is a non-hollow lattice $d_i$-polytope of covering radius $d_i/2$, with $\sum_{i=1}^m d_i=d-1$. The different $T_i$ can be translated to have their (unique) interior lattice point at different positions along the segment $[0,k+1]$ in much the same way as in the examples of \Cref{lemma:minimum_3d}. In the following we analyze the possibilities in dimensions two and three:

\begin{example}
\label{exm:dim2-k}
In dimension two we have a single $T_i$, the segment $[-1,1]$, but we can place it at different heights with respect to $[0,k+1]$.  For each~$k$ we can construct $\lfloor (k+3)/2 \rfloor$ non-isomorphic lattice polygons with $k$ interior lattice points and of covering radius $\frac12 + \frac1{k+1}$, namely:
\[
\conv(\{(0,0), (0,k), (-1,i), (1,i)\}), \quad i=0,\dots, \lfloor (k+1)/2 \rfloor.
\]
The case $i=0$ coincides with the triangle $M_k(0)$; the cases $i>0$ produce kite-shaped quadrilaterals.

Observe that the triangle $M_k(1)\cong S(k,1,1)$ is very similar to $M_k(0)$ but has  smaller area. One could expect it to achieve a larger covering radius but it does not, as computed in \cref{rem:M_k-dim2}:
\[
\mu(M_k(1)) = \frac{k+2}{2k+1} < \frac{k+3}{2k+2} =  \frac12 + \frac1{k+1}, \ \ \text{ if } k > 1.
\]
\end{example}

\begin{example}
\label{exm:dim3-k}
In dimension three we can have $[0,k+1]\oplus T$ with $\dim(T)=2$ or $[0,k+1]\oplus T_1\oplus T_2$ with $\dim(T_1)=\dim(T_2)=1$.

If the latter happens then $T_1=T_2=[-1,1]=I$ and, again, they can be placed at different heights along the segment $[0,k+1]$. Depending on whether  $T_1$ and $T_2$ intersect $[0,k+1]$ in the interior or at an end-point
this gives quadratically many octahedra or linearly many  triangular bipyramids, plus the square pyramid $[0,k+1] \oplus (I \oplus I)$ and the tetrahedron $M_k(0,0)$.

In the case $[0,k+1]\oplus T$, $T$ can be either $S(\one_3)$ or $I\oplus I'$; the case $T=I \oplus I$ being already covered above. This produces two tetrahedra $[0,k+1] \oplus S(\one_3)$ and $[0,k+1] \oplus I \oplus I'$, plus linearly many triangular bipyramids.

As happened in dimension two, the computations of \cref{rem:M_k-dim3} show that $M_k(1,0)$ and $M_k(1,1)$ have covering radius strictly smaller than  $1+\frac1{k+1}$, even if their volume is smaller than that of $M_k(0,0)$.
\end{example}

%\paco{MOSTLY REWROTE REST OF THIS SECTION. Dec 9, 2020}
%\matthias{I checked everything carefully and did some minor editing concerning typos, references and style.}
Since \cref{conj:one} holds in the plane (\cref{coro:ConjA_dim2}), 
to prove  \cref{conj:k} in dimension two it suffices to consider lattice polygons with at least two interior lattice points.
More precisely, we show:

\begin{theorem}
\label{thm:dim2-k}
Let $P$ be a non-hollow lattice polygon with $k \geq 2$ interior lattice points.
Then $\mu(P)\le \frac12+\frac1{k+1}$, with  equality if and only if $P$ is the direct sum of two lattice segments of lengths $2$ and $k+1$.
\end{theorem}
%

%Our proof is split up into five steps distinguishing cases with respect to the following parameters:
%A lattice polytope $P$ has \emph{(lattice) width} $\omega \in \N$ if there is an affine integer projection from~$P$ to the segment $[0,\omega]$ but not to $[0,\omega-1]$.
%Remember that the width is the reciprocal of the first covering minimum.
%The numbers $m$, $m'$, and $k$ will denote the maximum number of collinear lattice points, maximum number of collinear \emph{interior} lattice points, and the number of interior lattice points of~$P$, respectively.
%We proceed as follows:
%%
%\begin{itemize}
% \item[\hspace{10pt} {\bf Step 1:}] $(\omega = 2)$ in \cref{lemma:width2}
% \item[\hspace{10pt} {\bf Step 2:}] $(\omega \geq 3, m \geq 4)$ except for $(\omega=3, m=4, k\geq5)$ in \cref{lem:bigm}
% \item[\hspace{10pt} {\bf Step 3:}] $(\omega=3, m=4, k\geq5)$ in \cref{lem:34}
% \item[\hspace{10pt} {\bf Step 4:}] $(\omega \geq 3, m' \geq 3)$ in \cref{lem:3collinear}
% \item[\hspace{10pt} {\bf Step 5:}] $(\omega \geq 3, m' \leq 2)$
%\end{itemize}
%%

Remember that a lattice polytope $P$ has \emph{(lattice) width} $\omega \in \N$ if there is an affine integer projection from~$P$ to the segment $[0,\omega]$ but not to $[0,\omega-1]$. Equivalently, 
the width is the reciprocal of the first covering minimum. Every non-hollow lattice polytope has width at least two. Our next two lemmas deal with the case of width exactly two.

%It will turn out that equality in the bound of \cref{thm:dim2-k} can only occur in the first case, that is, when $P$ has width two.

\begin{lemma}
%\paco{split lemma in two parts. This is first part}
\label{lemma:width2-a}
For a non-hollow lattice polygon $P$ the following are equivalent:
\begin{enumerate}[i)]
\item $P$ has width equal to two.
\item The interior lattice points of $P$ are collinear.
\end{enumerate}
\end{lemma}

\begin{proof}
The fact that width two implies that all interior lattice points are collinear is straightforward to check. For the converse, without loss of generality assume that the $k$ interior lattice points of~$P$ are $(0,1), \dots, (0,k)$, with $k\ge 2$.
We claim that $P\subset[-1,1]\times \R$, which implies that $P$ has width two with respect to the first coordinate. Suppose to the contrary that $P$ has a lattice point $(x,y)$ with $|x|\ge 2$. Then the triangle with vertices $(0,1)$, $(0,2)$ and $(x,y)$ is not unimodular, which implies that it contains at least one lattice point other than its vertices, by Pick's formula (cf.~\cite[Ch.~1.4]{becksanyal2018combinatorial}). That point is necessarily in $\inter(P)$ and not on the line containing $(0,1)$ and $(0,2)$, a contradiction.
\end{proof}

\begin{lemma}
%\paco{... and this is second part}
\label{lemma:width2-b}
\Cref{thm:dim2-k} holds if $P$ has width two.
\end{lemma}

\begin{proof}
We keep the convention from the previous proof that the interior lattice points in $P$ are given by $(0,1), \dots, (0,k)$, which implies that $P\subset [-1,1]\times[0,k+1]$.
Let $S$ be the segment $P\cap(\{0\}\times \R)$, which contains all the interior lattice points.
Observe that one endpoint of $S$ is either $(0,0)$ or $(0,1/2)$ and the other is either $(0,k+1)$ or $(0,k+1/2)$.
We distinguish three cases, depending on whether none, one, or both of them are lattice points:
\begin{itemize}

\item If exactly one endpoint is a lattice point, then $P$ contains a copy of $M_k(1)$, whose covering radius is strictly smaller than $\frac12+\frac1{k+1}$ (see \cref{exm:dim2-k}).

\item If no endpoint is a lattice point, then $S=\{0\} \times [1/2, k+1/2]$, and $P$ is the convex hull of its two edges containing the endpoints of $S$. Without loss of generality we assume
\[
P= \conv(\{(-1,0), (1,1), (-1,a), (1, 1+b)\}),
\]
where $a$ and $b$ are nonnegative integers with $a+b=2k$. There are two possibilities: If $a=b=k$, then $P$ is a parallelogram of covering radius at most $1/2$, because $\frac12 P$ contains a fundamental domain of~$\Z^2$. If $a\ne b$, then one of them, say $a$, is at least $k+1$. In this case, $P$ contains the triangle $\conv(\{(-1,0), (-1,a), (1,1)\})$ whose covering radius is bounded by $1/2 + 1/a \le 1/2 + 1/(k+1)$. Since triangles are tight by \Cref{lemma:simplex_monotone}, equality can only hold when $P$ coincides with this triangle, implying $b=0$. But in that case $a=2k$ and $1/2 + 1/a < 1/2 + 1/(k+1)$, since $k\geq2$.

\item If both endpoints of $S$ are lattice points, then they are given by $(0,0)$ and $(0,k+1)$. Applying \cref{lem:projection} to the projection that forgets the second coordinate gives the upper bound: The fiber $S$ has length $k+1$ and the projection of $P$ has length $2$. For the case of equality, observe that if $P$ has lattice points $u\in \{-1\}\times \R$ and $v\in \{1\}\times \R$ such that the mid-point of $uv$ is integral, then $P$ contains (an affine image of) the direct sum of $[-1,1]$ and a segment of length $k+1$. Since that direct sum is tight by \Cref{lemma:sum_monotone},~$P$ either is given by this direct sum or it has strictly smaller covering radius. 

Thus, we can assume that $P$ does not have such points $u$ and $v$. This implies that $P$ has a single lattice point on each side of $S$. Without loss of generality we can assume 
\[
P= \conv(\{(0,0), (0,k+1), (-1,0), (1,a)\}),
\]
for an odd $a \in [1, 2k+1]$. We claim that the \emph{proof} of \cref{lem:projection} implies that $\mu(P)$ is \emph{strictly smaller} than $\lambda:=1/2 + 1/(k+1)$. Indeed, that proof is based on the fact that $\lambda P$ contains the following parallelogram $Q$, which is a fundamental domain for $\Z^2$:
\[
Q=\conv\left(\left\{
      \left(-\frac12,0\right),
      \left(-\frac12,1\right),
      \left(\frac12,\frac{a}2\right),
      \left(\frac12,1+\frac{a}2\right)
      \right\}\right).
\]
But we can argue that, moreover, the vertices of $Q$ are its only points not contained in the interior of $\lambda P$, and that each of these vertices is in the interior of some lattice translation of $\lambda P$ because the vertical offset of the left and right edges of $Q$ is not an integer. This implies $\lambda$ to be strictly larger than $\mu(P)$.\qedhere
\end{itemize}
\end{proof}

\noindent For the rest of the proof of \cref{thm:dim2-k}, we can assume $\omega\ge 3$.
Let $m$ be the maximum number of collinear lattice points in our polygon $P$. Applying \cref{lem:projection} to the projection along  the line containing those $m$ points gives:
\begin{align}
\mu(P) \le \frac1\omega + \frac1{m-1}.\label{eq:m_omega}
\end{align}
Another useful fact is that along the direction that attains the width $\omega$ there are $\omega-1$ parallel lines intersecting the interior of $P$, each of them contains at most $m$ lattice points, and with every lattice point of $P$ lying on one of those lines. Thus:
\begin{align}
k \le (\omega-1)m.\label{eq:bounded_k}
\end{align}

\noindent These two bounds are enough to show that:

\begin{lemma}
\label{lem:bigm}
\Cref{thm:dim2-k} holds if $m\ge 4$, except perhaps for $(\omega,m) = (3,4)$.
\end{lemma}

\begin{proof}
By \cref{eq:m_omega}, the statement is trivial unless
\[
\frac1\omega + \frac1{m-1} >\frac12.
\]
There are five  integer solutions of this equation with $\omega \geq 3$ and $m-1\ge 3$:
\[(\omega,m-1)\in \{(3,3), (3,4), (4,3), (3,5), (5,3)\}.\] 

\noindent We only need to look at the last four:
\begin{itemize}

\item If $(\omega,m-1)\in \{(3,5), (5,3)\}$ then \cref{eq:m_omega} gives $\mu(P) \le 1/3+1/5 = 1/2 + 1/30$. This is smaller than $1/2 + 1/(k+1)$, because \cref{eq:bounded_k} gives, respectively, $k\le 12$ and $k\le 16$.

\item If $(\omega,m-1)\in \{(3,4), (4,3)\}$ then  $\mu(P) \le 1/3+1/4 = 1/2 + 1/12$. For $(3,4)$ this is enough since  \cref{eq:bounded_k} gives $k\le 10$. For $(4,3)$, however,  \cref{eq:bounded_k} gives $k\le 12$, so we still need to consider the cases $k=11$ or $12$. For these we use the following argument:
$\omega=4$ implies that, along the direction where $\omega$ is attained, we have three intermediate lattice lines intersecting $P$. Along these lines we have to place our $k\ge 11$ points, and no more than $4$ on each line (because $m=4$). Thus, each line gets at least three points. This makes $P$ contain a parallelogram $Q$ with 
two parallel edges of lattice length two and of width two with respect to the direction of those edges. We have that $Q$ is a fundamental domain of~$(2\Z)^2$, which  implies $\mu(P) \le \mu(Q) \le 1/2$. 
\qedhere
\end{itemize}
\end{proof}

%If $m=3$ then our polygon  contains at most nine lattice points in total, since it cannot have two points in the same residue class modulo $(3\Z)^2$. In particular, we have $k\le 6$. 

\noindent Thus, the cases that remain are $m\le 3$ or $(\omega, m) = (3,4)$. These can be proven with a case study that we only sketch here.
The details can be found in \cref{sect:app-details-section-6} that is not contained in the published version~\cite{published} of this article.
The case study goes as follows:
%\paco{added this ``sketch of proof'' of the remaining cases}
\begin{itemize}
\item For the case $(\omega,m) = (3,4)$, in \cref{lem:34} we show that one of the following three things happen:
\begin{itemize}
\item $k < 5$, in which case $\mu(P) \le 1/3 + 1/3 < 1/2 + 1/(k+1)$.
\item $P$ contains a fundamental domain $Q$ of~$(2\Z)^2$. As in the last paragraph of the previous proof, this implies $\mu(P) \le \mu(Q) \le 1/2$.
\item $P$ has four collinear lattice points along one of the two intermediate lines in the direction attaining the width, and (at least) three of them are interior to $P$. In this case the intersection of $P$ with that line has length at least  $3+1/3=10/3$, so  \cref{eq:m_omega} can be strengthened to
\[
\mu(P) \le \frac13 + \frac3{10} = \frac{19}{30}  < \frac12 + \frac17.
\]
This gives the statement if $k \in \{5,6\}$.
In the case $k \geq 7$ we must have four collinear lattice points in one of the two intermediate lines, so that we can further improve \cref{eq:m_omega} using $11/3$ for the length. Then: 
\[
\mu(P) \le \frac13 + \frac3{11} = \frac{20}{33}  < \frac12 + \frac19.
\]
This is enough since $(\omega,m) = (3,4)$ implies $k\le 8$, by \cref{eq:bounded_k}.
\end{itemize}

\item The case $m\le 2$ is trivial: it implies that $P$ does not have two lattice points in the same class modulo $2\Z\times 2\Z$, so it has at most four lattice points. The only non-hollow lattice polygon with at most four lattice points is~$S(\one_3)$.

\item For the case $m = 3$, in \cref{lem:3collinear} we show that  $\omega\ge 3$ and $m=3$ imply that $P$ cannot have three collinear interior lattice points. Since the interior lattice points cannot all be collinear (by \cref{lemma:width2-a}), they must form either a unimodular triangle, a unit parallelogram, or $S(\one_3)$. Thus, $P$ is contained in one of the three polygons of~\cref{fig:m=3}. From there, ad-hoc arguments show that always $\mu(P) < 1/2 + 1/(k+1)$, see \cref{lem:ad-hoc}.

\begin{figure}[ht]
\includegraphics[scale=0.5]{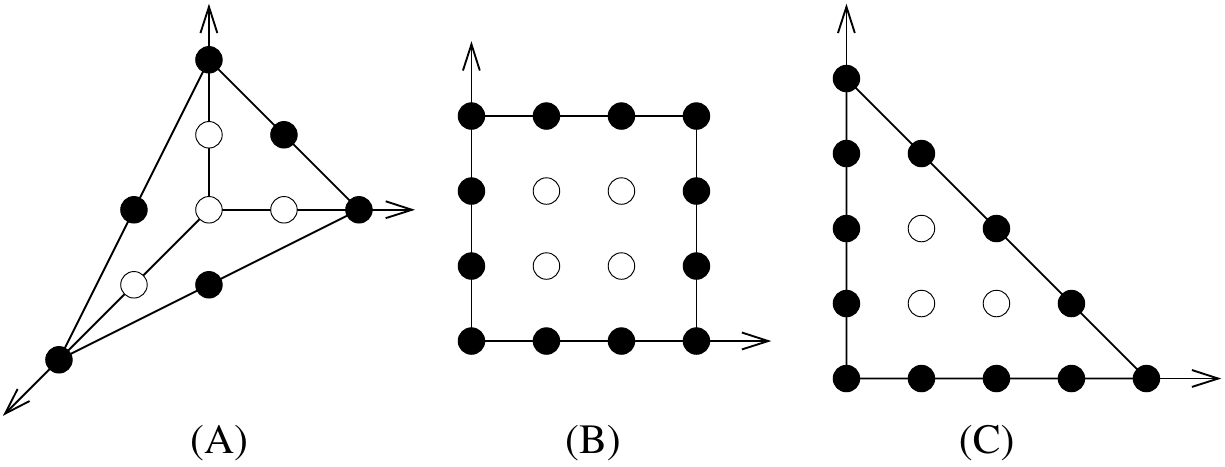}
\caption{The three possibilities in the case $m=3$.}
\label{fig:m=3}
\end{figure}
\end{itemize}

\begin{remark}
\label{rem:castryck}
Lattice polygons with $m \leq 3$ contain at most nine lattice points in total, since they cannot have two points in the same residue class modulo $(3\Z)^2$. In particular, they have $k\le 6$. On the other hand, the polytopes with $(\omega,m)=(3,4)$ have $k \leq 8$ by \cref{eq:bounded_k}. 
Thus, the cases not covered by \cref{lemma:width2-b} and \cref{lem:bigm} have between $3$ and $8$ interior lattice points.
Castryck~\cite{castryck2012movingout} enumerated all lattice polygons with $k\le 30$ up to unimodular equivalence, and showed that there are $120 + 211 + 403 + 714 + 1023 + 1830$ of them with $k$ equal to $3$, $4$, $5$, $6$, $7$, and $8$. Hence, the arguments sketched above can be replaced by a computer-aided computation of the covering radius of these $4301$ polygons. (In fact, the covering radius needs only to be computed for those with $m\in\{3,4\}$).
\end{remark}

%{\color{blue}
%\Paco{The following is a shorter alternative to  \cite[Lemma 6.7]{arXiv}, but it implies enumerating up to 12 points rather than 9:}
%
%\begin{lemma}
%\label{lem:bigk}
%If $P$ has width at least three and at least $13$ lattice points then $P$ satisfies \cref{conj:k} with strict inequality.
%\end{lemma}
%
%\begin{proof}
%Among the $13$ lattice points in $P$ there must be four in the same class modulo $(2\Z)^2$. If three of those are collinear then $P$ contains five collinear lattice points and we can apply \cref{lem:bigm}. So, we can assume that $P$ contains four points from the lattice 
%$(2\Z)^2$, no three of them collinear. That implies $P$ to contain a copy of either $[0,2]^2$ or $2S(\one_3)$, both of which have covering radius equal to $\frac12$. Thus, $\mu(P) \le \frac12$.
%\end{proof}
%}

\appendix

\section{The covering radius via a mixed-integer program}
%\section{An algorithm for the covering radius of a polytope}
%Second proof of \texorpdfstring{\Cref{thm:26minimal}}{Theorem \ref{thm:26minimal}}: computer calculations}
\label{sec:26computations}

Here we describe an algorithmic approach to the computation of covering radii based on a formulation of $\mu(P)$ as the optimal value of a mixed-integer program.
This formulation is already implicit in Kannan's paper~\cite[Sect.~5]{kannan1992latticetranslates}.

Let $P = \{x \in \R^d : a_i^\intercal x \leq b_i, 1 \leq i \leq m\}$ be a polytope with outer facet normals $a_i \in \R^d$ and right hand sides $b_i \in \R$.
Without loss of generality, we assume that $b_i > 0$, that is, $P$ contains the origin in its interior.
Since~$P$ is bounded, there exists a finite subset $N_P \subseteq \Z^d$ such that $\mu(P) P + N_P$ contains the unit cube $[0,1]^d$.

\begin{proposition}
\label{prop:MIP}
The covering radius $\mu(P)$ is equal to the optimal value of the following linear mixed-integer program:
\begin{equation*}
\begin{array}{rrcll}
\displaystyle \mathrm{maximize}  & \mu \\
\mathrm{s.t.} & a_i^\intercal x & \geq & \mu b_i + a_i^\intercal \ell - M(1-y_i^\ell), & \forall i=1,\ldots,m, \, \forall \ell \in N_P \\
& \sum_{i=1}^{m} y_i^\ell & \geq & 1, & \forall \ell \in N_P \\
& y_i^\ell & \in & \{0,1\}, & \forall i=1,\ldots,m, \, \forall \ell \in N_P \\
& x & \in & [0,1]^d . & \\
\end{array}
\end{equation*}
The constant $M > 0$ is chosen large enough such that every non-active inequality involving $M$ is redundant.
\end{proposition}
\begin{proof}
By the periodicity of the arrangement $\mu P + \Z^d$, we get that
\[
\mu(P) = \min\{\mu \geq 0 : [0,1]^d \subseteq \mu P + N_P\}.
\]
Hence, the covering radius equals the minimal $\mu \geq 0$ such that for all $x \in [0,1]^d$ there exists an $\ell \in N_P$ such that $x \in \mu P + \ell$.
This gives a mixed-integer program with infinitely many constraints.
In order to turn it into a finite program, we may also interpret the covering radius as the supremum among $\mu \geq 0$ such that there exists an $x \in [0,1]^d$ such that $x \notin \mu P + N_P$.

Modeling this non-containment condition can be done as follows:
For a fixed $\ell \in N_P$, we have $x \notin \mu P + \ell$ if and only if there exists a defining inequality of $P$ that is violated, that is, there exists an $i \in \{1,\ldots,m\}$ such that $a_i^\intercal x > \mu b_i + a_i^\intercal \ell$.
Introducing the binary variable $y_i^\ell$ for each $1 \leq i \leq m$ and each $\ell \in N_P$, and using a large enough constant $M > 0$, this is modeled by the first two lines in the program, as the condition $\sum_{i=1}^{m} y_i^\ell \geq 1$ ensures that at least one inequality is violated for~$\ell$.

We can replace the supremum by a maximum and the strict inequality $a_i^\intercal x > \mu b_i + a_i^\intercal \ell$ by a non-strict one, since $P$ is compact and the covering radius is in fact an attained maximum.
\end{proof}

In order to make this formulation effective, we need to find a suitable finite subset $N_P \subseteq \Z^d$:
A point $x \in [0,1]^d$ is contained in $z + \mu(P)P$, for some $z \in \Z^d$, if and only if $z \in [0,1]^d - \mu(P)P$.
Hence, for any theoretically proven upper bound $\mu(P) \leq \mu$, we can solve the mixed-integer program in \Cref{prop:MIP} with respect to $N_P = ([0,1]^d - \mu P) \cap \Z^d$ and obtain the covering radius of~$P$.

\section{Graphical method for covering radii of simplices}
%\section{A graphical method for the covering radius of a lattice simplex}
\label{sec:graph-method}

Let $T=\conv(\{v_0,v_1,\ldots,v_d\})$ be a lattice simplex of normalized volume $V=\Vol_\Lambda(T)$ with respect to a certain lattice $\Lambda$. The affine map defined by $v_0 \mapsto \zero$ and $v_i \mapsto V \cdot e_i$, $i=1,\dots,d$, sends $T$ to the dilated standard simplex $V \cdot\conv(\{\zero,e_1,\dots,e_d\})$ and $\Lambda$ to an intermediate lattice between~$V\Z^d$ and~$\Z^d$, which we still denote by $\Lambda$. Observe that $\Lambda_T:= \Lambda/ V\Z^d$ is a subgroup of $\Z^d/ V\Z^d = (\Z_V)^d$ of order $V$ and that
\[
\Z^d / \Lambda = (\Z_V)^d / (\Lambda/ V\Z^d)
\]
is, hence, a finite abelian group of order $V^{d-1}$.
%\paco{changed graph to digraph throughout.}
The \emph{Cayley digraph} $G$ associated with the quotient group $\Z^d / \Lambda$ is defined as the directed graph with vertex set $\Z^d / \Lambda$ and edges $(x+\Lambda,x+e_i+\Lambda)$, for $x \in \Z^d$ and $1 \leq i \leq d$.
The following is a particular case of~\cite[Lem.~3 \& 4]{marklofstrom2013diameters} (cf.~also \cite[Thm.~4.11]{gonzalezschymura2017ondensities}):

\begin{lemma}
\label{lem:graph_method}
%\paco{shortened statement}
In these conditions, 
%let $G$ be the Cayley digraph of the quotient group $\Z^d / \Lambda$ and 
let $\delta(G)$ be the (directed) diameter of~$G$. That is, $\delta(G)$ is the maximum distance from $\zero$ to any other node of $G$.
Then, 
\[
\mu(T) = \frac{\delta(G)+d}V.
\]
\end{lemma}

\begin{proof}
The covering radius of the standard $d$-simplex $\conv(\{\zero,e_1,\dots,e_d\})$ with respect to the sublattice $\Lambda$ of $\Z^d$ equals $\delta(G)+d$. 
(This is the case $v=(1,\ldots,1)$ of \cite[Thm.~4.11]{gonzalezschymura2017ondensities}).
We divide this by $V$ since we are looking at the $V$th dilation of the standard simplex.
\end{proof}

%\paco{I have edited from here, putting emphasis on the form $T(a_0,\dots,a_d)$ rather than cyclicity. It is more to the point (and a bit shorter, although not really much)}
%In what follows we apply this lemma to the following class of simplices:
%Let $T=\conv(v_0,\dots,v_d)$ be a lattice simplex of normalized volume $V$ and let $p\in T$ be a lattice point in it.
%\matthias{Reformulated this paragraph a bit in order to avoid repetition from the first paragraph of Appendix B.}
Let $p \in T$ be a lattice point  in the lattice simplex $T=\conv(\{v_0,\dots,v_d\})$ from above.
For $i\in \{0,\dots,d\}$, let $a_i$ be the normalized volume of the pyramid with apex~$p$ over the facet of~$T$ opposite to $v_i$. 
Observe that the normalized volume of $T$ is given by $V=\sum_i a_i$. In fact, $\frac1V(a_0,\dots,a_d)$ is the vector of barycentric coordinates of $p$ in $T$.

\begin{lemma}
\label{lemma:cyclic}
If $\gcd(a_0, a_1, \dots, a_d) = 1$ then $T$ is equivalent to the simplex $V \cdot\conv(\{\zero,e_1,\dots,e_d\})$ with respect to the lattice $V\Z^d + (a_1,\dots,a_d)\Z$.

In particular, the graph $G$ of \Cref{lem:graph_method} equals the Cayley digraph  of $(\Z_V)^d / \langle(a_1,\dots,a_d)\rangle$ with respect to the standard generators.
\end{lemma}

\begin{proof}
The affine map $f$ sending $v_0$ to $\zero$ and every other $v_i$ to $Ve_i$   has $f(T)=V \cdot\conv(\{\zero,e_1,\dots,e_d\})$ and $f(p)=p':= (a_1,\dots,a_d)$, since $p$ has the same barycentric coordinates in $T$ as $p'$ has in $f(T)$.
Also, $\Lambda':= V\Z^d + (a_1,\dots,a_d)\Z$ is clearly a sublattice of $f(\Lambda)$, where $\Lambda$ is the ambient lattice of~$T$. 
We only need to show that $\Lambda' = f(\Lambda)$ and for this it is enough to check that the normalized volume of $f(T)$ with respect to $\Lambda'$ equals $V$. 

This normalized volume is the order of the quotient $\Lambda'/V\Z^d$, and this quotient is a cyclic group generated by $p' + V\Z^d$. Thus, the normalized volume is the smallest $k\in \N$ such that $k (a_1,\dots,a_d) \in V\Z^d$. We have $k=V$ since
$V=a_0+\dots+a_d$ gives $\gcd(V, a_1, \dots, a_d) = \gcd(a_0, a_1, \dots, a_d) = 1$.
\end{proof}

\cref{lemma:cyclic}  implies that a lattice simplex is determined, modulo unimodular equivalence, by the \emph{volume vector} $(a_0, \dots, a_d)$ of any lattice point $p$ in it, as long as $\gcd(a_0, \dots, a_d)=1$. Since $\Lambda/ V\Z^d$ is then cyclic with generator $p+V\Z^d$, we call $T$ the \emph{cyclic simplex generated by $(a_0, \dots, a_d)$}. We denote  it by $T(a_0, \dots, a_d)$ and denote
$G(V; a_1,\dots,a_d)$ the digraph in the statement.

In what follows we are interested in cyclic \emph{tetrahedra} $T(a,b,c,d)$.
When $a=1$,
%For convenience of notation, we denote the determining parameters by $a$, $b$, $c$, $d$, and correspondingly write $T(a,b,c,d)$ for the tetrahedron at hand. 
%We normally assume $a\le b\le c\le d$ and use $d$ as the $a_0$ that is forgotten from $(a_0,\dots,a_d)$ in the description of the graph in \cref{lemma:cyclic}.
%
\cref{lem:graph_method,lemma:cyclic} are particularly easy to apply, 
%\paco{my edits finish here (except for minor things}
since then $G(V;a,b,c)$ coincides with the Cayley digraph of $(\Z_V)^2$ with respect to the generators $(1,0)$, $(0,1)$ and $(-b,-c)$. That is, $G(V; a,b,c)$ has $(\Z_V)^2$ as vertex set and from each vertex $(i,j)$ we have the following three arcs:
\[
(i,j) \to (i,j+1), \quad
(i,j) \to (i+1,j), \quad
(i,j) \to (i-b,j-c).
\]

\begin{example}
\label{exm:T1112}
\cref{fig:grids} shows the computation of the covering radii of $T(1,1,1,2)$ ($V=5$) and $T(1,1,2,3)$ ($V=7$):
a grid with $V^2$ cells represents the nodes of $G(5;1,1,1)$ and $G(7;1,1,2)$ with the origin at the south-west corner. 
The grid has to be regarded as a torus, so that every cell has an east, west, north and south neighbor. 
A step north or east  increases the distance from the origin by one, unless the cell that we move to can be reached by a shorter path. When this happens, the corresponding arc of $G(V; a,b,c)$ is not used in any shortest path from the origin, and we highlight n bold the corresponding wall between cells. 
Using this idea, one can compute the distance from the origin to each cell in a breadth-first search manner.

Observe that, by commutativity, we only need to consider paths that first use edges with step $(-b,-c)$ and then east or north steps. Thus, in order to verify that the distances we have put are correct in the whole diagram, 
only the distances along the path with steps $(-b,-c)$ starting at the origin need to be checked. 
The cells along that path have their distances also in bold, and they coincide with the cells with bold south and west walls. 
The path finishes when it arrives in a cell that can be more shortly reached from the origin by only east and north steps.
% For example, the cell $(V-b, V-c)$ is the one labeled $\bf 1$ and with south and west edges in bold. 

\begin{figure}[htb]
\begin{tabular}{cc}
\includegraphics[scale=0.5]{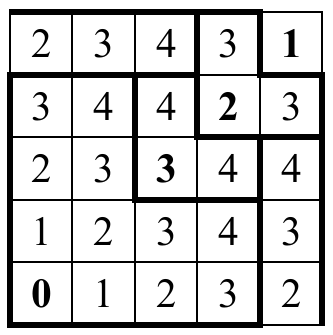} \quad&\quad
\includegraphics[scale=0.45]{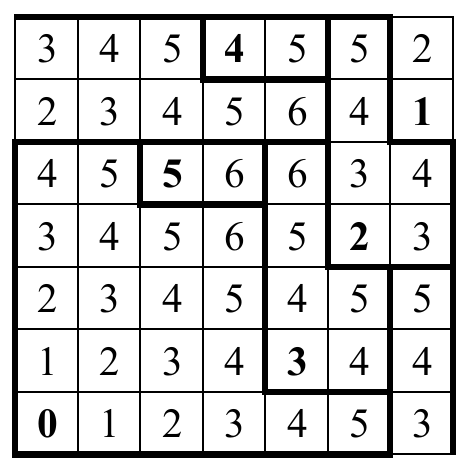}\\
$G(5;1,1,1)$ & $G(7;1,1,2)$ 
\end{tabular}
\caption{Graphical computation of $\delta(G(5;1,1,1))=4$ and $\delta(G(7;1,1,2))=6$, implying $\mu(T(1,1,1,2))=\frac75$ and $\mu(T(1,1,2,3))=\frac97$. 
}
\label{fig:grids}
\end{figure}

%Since these two examples are needed in the next section, we record the outcome of this computation:

\begin{corollary}
\label{coro:T1112}
$\mu(T(1,1,1,2))=\frac75$ \ and \ $\mu(T(1,1,2,3))=\frac97$.
\end{corollary}
\end{example}

This method can also be applied to the tetrahedra $M_k(1,1)$ of \cref{sec:smalldim}:

\begin{lemma}
\label{lem:Mk11}
For every $k\in \N$ we have
\[
\mu(M_k(1,1)) = 1+ \frac{1}{2k}.
\]
\end{lemma}

\begin{proof}
$M_k(1,1)$ has normalized volume $4k$ and the point $p=(0,0,1)$ has barycentric 
coordinates $\frac1{4k}(1,1,2k-1,2k-1)$. Thus, $M_k(1,1) \cong T(1,1,2k-1,2k-1)$. 
\cref{fig:Gk11} shows that $\delta(G(4k;1,2k-1,2k-1))= 4k-1$, from which \cref{lem:graph_method,lemma:cyclic} give
$\mu(M_k(1,1)) = (4k+2)/4k$. 
\end{proof}

\begin{figure}[ht]
\includegraphics[scale=0.7]{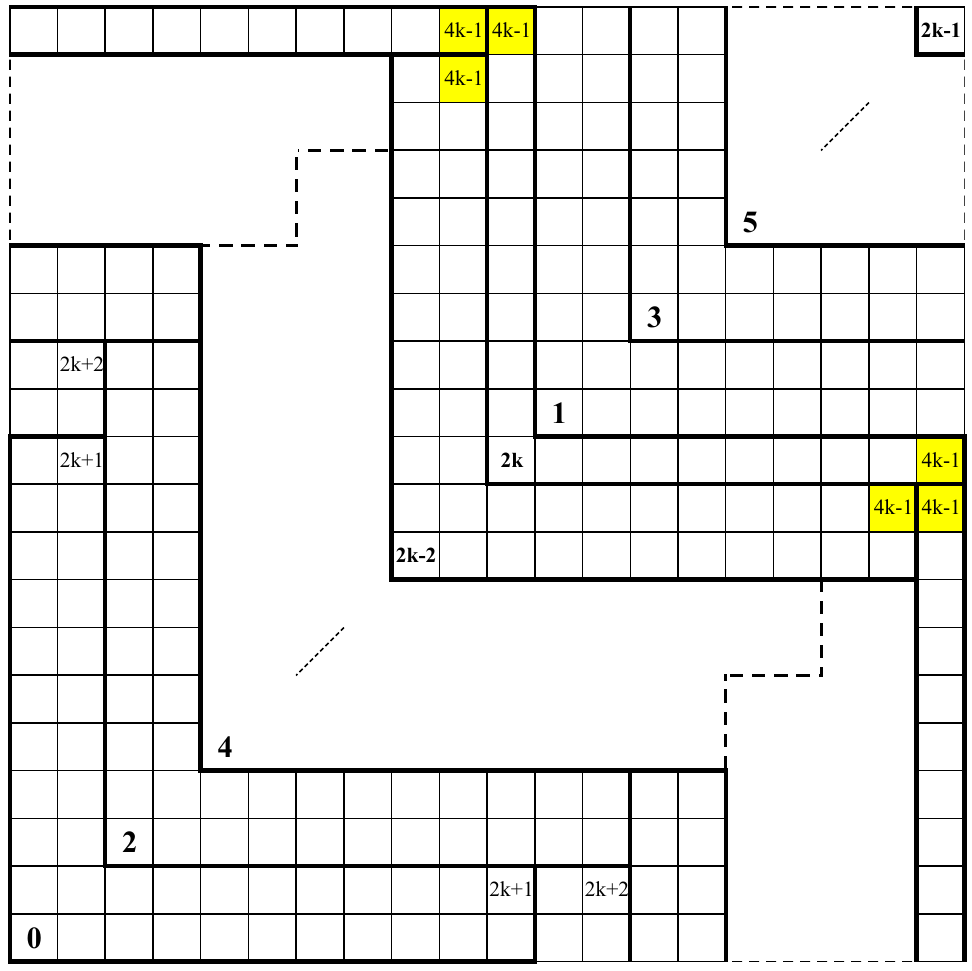}
\caption{Computation of $\delta(G(4k;1,2k-1,2k-1))$, implying $\mu(M_k(1,1))=1+1/2k$. Only the distance to some cells is shown. Cells achieving the diameter are highlighted.}
\label{fig:Gk11}
\end{figure}

Also, we compute the covering radius of a triangle needed in \cref{sect:app-details-section-6}.

\begin{lemma}
\label{lem:2triangles}
\[
%\mu(\conv\{(0,0),(3,1),(1,3)\}) = \frac58, \quad
\mu(\conv(\{(0,1),(3,0),(1,3)\})) = \frac57.
\]
\end{lemma}

\begin{proof}
%The first triangle has area $8$ and the point $(2,1)$ has barycentric coordinates $\frac18(1,2,5)$. 
The triangle has normalized area $7$ and the point $(1,1)$ has barycentric coordinates $\frac17(1,2,4)$. 
Thus,
\begin{align*}
%\mu(\conv\{(0,0),(3,1),(1,3)\}) &= \frac{\delta(G(8;1,2))+2}8 = \frac{3+2}8, \ \text{ and}\\
\mu(\conv(\{(0,1),(3,0),(1,3)\})) &= \frac{\delta(G(7;1,2))+2}7 = \frac{3+2}7.
\end{align*}
The (easy) computation $\delta(G(7;1,2)) = 3$ is left to the reader.
\end{proof}

\section{The 26 minimal non-hollow lattice 3-polytopes; Proof of \texorpdfstring{\Cref{thm:26minimal}}{Theorem \ref{thm:26minimal}}}
\label{sec:26minimal}

The $26$ minimal non-hollow lattice $3$-polytopes with a single interior lattice point were classified by Kasprzyk~\cite{kasprzyk2010threefolds}. We list them 
in Tables~\ref{tbl:minimal_tetrahedra} and~\ref{tbl:minimal_non-tetrahedra}, in the same order as they appear in Kasprzyk's Tables~2 and 4. \Cref{tbl:minimal_tetrahedra} contains the 16 that are tetrahedra and \Cref{tbl:minimal_non-tetrahedra} the 10 that are not.

Vertex coordinates are given as the columns of a matrix, and chosen so that the unique interior point is the origin.
For the tetrahedral examples in \Cref{tbl:minimal_tetrahedra} we include the \emph{volume vector} $(a,b,c,d)$, consisting of the normalized volumes of the pyramids from the origin over the facets. When the volume vector is primitive, our tetrahedron equals the cyclic tetrahedron $T(a,b,c,d)$ of \cref{lemma:cyclic}.
In some cases, an additional ``description'' of the example is given, 
which helps us later to bound its covering radius. For example, via this description we can identify the nine polytopes from \Cref{lemma:minimum_3d} that have covering radius $3/2$.

The exact covering radius, computed with the algorithm of \cref{sec:26computations}
using the SCIP solver in exact solving mode~\cite{exact-scip}, is also shown. 
All except those of \Cref{lemma:minimum_3d} have $\mu<3/2$, which provides a computer proof of \Cref{thm:26minimal}. In the rest of this section we include a computer-free proof.

\subsection{The 16 tetrahedra}

\begin{table}[ht]
\begin{adjustwidth}{-.4cm}{-.4cm}
\centering
\begin{tabular}{|c|c|c|c|}
\hline
\footnotesize $\left(\begin{array}{cccc}
-1 & 1 & 0 & 0 \\
-1 & 0 & 1 & 0 \\
-1 & 0 & 0 & 1
\end{array}\right)$ & 
\footnotesize $\left(\begin{array}{cccc}
-2 & 2 & 0 & 0 \\
-2 & 1 & 1 & 0 \\
-1 & 0 & 0 & 1
\end{array}\right)$ & 
\footnotesize $\left(\begin{array}{cccc}
-5 & 5 & 0 & 0 \\
-3 & 2 & 1 & 0 \\
-2 & 1 & 0 & 1
\end{array}\right)$ &
\footnotesize $\left(\begin{array}{cccc}
-1 & 1 & 0 & 0 \\
-1 & 0 & 1 & 0 \\
-2 & 0 & 0 & 1
\end{array}\right)$  \\ 
%volume vectors:
$(1,1,1,1)$ & $(2,2,2,2)$ & $(5,5,5,5)$ & $(1,1,1,2)$ \\ 
\hline
%old: $S(\one_4)=T(1,1,1,1)$ & $(I \oplus I')'\oplus I$ & $T(5,5,5,5)$ & $T(1,1,1,2)$ \\
$S(\one_4)$ & $(I \oplus I')'\oplus I$ & & \\
%\hline
\boldmath{$\mu = 3/2$} & \boldmath{$\mu = 3/2$}       & $\mu=9/10$ & $\mu=7/5$       \\ \hline \hline
\footnotesize $\left(\begin{array}{cccc}
-1 & 1 & 0 & 0 \\
-1 & 0 & 1 & 0 \\
-3 & 0 & 0 & 1
\end{array}\right)$ & 
\footnotesize $\left(\begin{array}{cccc}
-1 & 1 & 0 & 0 \\
-2 & 0 & 1 & 0 \\
-2 & 0 & 0 & 1
\end{array}\right)$ & 
\footnotesize $\left(\begin{array}{cccc}
-1 & 1 & 0 & 0 \\
-2 & 0 & 1 & 0 \\
-3 & 0 & 0 & 1
\end{array}\right)$ &
\footnotesize $\left(\begin{array}{cccc}
-1 & 1 & 0 & 0 \\
-2 & 0 & 1 & 0 \\
-4 & 0 & 0 & 1
\end{array}\right)$  \\ 
%volume vectors:
$(1,1,1,3)$ & $(1,1,2,2)$ & $(1,1,2,3)$ & $(1,1,2,4)$ \\ 
\hline
% old: $S(\one_3) \oplus I'$ & $S(\one_3)' \oplus I$ & $T(1,1,2,3)$ & $(I \oplus I')^\circ \oplus I'$ \\
$S(\one_3) \oplus I'$ & $S(\one_3)' \oplus I$ &  & $(I \oplus I')^\circ \oplus I'$ \\
%\hline
\boldmath{$\mu = 3/2$} & \boldmath{$\mu = 3/2$} & $\mu=9/7$ & \boldmath{$\mu = 3/2$} \\ \hline \hline
\footnotesize $\left(\begin{array}{cccc}
-1 & 1 & 0 & 0 \\
-3 & 0 & 1 & 0 \\
-4 & 0 & 0 & 1
\end{array}\right)$ & 
\footnotesize $\left(\begin{array}{cccc}
-1 & 1 & 0 & 0 \\
-3 & 0 & 1 & 0 \\
-5 & 0 & 0 & 1
\end{array}\right)$ & 
\footnotesize $\left(\begin{array}{cccc}
-1 & 1 & 0 & 0 \\
-4 & 0 & 1 & 0 \\
-6 & 0 & 0 & 1
\end{array}\right)$ &
\footnotesize $\left(\begin{array}{cccc}
-2 & 1 & 0 & 0 \\
-3 & 0 & 1 & 0 \\
-5 & 0 & 0 & 1
\end{array}\right)$  \\ 
%volume vectors:
$(1,1,3,4)$ & $(1,1,3,5)$ & $(1,1,4,6)$ & $(1,2,3,5)$ \\ 
\hline
%old: $\text{Pyr}_3(S(\one_3))$ & $I'\oplus M_2(1)$ & $I'\oplus M_2(0)$ & $T(1,2,3,5)$ \\
%\hline
$\mu=11/9$ & $\mu=13/10$ & $\mu=4/3$   & $\mu=12/11$ \\ \hline \hline
\footnotesize $\left(\begin{array}{cccc}
-3 & 1 & 0 & 0 \\
-4 & 0 & 1 & 0 \\
-5 & 0 & 0 & 1
\end{array}\right)$ & 
\footnotesize $\left(\begin{array}{cccc}
-1 & 1 & 0 & 0 \\
-3 & 0 & 2 & 0 \\
-4 & 0 & 1 & 1
\end{array}\right)$ & 
\footnotesize $\left(\begin{array}{cccc}
-3 & 2 & 0 & 0 \\
-4 & 1 & 1 & 0 \\
-5 & 1 & 0 & 1
\end{array}\right)$ &
\footnotesize $\left(\begin{array}{cccc}
-4 & 3 & 0 & 0 \\
-3 & 1 & 1 & 0 \\
-5 & 2 & 0 & 1
\end{array}\right)$  \\
%volume vectors:
$(1,3,4,5)$ & $(2,2,3,5)$ & $(2,3,5,7) $ & $(3,4,5,7)$ \\ 
%\hline
%old: $T(1,3,4,5)$ & $\text{Pyr}_4(S(\one_3))$ & $T(2,3,5,7) $ & $T(3,4,5,7)$ \\
\hline
& $\text{Pyr}_4(S(\one_3))$ & &\\
%\hline
$\mu=14/13$ & $\mu=7/6$       & $\mu=1$   & $\mu=18/19$             \\ \hline
\end{tabular}
\caption{The sixteen minimal non-hollow tetrahedra with exactly one interior lattice point, with their covering radii.}
\label{tbl:minimal_tetrahedra}
\end{adjustwidth}
\end{table}

%\subsection{First proof of \texorpdfstring{\Cref{thm:26minimal}}{Theorem \ref{thm:26minimal}}: theoretical bounds}
%\label{sec:26bounds}
%In what follows, we use the convention $\Delta_v := \conv(\{-v,e_1,e_2\})$, $v \in \R^2$, and rely on the following bound on the covering radius of such triangles:

%\paco{moved def of $A,B,C,D$ later, and moved def of $\Delta_v$ to the statement of \cref{prop:conj:v_plane} }
For most of the tetrahedra in \cref{tbl:minimal_tetrahedra} we are going to bound the covering radius based solely on the volume vector of its interior point. We need the following auxiliary result about the covering radius of some (perhaps non-lattice) triangles.

\begin{proposition}
\label{prop:conj:v_plane} 
For each $v \in  \R_{\ge1}^2$, let $\Delta_v := \conv(\{-v,e_1,e_2\})$.
We have
$
\mu(\Delta_v) \le 1.
$
Equality holds if and only if $v \in \{(a,1),(1,a)\}$ with $1 \leq a \leq 2$.
\end{proposition}

\begin{proof}
Due to symmetry, we can assume $v = (v_1,v_2)$ with $v_1 \geq v_2$.
If $v_2 > 1$, then $\Delta_v$ strictly contains the triangle $\Delta_w$, for some $w = (w_1,1) \in \R^2_{\geq 1}$.
By \cref{lemma:simplex_monotone} triangles are tight for every lattice, so that $\mu(\Delta_v) < \mu(\Delta_w)$ and it thus suffices to consider $v = (a,1)$, for $a \geq 1$.

Let $F_0$ be the edge of $\Delta_v$ not containing~$v$, and let $F_1$ and $F_2$ be the edges of~$\Delta_v$ not containing $e_1$ and $e_2$, respectively.
Further, let $\ell = \{(x,y) : x+y=1\}$ be the line containing~$F_0$.
An elementary calculation gives:
\begin{align*}
\ell \cap (F_1 + e_1) &= \left\{\left(\tfrac{2}{a+2},\tfrac{a}{a+2}\right)\right\}, & \ell \cap (F_2 + e_2) &= \left\{\left(\tfrac{1}{a+2},\tfrac{a+1}{a+2}\right)\right\}, \\
\ell \cap (F_1 + (1,1)) &= \left\{\left(\tfrac{2-a}{a+2},\tfrac{2a}{a+2}\right)\right\}, & \ell \cap (F_2 + (1,1)) &= \left\{\left(\tfrac{2}{a+2},\tfrac{a}{a+2}\right)\right\}.
\end{align*}
This already shows that the translates $\{0,1\}^2 + \Delta_v$ cover the unit cube $[0,1]^2$, for every $a \geq 1$, so that $\mu(\Delta_v) \leq 1$ as claimed.

In order to decide the equality case, observe that in the covering of $[0,1]^2$ by these four translates, the point $\left(\tfrac{2}{a+2},\tfrac{a}{a+2}\right)$ is covered last, and is not contained in the interior of any of the four triangles.
However, the translate $(2,1) + \Delta_v$ may contain this point in the interior.
Noting that
\[
\ell \cap (F_1 + (2,1)) = \left\{\left(\tfrac{4-a}{a+2},\tfrac{2a-2}{a+2}\right)\right\},
\]
this happens if and only if $4-a < 2$, that is, $a > 2$.
\end{proof}

\begin{remark}
%\paco{removed reference to arxiv. We do not really need it here}
Every non-hollow lattice triangle is isomorphic to some~$\Delta_v$ considered with respect to a superlattice of $\Z^2$. 
(Let $(a,b,c)$ be the volume vector of an interior point, with $a\le b \le c$, and take $v=(b/a, c/a)$).
With this, 
\cref{prop:conj:v_plane} provides another proof of \cref{conj:one} in the plane. This approach fails in higher dimensions since, for example, we have computed that the tetrahedron $\Delta_{(3/2,1,1)}$ has covering radius $14/9 >3/2$.
\end{remark}

Let $T$ be a lattice tetrahedron with the origin $\zero$ in its interior and let $(a,b,c,d)\in \N^4$ be its volume vector, written with $a\le b \le c \le d$. Let $A$, $B$, $C$, $D$ be the vertices of 
%\paco{I prefer not to use $T(a,b,c,d)$ unless $(a,b,c,d)$ is primitive, since then $(a,b,c,d)$ does not characterize the simplex}
$T$ labeled in the natural way (so that $a$ is the determinant of $BCD$, etc.).
%We give two criteria that guarantee $\mu(T)< 3/2$ and which deal with most of the cases from \cref{tbl:minimal_tetrahedra}.
%\paco{removed sentence about the origin being the \emph{only} interior point. It is not true in general}
%The origin is the unique interior lattice point of these tetrahedra and is denoted by~$O$.

\begin{lemma}
\label{lemma:16tetra}
%\paco{new lemma, although it uses arguments that were already there}
With this notation, suppose that the triangle $OCD$ is unimodular.
%(that is, it contains no lattice point other than $C$, $D$ and the origin).
If either of the conditions i) or ii) below holds, then $\mu(T) < 3/2$.
%
%Moreover, $\mu(T) = 3/2$ can only happen if $(a,b,c)=(1,1,2)$ in case (i) or  $(a,b,c,d)=(2,2,2,2)$ in case (ii).
\begin{enumerate}[i)]
\item $a+b\le c$ and  $(a,b,c) \ne (1,1,2)$.
\item  $a+b\ge4$, \ $3c\ge a+b+d$, and $(a,b,c,d) \ne (2,2,2,2)$.
\end{enumerate}
\end{lemma}

\begin{proof}
Since the triangle $OCD$ is unimodular, there is no loss of generality in taking~$C=(1,0,0)$ and~$D=(0,1,0)$. Once this is done,~$A$ and~$B$ must have $z$ coordinate equal to $b$ and $-a$, in order for the determinants of $BCD$ and $ACD$ to be $a$ and $b$, respectively. In order for the determinants of $ABD$ and $ABC$ to be $c$ and $d$, the segment $AB$ must intersect the plane $z=0$ at the point $(-c/(a+b), -d/(a+b))$. That is, $T\cap \{z=0\}$ is the triangle $\Delta_{(c/(a+b), d/(a+b))}$ of \Cref{prop:conj:v_plane}.
Then, \Cref{lem:projection} applied to projecting along the $z$ coordinate gives:
\[
\mu(T) \le \mu\left(\Delta_{\left(\frac{c}{a+b}, \frac{d}{a+b}\right)}\right) + \frac1{a+b}.
\]
We now consider the two cases in the statement separately:
For part i),  $c \ge a+b$ implies that  $\frac{d}{a+b} \ge \frac{c}{a+b} \ge 1$.
\Cref{prop:conj:v_plane} says that the first summand is $\le 1$, with equality possible only if $c=a+b$. Thus:
\[
\mu\left(\Delta_{\left(\frac{c}{a+b}, \frac{d}{a+b}\right)}\right) + \frac1{a+b} \le 1+\frac12 = \frac32,
\]
with equality only if $c=a+b$ and $a=b=1$.
%
%\paco{I BELIEVE THE COMPUTATION FOR $T(2,3,5,7)$ IS WRONG}
%\matthias{Corrected values in \cref{tbl:minimal_tetrahedra} for $T(5,5,5,5)$, $T(2,3,5,7)$, $T(3,4,5,7)$; and checked and confirmed the others that we cannot compute by hand.}
%

For part ii), $3c\ge a+b+d$ implies that the point $(-1/2,-1/2)$ is in 
$\Delta_{\left(\frac{c}{a+b}, \frac{d}{a+b}\right)}$, so $\Delta_{\left(\frac{c}{a+b}, \frac{d}{a+b}\right)}$ contains the triangle $\Delta_{(1/2,1/2)} = S(1,1,1/2)$. Its covering radius is $5/4$ by \Cref{thm:lambda}.
Hence,
\[
\mu(T) \leq \mu\left(\Delta_{\left(\frac{c}{a+b}, \frac{d}{a+b}\right)}\right) + \frac1{a+b} \le 
 \mu\left(\Delta_{\left(\frac12, \frac12\right)}\right) + \frac1{a+b} \le \frac54+\frac14 = \frac32.
\]
The third inequality is strict unless $a+b=4$. Because simplices are tight, the second inequality is strict unless 
$\Delta_{\left(\frac{c}{a+b}, \frac{d}{a+b}\right)} = \Delta_{(1/2,1/2)}$, that is, unless $c=d=(a+b)/2$, which implies $a=b=c=d$.
\qedhere
\end{proof}

\noindent With this we can prove that all the tetrahedra in \Cref{tbl:minimal_tetrahedra}, except for the five from \cref{lemma:minimum_3d}, have $\mu<3/2$:
%In the following we denote each tetrahedron by its volume vector:

\begin{itemize}
\item $(1,1,1,2)$, $(1,1,2,3)$ are the ones whose $\mu$ we computed in \cref{coro:T1112}. 

\item (5,5,5,5) is in the conditions of part (ii) of \cref{lemma:16tetra}. The hypothesis that $OCD$ is unimodular is trivial since $C=(0,1,0)$ and $D=(0,0,1)$.
%With this we finish the first two rows in the table.

\item The four volume vectors in the third row satisfy the conditions $a+b\le c$ and $(a,b,c)\ne (1,1,2)$ of part (i) of \cref{lemma:16tetra}. That $OCD$ is unimodular for them follows from $\gcd(a,b)=1$, because the normalized volume of $OCD$ divides those of $OBCD$ and $OACD$, which equal $a$ and $b$, respectively.

\item The four in row four satisfy the conditions $a+b\ge 4$ and $(a,b,c,d)\ne (2,2,2,2)$. The condition that $OCD$ is unimodular follows again from $\gcd(a,b)=1$, except for the tetrahedron (2,2,3,5).

\item The remaining tetrahedron (2,2,3,5) is marked  ``$\text{Pyr}_4(S(\one_3))$'' because it has a facet isomorphic to $S(\one_3)$ (the facet in the  plane $x+z=2y+1$)
%\paco{plane equation was wrong in previous version}
 and the opposite vertex is at distance four from that facet. 
\Cref{lem:projection} applied to the projection along the base of the pyramid gives
%\paco{The same argument could have been used for (1,1,3,4), labeled $\text{Pyr}_3(S(\one_3))$ because it is a pyramid of height three over (the same) $S(\one_3)$.}
\[
\mu(\text{Pyr}_4(S(\one_3))) \le \mu(S(\one_3)) + \frac14 = \frac54.
\]
\end{itemize}

\begin{remark}
All these tetrahedra except $(2,2,2,2)$ and $(5,5,5,5)$ have $\gcd(a,b,c,d)=1$. Thus, \cref{lem:graph_method,lemma:cyclic}
can be used to compute their exact covering radii,
as we did for $(1,1,1,2)$, $(1,1,2,3)$ in \cref{exm:T1112}. 
The condition $a=1$ used in that computation can be weakened to $\gcd(a,V)=1$, (which these fourteen tetrahedra satisfy) since then $G(V;a,b,c)=G(V;1,ba^{-1}, ca^{-1})$, where $a^{-1}$ is the inverse of $a$ modulo $V$.
\end{remark}

\subsection{The 10 non-tetrahedra}

\begin{table}[ht]
\centering
\begin{tabular}{|c|c|}
\hline
\footnotesize $\left(\begin{array}{rrrrr}
1 & 0 & 0 &  0 & -1 \\
0 & 1 & 0 &  0 & -1 \\
0 & 0 & 1 & -1 &  0
\end{array}\right)$ &
\footnotesize $\left(\begin{array}{rrrrr}
1 & 0 & 0 & -2  & -1 \\
0 & 1 & 0 & -1  &  0 \\
0 & 0 & 1 &  0  & -1
\end{array}\right)$  \\ \hline

$S(\one_3) \oplus I$, \quad \boldmath{$\mu = 3/2$} & $I\oplus Q_4$, \quad $\mu = 4/3$ \\ \hline \hline

%$S(\one_3) \oplus I$ & $I\oplus Q_4$\\ \hline
%\boldmath{$\mu = 3/2$} & $\mu = 4/3$ \\ \hline \hline
                                                                                                                                      \footnotesize $\left(\begin{array}{rrrrr}
1 & 0 & -1 & 1 & -1 \\
0 & 1 & -1 & 2 & -2 \\
0 & 0 &  0 & 3 & -3
\end{array}\right)$ &
\footnotesize $\left(\begin{array}{rrrrr}
1 & 0 & 0 & -2  & -2 \\
0 & 1 & 0 & -1  &  0 \\
0 & 0 & 1 &  0  & -1
\end{array}\right)$  \\ \hline

$\text{Bipyr}_3(S(\one_3) \oplus I)$, \quad $\mu = 17/18$ & $I\oplus I\oplus I'$, \quad \boldmath{$\mu = 3/2$} \\ \hline \hline

%$\text{Bipyr}_3(S(\one_3) \oplus I)$ & $I\oplus I\oplus I'$\\ \hline
%$\mu = 17/18$ & \boldmath{$\mu = 3/2$} \\ \hline \hline
                                                                                                                                      \footnotesize $\left(\begin{array}{rrrrr}
1 & 0 & 0 &  0 & -2 \\
0 & 1 & 0 &  0 & -1 \\
0 & 0 & 1 & -1 &  0
\end{array}\right)$ &
\footnotesize $\left(\begin{array}{rrrrr}
1 & 0 & -2 & 1  & -3 \\
0 & 1 & -1 & 1  & -1 \\
0 & 0 &  0 & 2  & -2
\end{array}\right)$  \\ \hline

$(I \oplus I')^\circ \oplus I$, \quad \boldmath{$\mu = 3/2$} & $\text{Bipyr}_2(I\oplus I\oplus I')$, \quad $\mu=7/8$ \\ \hline \hline

%$(I \oplus I')^\circ \oplus I$ & $\text{Bipyr}_2(I\oplus I\oplus I')$\\ \hline
%\boldmath{$\mu = 3/2$} & $\mu=7/8$ \\ \hline \hline
                                                                                                                                      \footnotesize $\left(\begin{array}{rrrrr}
1 & 0 & -2 & 1  & -1 \\
0 & 1 & -1 & 1  & -1 \\
0 & 0 &  0 & 2  & -2
\end{array}\right)$ &
\footnotesize $\left(\begin{array}{rrrrrr}
1 & 0 & 0 & -1 &  0 &  0 \\
0 & 1 & 0 &  0 & -1 &  0\\
0 & 0 & 1 &  0 &  0 & -1
\end{array}\right)$  \\ \hline

$\text{Bipyr}_2((I \oplus I')^\circ \oplus I)$, \quad $\mu = 1$ & $I\oplus I\oplus I$, \quad \boldmath{$\mu = 3/2$} \\ \hline \hline

%$\text{Bipyr}_2((I \oplus I')^\circ \oplus I)$ & $I\oplus I\oplus I$\\ \hline
%$\mu = 1$ & \boldmath{$\mu = 3/2$} \\ \hline \hline

\footnotesize $\left(\begin{array}{rrrrr}
1 & 0 & 0 & -1 &  1 \\
0 & 1 & 0 & -1 &  1 \\
0 & 0 & 1 &  0 & -1
\end{array}\right)$ &
\footnotesize $\left(\begin{array}{rrrrrr}
1 & 0 & -1 &  0 & 1 & -1 \\
0 & 1 &  0 & -1 & 1 & -1 \\
0 & 0 &  0 &  0 & 2 & -2
\end{array}\right)$  \\  \hline

$ \text{Pyr}_3([0,1]^2)$, \quad $\mu = 4/3$ & $\text{Bipyr}_2(I\oplus I\oplus I)$, \quad $\mu = 3/4$ \\ \hline

%$ \text{Pyr}_3([0,1]^2)$ & $\text{Bipyr}_2(I\oplus I\oplus I)$\\ \hline
%$\mu = 4/3$ & $\mu = 3/4$ \\ \hline
\end{tabular}
\caption{The ten minimal non-hollow non-tetrahedra with exactly one interior lattice point, with their covering radii.}
\label{tbl:minimal_non-tetrahedra}
\end{table}

For the ten polytopes in \cref{tbl:minimal_non-tetrahedra} we use the following direct arguments:

\begin{itemize}
\item Four of them are the non-tetrahedra in \Cref{lemma:minimum_3d}, of covering radius $3/2$.

\item There are another four that are affinely equivalent to the previous four, 
except considered with respect to a finer lattice. 
They are marked as $\text{Bipyr}_i($--$)$, where $i$ is the index of the superlattice, since they are also (skew) bipyramids over their intersection with the plane $z=0$.
This intersection is, in the four cases, one of the three non-hollow lattice polygons with $\mu=1$.  \Cref{lem:projection} for the projection $\pi$ onto the $z$-coordinate gives:
\[
\mu(P) \le \mu(P\cap\{z=0\}) + \mu(\pi(P)) \le 1 + \frac14 =\frac54,
\]
since $\pi(P)$ has length at least four in all four cases.

\item The one marked $\text{Pyr}_3([0,1]^2)$ is a  pyramid with base a unimodular parallelogram in the plane $x+y+z=1$ and apex at distance three. \Cref{lem:projection} applied to the projection along the base gives
\[
\mu(\text{Pyr}_3([0,1]^2)) \le 1 + \frac13.
\]

\item The remaining one is marked $I\oplus Q_4$ because it decomposes as
\[
\left(\begin{array}{rrrrr}
0 & -2  \\
1 & -1  \\
0 &  0 
\end{array}\right)
\oplus
\left(\begin{array}{rrrrr}
1 &-1 & 0  & -1 \\
0 & 0 & 0  & 0 \\
0 & 0 & 1  & -1
\end{array}\right),
\]
where the first summand is equivalent to $I=[-1,1]$ and the second is a quadrilateral $Q_4$.
Since $Q_4$ strictly contains a translation of $S(\one_3)$ and $S(\one_3)$ is tight (\cref{lemma:simplex_monotone}), we have:
\[
\mu(I\oplus Q_4) = \mu(I)+ \mu(Q_4) < \mu(I)+ \mu(S(\one_3)) = \frac{3}{2}.
\]
\end{itemize}

\section{Details for \texorpdfstring{\Cref{thm:dim2-k}}{Theorem \ref{thm:dim2-k}}}
\label{sect:app-details-section-6}

Here, we discuss the missing details for the proof of \cref{thm:dim2-k} in \cref{sec:k_points}.
Note that this is not part of the published version~\cite{published} of the article.

\begin{lemma}
 \label{lem:34}
 Suppose $P$ is a non-hollow lattice-polygon of width $\omega = 3$ (assume it is contained in $[0,3]\times \R$), it has at most $m = 4$ collinear lattice points, and it contains $k \geq 5$ interior lattice points. Then, at least one of the following conditions holds:
 \begin{enumerate}[(i)]
 \item $P$ has four collinear lattice points along one of the intermediate vertical lines $\{1\}\times \R$ or $\{2\}\times \R$ and (at least) three of them are interior to $P$,
 \item $P$ contains a parallelogram with one vertical edge of length two and horizontal width two.
 \end{enumerate}
% In both cases we have
% \[
% \mu(P) < \frac12 + \frac1{k+1}.
% \]
 \end{lemma}
 
 \begin{proof}
% We first prove the conclusion. If $P$ contains a parallelogram $Q$ as stated in (ii) then $\mu(P) \le \mu(Q) =\frac12$ and we are done. Suppose, then, that~$P$ contains the four collinear points $(1,i)$, $i=1,2,3,4$ and that the first three are interior. Then the segment $P\cap \{x=1\}$ has length at least $3+\frac13=\frac{10}3$, because its bottom end-point cannot be above $(1, \frac23)$. Thus
% \[
% \mu(P) \le \frac13 + \frac3{10} = \frac{19}{30} \quad < \quad \frac{9}{14} =\frac12 + \frac17 \le \frac12 +\frac1{k+1},
% \]
% if $k \in \{5,6\}$.
% In the case $k \geq 7$, we can assume the four collinear lattice points $(1,i)$, $i=1,2,3,4$ are actually interior to~$P$.	
% Therefore, the segment $P\cap \{x=1\}$ has length at least $3+\frac23=\frac{11}3$, because its bottom end-point cannot be above $(1, \frac23)$ and its top end-point cannot be below $(1,4+\frac13)$. Thus
% \[
% \mu(P) \le \frac13 + \frac3{11} = \frac{20}{33} \quad < \quad \frac{11}{18} =\frac12 + \frac19 \le \frac12 +\frac1{k+1}.
% \]
% 
 We assume that $P$ is neither in the conditions of (i) or (ii), and we derive a contradiction. By (the negation of) (i), $P$ has at most three interior lattice points along each of the two vertical lines. Since it has at least five in total, we assume without loss of generality that
 \[
 (1,1), (1,2), (1,3), (2,1), (2,2) \in \inter(P).
 \]
 The proof is based on arguing that certain additional points must or cannot be in~$P$. This is illustrated in \cref{fig:34} where the points that \emph{must} be in~$P$ are drawn as black dots and the ones that \emph{cannot} as crosses. The initial points that we assume in $\inter(P)$ are drawn as white dots. The labels of the points indicate the order in which they appear in the proof:
 
 \begin{figure}[htb]
 \includegraphics[scale=0.6]{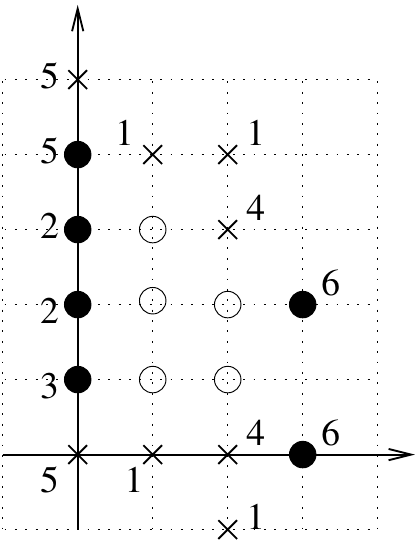}
 \caption{Illustration of the proof of \cref{lem:34}.}
 \label{fig:34}
 \end{figure}
 
 \begin{enumerate}[1)]
 \item None of the points $(1,0)$, $(1,4)$, $(2,-1)$, or $(2,4)$ can be in $P$, since their presence would give condition (i).
 
 \item  The left end-point of the top (respectively, bottom) edge of $P$ meeting the line $\{x=1\}$ must then be of the form $(0,a)$ with $a\ge 3$ (respectively, with $a\le 2$). Hence, $(0,2)$ and $(0,3)$ are in $P$.
 
 \item One of $(0,1)$ or $(0,4)$ must be in $P$, for otherwise the edges going from $(0,2)$ and $(0,3)$ to the right must go strictly below and above $(2,0)$ and $(2,3)$ respectively, giving four interior points along $\{x=2\}$. Assume without loss of generality that $(0,1) \in P$.
 
 \item Since we already have an intersection of length two with $\{x=0\}$, the intersection with $\{x=2\}$ must have length strictly smaller than $2$, in order for~$P$ not to be in the conditions of part (ii). Thus, $(2,0), (2,3)\not\in P$.
 
 \item Now the top edge of $P$ crossing $\{x=1\}$ must have its left end-point above $(0,3)$, because $(2,3) \notin P$, which implies $(0,4) \in P$. Since we already have four collinear points in $\{x=0\}$, neither $(0,0)$ nor $(0,5)$ is in $P$.
 
 \item Now the only possibility for the right end-points of the top and bottom edges of $P$ are $(3,0)$ and $(3,2)$ (remember that the white dots in the figure are meant to be in the interior of $P$).
 \end{enumerate}
 This gives a contradiction, since $P$ is now as described in part (ii).
 \end{proof}

 \begin{lemma}
 \label{lem:3collinear}
 Suppose $P$ is a non-hollow lattice-polygon of width at least three and that it contains three interior collinear lattice points. Then, $P$ has four collinear lattice points. 
 %In particular,
 %\[
 %\mu(P) < \frac12 + \frac1{k+1},
 %\]
 %where $k$ is the number of interior lattice points of~$P$.
 \end{lemma}
 
 \begin{proof}
 %The bound on the covering radius of~$P$ follows from having four collinear points, as proved in \cref{lem:bigm} and \cref{lem:34}.
 
 %In order to show that $P$ contains four collinear lattice points, we prove that~$P$ contains at least $10$ lattice points overall if we assume otherwise.
 %This implies that $P$ must contain two lattice points $x,y$ in the same residue class modulo $(3\Z)^2$, so the segment $[x,y] \subset P$ contains four lattice points, giving a contradiction.
 
 \begin{figure}[htb]
 \includegraphics[scale=0.6]{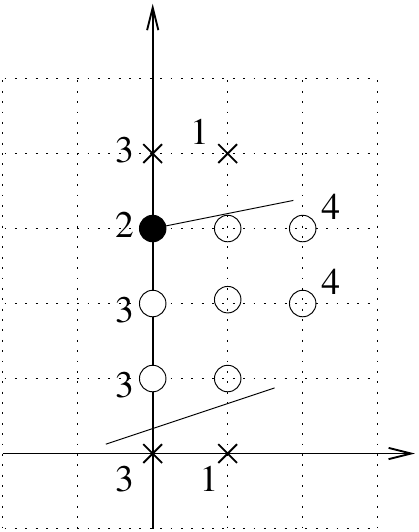}
 \caption{Illustration of the proof of \cref{lem:3collinear}.}
 \label{fig:3collinear}
 \end{figure}
 
 Suppose $P$ contains $(1,1)$, $(1,2)$ and $(1,3)$ in its interior and moreover that $P$ does not contain four collinear lattice points. We are going to arrive in a contradiction. 
 Similarly to the proof of \cref{lem:34}, we illustrate our reasoning in \cref{fig:3collinear}:
 
 \begin{enumerate}[1)]
 \item $(1,0)$ and $(1,4)$ are exterior to $P$, and the length of $P\cap\{x=1\}$ is greater than two.
 
 \item  Since $P$ does not have a vertex in $\{x=1\}$, one of the intersections $P\cap \{x=0\}$ or $P\cap \{x=2\}$ has at least the same length as $P\cap \{x=1\}$. Suppose it is $P\cap \{x=0\}$.
 If $P$ does not have a vertex in $\{x=0\}$ then it has at least three lattice points in $\{x=-1\}$ and at least one in $\{x\geq2\}$. That would make at least ten lattice points in total, which would imply $m\ge 4$ as we observed in \Cref{rem:castryck}.
 So, without loss of generality we assume that the top edge crossing $\{x=1\}$ has a vertex at $(0,3)$. 
 
 \item Then $(0,0)$ and $(0,4)$ are exterior to $P$, in order not to have four collinear points, and $(0,2)$ and $(0,1)$ are interior to $P$, since $P\cap \{x=0\}$ has length larger than two.
 
 \item The edges crossing $\{x=1\}$ must cross $\{x=2\}$ above $(2,3)$ and below $(2,2)$ respectively, so these two points are also in the interior of~$P$.
 \end{enumerate}
 So, we have identified eight lattice points in $P$. But none of them can be an end-point of the bottom edge of $P$ crossing $\{x = 1\}$. Thus, $P$ has at least ten lattice points, which implies $m\ge 4$.
 \end{proof}

 %{\color{blue}
 %\Paco{The following is a shorter alternative to  \cref{lem:34}, but it implies enumerating up to 12 points rather than 9:}
 %
 %\begin{lemma}
 %\label{lem:bigk}
 %If $P$ has width at least three and at least $13$ lattice points then $P$ satisfies \cref{conj:k} with strict inequality.
 %\end{lemma}
 %
 %\begin{proof}
 %Among the $13$ lattice points in $P$ there must be four in the same class modulo $(2\Z)^2$. If three of those are collinear then $P$ contains five collinear lattice points and we can apply \cref{lem:bigm}. So, we can assume that $P$ contains four points from the lattice 
 %$(2\Z)^2$, no three of them collinear. That implies $P$ to contain a copy of either $[0,2]^2$ or $2S(\one_3)$, both of which have covering radius equal to $\frac12$. Thus, $\mu(P) \le \frac12$.
 %\end{proof}
 %}

 \begin{lemma}
 \label{lem:ad-hoc}
 Let $P$ be a non-hollow lattice polygon with $m=3$ and which is contained in one of the three polygons depicted in \cref{fig:m=3}.
 Then, \[\mu(P) < \frac12 + \frac1{k+1}.\]
 \end{lemma}
 
 \begin{proof}
 Let $Q$ be the convex hull of all the interior lattice points in $P$.
 Denoting by $R$ any of the three polygons in \cref{fig:m=3}, we want to show that every subpolygon of $R$ containing all the white dots (the polygon $Q$) in its interior has covering radius strictly smaller than $\frac12+\frac1{k+1}$, where $k=4$ in cases (A) and (B), and $k=3$ in case (C): 
 \begin{enumerate}[(A)]
 
 \item If $Q=S(\one_3)$, then $R=2Q=S(2,2,2)$. The only lattice subpolygon of~$R$ containing $Q$ in its interior is $R$ itself, whose covering radius is $1/2$ by \cref{thm:lambda} (or by the fact that it coincides with $2S(\one_3)$).
 
 \item If $Q$ is a unit parallelogram, without loss of generality we assume that $Q = [1,2]^2$ and $R=[0,3]^2$. We distinguish cases:
 	\begin{enumerate}[1)]
 	\item $P$ contains at least one lattice point from the relative interior of each edge of $R$. The only possibility for $R$ not to contain four collinear points is that $P$ equals $S := \conv(\{(1,0), (3,1), (2,3), (0,2)\})$ (or its mirror reflection). It is easy to calculate that $\mu(S)=\frac35 <\frac12 + \frac15$, since $Q$, which is a fundamental domain, is inscribed in the dilation of $S$ of factor $\frac35$ centered at $(\frac32,\frac32)$.
 	\item Along some edge, $P$ does not contain any relative interior point of~$R$. Say $P$ contains neither $(1,0)$ nor $(2,0)$. Then, it must contain the edge from $(0,0)$ to $(3,1)$ (or its mirror reflection, which gives an analogous case). If $P$ contains $(0,1)$ 
 then we have four collinear points.  If it does not, then it contains the edge from $(0,0)$ to $(1,3)$. In particular,~$P$ contains the triangle with vertices $(0,0)$, $(3,1)$ and $(1,3)$. This triangle is a translate of $S(1,2,2)$, hence its covering radius equals $\frac58 < \frac12+\frac15$ by \cref{thm:lambda}.
 %		\begin{enumerate}[i)]
 %		\item \label{itm:I1i} $P$ contains the triangle $\conv((0,0), (1, 3), (3,1))$; this triangle has covering radius $5/8$, which is strictly smaller than $\frac12 +\frac15$ (see \cref{lem:2triangles} for the computation).
 %		\item \label{itm:I1ii}$P$ contains a parallelogram of covering radius at most $\frac12$, and thus $\mu(P) \le \frac12$.
 %		\end{enumerate}
 %One of these two scenarios must hold. Suppose we are not in the conditions of \Cref{itm:I1i}, w.l.o.g $(1, 3) \notin P$. Since $(1, 3)$ is the only point seen by $(0,0)$, if these are not in $P$, $P$ must contain $(0,2)$. Now on the top edge, $P$ may contain either $(2,3)$ or $(3,3)$. In the first case, we fall into case \Cref{itm:I1ii}, since $\conv((0,0),(0,2), (2,1), (2,3)) \subseteq P$. If $P$ contains $(3,3)$, there are two more cases: $P$ either contains $(3,1)$, and therefore the parallogram $\conv((0,0),(0,2), (3,1), (3,3))$, or it must contains $(2,0)$ and thus the square $\conv((0,0),(0,2), (2,0), (2,2))$, and in both cases we fall into case \Cref{itm:I1ii}.
 	\end{enumerate}

 \item If $Q$ is a unimodular triangle, then without loss of generality we can assume that $Q = \conv(\{(1,1), (1,2), (2,1)\})$, so that $P$ is contained in $R = \conv(\{(0,0), (4,0), (0,4)\})$. There are two possibilities:
 	\begin{enumerate}[1)]
 	\item $P$ contains a vertex of $R$, say $(0,0)$. It must also contain (at least) one lattice point on the opposite edge $\{x+y=4\}$. But:
 		\begin{enumerate}[i)]
 		\item If $(2,2)\in P$ then $\mu(P) \le \mu([0,2]^2)=\frac12$.
 		\item If $(4,0)$ or $(0,4)$ is in $P$ then $P$ contains five collinear points.
 		\item If $(3,1)$ is in $P$ then $P\cap\{y=1\}$ has length at least $8/3$. \Cref{lem:projection} for the projection along this line gives $\mu(P) \le \frac13 + \frac38= \frac{17}{24} < \frac34$. The case $(1,3)\in P$ is symmetric to this one.
 		\end{enumerate}
 
 	\item $P$ does not contain a vertex of $R$. Then, in order for $(1,1)$ to be in the interior of~$P$, $P$ must contain (at least) one of the points $(1,0)$ or $(0,1)$. The same reasoning for the other two interior points gives that $P$ contains one of $(3,0)$ and $(3,1)$, and one of $(0,3)$ and $(1,3)$. Out of the eight combinations of one point from each pair the only ones that do not produce four collinear points in $P$ are the triangle $\conv(\{(0,1), (3,0), (1,3)\})$ and its reflection along the diagonal $\{x=y\}$. In \cref{lem:2triangles} we compute the covering radius of this triangle to be $5/7$, which is smaller than $3/4$.
 	\end{enumerate}
 \end{enumerate}
 \end{proof}

\subsection*{Acknowledgements}

We thank Ambros Gleixner for useful discussions and guidance with the exact computer calculations presented in \cref{sec:26minimal}, and Gabrielle Balletti for helpful indications regarding \cite[Prop.~4.2]{ballettikasprzyk2016twopoints}.
We also thank the referees for their careful reading and their suggestions to improve the presentation of the material.

% ---------------------------------bibliography-----------------------------
\bibliographystyle{amsplain}
\bibliography{mybib}

\end{document}